\documentclass[reqno]{amsart}
\usepackage{amsmath}
\usepackage{amssymb}
\usepackage{enumerate}
\usepackage[all]{xy}
\usepackage{amsthm}
\usepackage{amscd}
\usepackage{amsfonts}
\usepackage{latexsym}
\usepackage{amscd}
\usepackage{graphicx}
\usepackage{setspace}
\usepackage{bbm}
\usepackage{tikz}

\usetikzlibrary{decorations.pathmorphing}
\usetikzlibrary{cd}

\definecolor{darkblue}{HTML}{111199}
\definecolor{darkgreen}{HTML}{336633}
\definecolor{darkred}{HTML}{993333}
\definecolor{darkpurple}{HTML}{995599}
\makeatletter
\newcommand{\arxiv}[1]{{\tt arXiv:#1}}

\newtheorem{theorem}{Theorem}[section]
\newtheorem{lemma}[theorem]{Lemma}

\newtheorem{corollary}[theorem]{Corollary} 
\theoremstyle{definition}  
\newtheorem{definition}[theorem]{Definition}
\newtheorem{example}[theorem]{Example}
\newtheorem{examples}[theorem]{Example}

\newtheorem{remark}[theorem]{Remark}

\@addtoreset{equation}{section}

\def\ii{\text{\i}}
\def\jj{\text{\j}}
\def\SCat{\mathcal{SC}at}
\def\SMon{\mathcal{SM}on}
\def\SCAT{\mathfrak{SCat}}
\def\TSCAT{\text{2-}\mathfrak{SCat}}
\def\TGSCAT{\text{2-}\mathfrak{GSCat}}
\def\GSCat{\mathcal{GSC}at}
\def\GSCAT{\mathfrak{GSCat}}
\def\piSCat{\Pi\text{-}\mathcal{SC}at}
\def\piSMon{\Pi\text{-}\mathcal{SM}on}
\def\piSCAT{\Pi\text{-}\mathfrak{SCat}}

\def\qpiGSCAT{(Q,\Pi)\text{-}\mathfrak{GSCat}}
\def\qpiTGSCAT{(Q,\Pi)\text{-2-}\mathfrak{GSCat}}

\def\qpiCAT{(Q,\Pi)\text{-}\mathfrak{Cat}}
\def\piTSCat{\Pi\text{-2-}\mathcal{SC}at}
\def\piTSCAT{\Pi\text{-2-}\mathfrak{SCat}}

\def\piCat{\Pi\text{-}\mathcal{C}at}
\def\piMon{\Pi\text{-}\mathcal{M}on}
\def\piCAT{\Pi\text{-}\mathfrak{Cat}}
\def\piTCat{\Pi\text{-2-}\mathcal{C}at}
\def\piTCAT{\Pi\text{-2-}\mathfrak{Cat}}
\def\SVec{\underline{\mathcal{SV}ec}}
\def\fdSVec{\underline{\mathcal{SV}ec}_{fd}}
\def\SVEC{\mathcal{SV}ec}
\def\fdSVEC{\mathcal{SV}ec_{fd}}
\def\GSVec{\mathcal{GSV}ec}
\def\deg{\operatorname{deg}}
\def\ob{\operatorname{ob}}
\def\sop{\operatorname{sop}}

\newcommand{\rSMod}{\mathcal{S\!M}od\text{-}}
\newcommand{\lrSMod}{\text{-}\mathcal{S\!M}od\text{-}}

\newcommand{\lrGSMod}{\text{-}\mathcal{GS\!M}od\text{-}}
\newcommand{\End}{\operatorname{End}}

\newcommand{\Hom}{\operatorname{Hom}}

\newcommand{\Kar}{\operatorname{Kar}}
\newcommand{\KAR}{\operatorname{Kar}}
\newcommand{\SKar}{\operatorname{SKar}}
\newcommand{\GSKar}{\operatorname{GSKar}}
\newcommand{\GSKAR}{\operatorname{GSKar}}

\newcommand{\id}{\text{id}}

\newcommand{\g}{\mathfrak{g}}
\newcommand{\Z}{\mathbb{Z}}
\newcommand{\N}{\mathbb{N}}
\newcommand{\Q}{\mathbb{Q}}
\renewcommand{\k}{\Bbbk}

\newcommand{\eps}{\varepsilon}

\def\la{\lambda}

\def\C{\mathcal{C}}

\def\A{\mathcal{A}}
\def\AA{\mathfrak{A}}
\def\B{\mathcal{B}}
\def\BB{\mathfrak{B}}

\def\a{{\mathtt{a}}}

\def\0{{\bar{0}}}
\def\1{{\bar{1}}}
\def\unit{\mathbbm{1}}
\def\smallcat{}

\allowdisplaybreaks

\begin{document}

\title{Monoidal supercategories}
\author[J. Brundan]{Jonathan Brundan}
\author[A. Ellis]{Alexander P. Ellis}
\address{Department of Mathematics,
University of Oregon, Eugene, OR 97403, USA}
\email{brundan@uoregon.edu, apellis@gmail.com}
\thanks{2010 {\it Mathematics Subject Classification}: 17B10, 18D10.}
\thanks{Research of J.B. supported in part by NSF grant DMS-1161094.}

\iffalse
\begin{abstract}
In this article, we develop some categorical foundations of superalgebra ($=$
$\Z/2$-graded algebra). In particular, we discuss
{\em monoidal supercategories} and their superadditive envelopes.
More generally, we set up some basic notions of {\em 2-supercategories}.
In fact, there are already several competing definitions in the
literature, and one of our goals is to clarify the connections between them.
\end{abstract}
\fi

\begin{abstract}
This work is a companion to our article ``Super Kac-Moody
2-categories,''  which introduces super analogs of the
Kac-Moody 2-categories of Khovanov-Lauda and Rouquier. In the case of
$\mathfrak{sl}_2$, the super Kac-Moody 2-category was
constructed already in [A. Ellis and A. Lauda, ``An odd categorification
of $U_q(\mathfrak{sl}_2)$''], 
but we found that the
formalism adopted there became too cumbersome in the general
case. Instead, it is better to work with {\em
  2-supercategories} (roughly, 2-categories enriched in
vector superspaces). Then the Ellis-Lauda
2-category, 
which we call here a {\em $\Pi$-2-category} (roughly, a 2-category equipped with a distinguished 
involution in its Drinfeld center), 
can be recovered by taking the superadditive
envelope then passing to the underlying 2-category. 
The main goal of this article is to develop this language and the
related formal constructions, in the hope that these foundations may prove useful in other contexts.
\end{abstract}

\maketitle  

\section{Introduction}

\noindent
1.1.
In representation theory,
one finds many monoidal
categories and 2-categories playing an increasingly prominent role.
Examples include
the {\em Brauer category} $\mathcal B(\delta)$,
the {\em oriented Brauer category} $\mathcal{OB}(\delta)$,
the  {\em Temperley-Lieb category}
$\mathcal{TL}(\delta)$, the 
{\em web category} $\mathcal{W}eb(U_q(\mathfrak{sl}_n))$,
the category of {\em Soergel bimodules} $\mathcal{S}(W)$ associated to a Coxeter group $W$,
and the {\em Kac-Moody 2-category} $\mathfrak{U}(\mathfrak{g})$ associated to
a Kac-Moody algebra $\mathfrak{g}$.
Each of these categories, or perhaps its additive Karoubi
envelope,
has a definition ``in nature,'' as well as a diagrammatic description by
generators and relations.
It is also often instructive after taking additive Karoubi
envelope to pass to the {\em Grothendieck ring}.
Let us go through our examples in turn.
\begin{itemize}
\item
The Brauer category $\mathcal B(\delta)$
is the symmetric monoidal category generated by a self-dual object of
dimension $\delta \in \mathbb{C}$.
By \cite[Theorem 2.6]{LZ}, it may be presented as the strict monoidal category
with one generating object $\color{darkblue}{\cdot}$ and three generating morphisms $\mathord{
\begin{tikzpicture}[baseline = -0.8]
	\draw[-,thick,darkblue] (0.18,-.2) to (-0.18,.3);
	\draw[thick,darkblue,-] (-0.18,-.2) to (0.18,.3);
\end{tikzpicture}
}:
{\color{darkblue}\cdot}\,\otimes\,{\color{darkblue}\cdot}\,
\rightarrow 
{\color{darkblue}\cdot}\,\otimes\,{\color{darkblue}\cdot}\,$,
$\mathord{
\begin{tikzpicture}[baseline = 0]
	\draw[-,thick,darkblue] (0.3,0.25) to[out=-90, in=0] (0.1,-0.05);
	\draw[-,thick,darkblue] (0.1,-0.05) to[out = 180, in = -90] (-0.1,0.25);
\end{tikzpicture}
}: \unit\rightarrow
{\color{darkblue}\cdot}\,\otimes\,{\color{darkblue}\cdot}\,$
and 
$\mathord{
\begin{tikzpicture}[baseline = 0]
	\draw[-,thick,darkblue] (0.3,-0.1) to[out=90, in=0] (0.1,0.2);
	\draw[-,thick,darkblue] (0.1,0.2) to[out = 180, in = 90] (-0.1,-.1);
\end{tikzpicture}
}:
{\color{darkblue}\cdot}\,\otimes\,{\color{darkblue}\cdot}\,
\rightarrow \unit$, 
subject to the following relations:
\begin{align*}
\mathord{
\begin{tikzpicture}[baseline = 0]
	\draw[-,thick,darkblue] (0.28,0) to[out=90,in=-90] (-0.28,.6);
	\draw[-,thick,darkblue] (-0.28,0) to[out=90,in=-90] (0.28,.6);
	\draw[-,thick,darkblue] (0.28,-.6) to[out=90,in=-90] (-0.28,0);
	\draw[-,thick,darkblue] (-0.28,-.6) to[out=90,in=-90] (0.28,0);
\end{tikzpicture}
}\,
&=\,
\mathord{
\begin{tikzpicture}[baseline = 0]
	\draw[-,thick,darkblue] (0.18,-.4) to (0.18,.4);
	\draw[-,thick,darkblue] (-0.18,-.4) to (-0.18,.4);
\end{tikzpicture}
}\:,
&
\mathord{
\begin{tikzpicture}[baseline = 0]
	\draw[-,thick,darkblue] (0.45,.6) to (-0.45,-.6);
	\draw[-,thick,darkblue] (0.45,-.6) to (-0.45,.6);
        \draw[-,thick,darkblue] (0,-.6) to[out=90,in=-90] (-.45,0);
        \draw[-,thick,darkblue] (-0.45,0) to[out=90,in=-90] (0,0.6);
\end{tikzpicture}
}
&=
\mathord{
\begin{tikzpicture}[baseline = 0]
	\draw[-,thick,darkblue] (0.45,.6) to (-0.45,-.6);
	\draw[-,thick,darkblue] (0.45,-.6) to (-0.45,.6);
        \draw[-,thick,darkblue] (0,-.6) to[out=90,in=-90] (.45,0);
        \draw[-,thick,darkblue] (0.45,0) to[out=90,in=-90] (0,0.6);
\end{tikzpicture}
}\:,
&\!\!\!\!\!\!\!\!\mathord{
\begin{tikzpicture}[baseline = 0]
  \draw[-,thick,darkblue] (0.2,0) to (0.2,.5);
	\draw[-,thick,darkblue] (0.2,0) to[out=-90, in=0] (0,-0.35);
	\draw[-,thick,darkblue] (0,-0.35) to[out = 180, in = -90] (-0.2,0);
	\draw[-,thick,darkblue] (-0.2,0) to[out=90, in=0] (-0.4,0.35);
	\draw[-,thick,darkblue] (-0.4,0.35) to[out = 180, in =90] (-0.6,0);
  \draw[-,thick,darkblue] (-0.6,0) to (-0.6,-.5);
\end{tikzpicture}
}
\,&=\,
\mathord{\begin{tikzpicture}[baseline=0]
  \draw[-,thick,darkblue] (0,-0.4) to (0,.4);
\end{tikzpicture}
}\:,
\qquad
\mathord{
\begin{tikzpicture}[baseline = 0]
  \draw[-,thick,darkblue] (0.2,0) to (0.2,-.5);
	\draw[-,thick,darkblue] (0.2,0) to[out=90, in=0] (0,0.35);
	\draw[-,thick,darkblue] (0,0.35) to[out = 180, in = 90] (-0.2,0);
	\draw[-,thick,darkblue] (-0.2,0) to[out=-90, in=0] (-0.4,-0.35);
	\draw[-,thick,darkblue] (-0.4,-0.35) to[out = 180, in =-90] (-0.6,0);
  \draw[-,thick,darkblue] (-0.6,0) to (-0.6,.5);
\end{tikzpicture}
}\,
=\,
\mathord{\begin{tikzpicture}[baseline=0]
  \draw[-,thick,darkblue] (0,-0.4) to (0,.4);
\end{tikzpicture}
}\:,\\
\mathord{
\begin{tikzpicture}[baseline = 0]
\draw[-,thick,darkblue](.6,-.3) to (.1,.4);
	\draw[-,thick,darkblue] (0.6,0.4) to[out=-140, in=0] (0.1,-0.1);
	\draw[-,thick,darkblue] (0.1,-0.1) to[out = 180, in = -90] (-0.2,0.4);
\end{tikzpicture}
}&=
\mathord{
\begin{tikzpicture}[baseline = 0]
\draw[-,thick,darkblue](-.5,-.3) to (0,.4);
	\draw[-,thick,darkblue] (0.3,0.4) to[out=-90, in=0] (0,-0.1);
	\draw[-,thick,darkblue] (0,-0.1) to[out = 180, in = -40] (-0.5,0.4);
\end{tikzpicture}
}\,,\!\!\!\!\!\!\!\!
&
\mathord{
\begin{tikzpicture}[baseline = 0]
	\draw[-,thick,darkblue] (0.28,-0.2) to[out=90,in=-90] (-0.28,.4);
	\draw[-,thick,darkblue] (-0.28,-0.2) to[out=90,in=-90] (0.28,.4);
	\draw[-,thick,darkblue] (0,-.5) to[out=180,in=-90] (-0.28,-0.2);
	\draw[-,thick,darkblue] (0,-.5) to[out=0,in=-90] (0.28,-0.2);
\end{tikzpicture}
}
&=
\mathord{
\begin{tikzpicture}[baseline = 0]
	\draw[-,thick,darkblue] (0,-.2) to[out=180,in=-90] (-0.25,0.3);
	\draw[-,thick,darkblue] (0,-.2) to[out=0,in=-90] (0.25,0.3);
\end{tikzpicture}
}\,,
&
\mathord{
\begin{tikzpicture}[baseline = 0]
	\draw[-,thick,darkblue] (0,-.25) to[out=180,in=-90] (-0.28,0.05);
	\draw[-,thick,darkblue] (0,-.25) to[out=0,in=-90] (0.28,0.05);
	\draw[-,thick,darkblue] (0,.35) to[out=180,in=90] (-0.28,0.05);
	\draw[-,thick,darkblue] (0,.35) to[out=0,in=90] (0.28,0.05);
\end{tikzpicture}
}\,
&=\delta.
\end{align*}
Here, we are using the well-known string calculus for morphisms in a
strict monoidal category as in \cite{BK}.
We remark also that the additive Karoubi envelope of $\mathcal B(\delta)$
is Deligne's interpolating category
$\operatorname{REP}(O_\delta)$.
There is a similar story for the oriented Brauer category.
It is the symmetric monoidal category generated by a dual pair of
objects of
dimension $\delta$. An explicit presentation is recorded in
\cite[Theorem 1.1]{BCNR}.
Its additive Karoubi envelope 
is Deligne's interpolating category
$\operatorname{REP}(GL_\delta)$.
\item
For $\delta = -(q+q^{-1}) \in \Q(q)$,
the additive Karoubi envelope of 
$\mathcal{TL}(\delta)$ is monoidally equivalent to the 
category of finite-dimensional
representations of the quantum group
$U_q(\mathfrak{sl}_2)$.
More generally, for $n \geq 2$,
the additive Karoubi envelope of 
$\mathcal{W}eb(U_q(\mathfrak{sl}_n))$ is monoidally equivalent to the 
category of finite-dimensional
representations of 
$U_q(\mathfrak{sl}_n)$.
An explicit diagrammatic presentation 
was derived by Cautis,
Kamnitzer and Morrison \cite{CKM}, building on
the influential work of Kuperberg \cite{K} which treated the case $n=3$.
\item
When $W$ is a Weyl group, Soergel \cite{S} showed that 
$\mathcal{S}(W)$ is monoidally equivalent to
the Hecke category $\mathcal H(G/B)$ of Kazhdan-Lusztig (certain 
$B$-equivariant sheaves on the associated Lie group $G$).
In general, $\mathcal{S}(W)$ 
is the additive Karoubi envelope of the category
of {\em Bott-Samelson bimodules}. In almost all cases, a diagrammatic presentation of the latter monoidal category 
has been derived by Elias and Williamson
\cite{EW}.
The Grothendieck ring $K_0(\mathcal{S}(W))$
is isomorphic to the group ring of $W$;
if one incorporates the natural grading into the picture one actually
gets the Iwahori-Hecke algebra $H_q(W)$ associated to $W$.
\item 
The Kac-Moody 2-category $\mathfrak{U}(\mathfrak{g})$ was defined by generators and relations by 
Rouquier \cite{Rou} and
Khovanov-Lauda \cite{KL3}; see also \cite{B}. The Grothendieck ring of its additive
Karoubi envelope is naturally an idempotented ring, with idempotents
indexed by the underlying weight lattice, and is
isomorphic
to the idempotented integral form $\dot{U}(\mathfrak{g})$ of the universal
enveloping algebra of $\mathfrak{g}$;
if one incorporates the grading one gets Lusztig's idempotented
integral form 
$\dot U_q(\mathfrak{g})$ of the associated quantum group.
(These statements are still only conjectural
outside of finite type.)
\end{itemize}

\vspace{2mm}
\noindent
1.2.
We are interested in this article in {\em superalgebra},
i.e. $\Z/2$-graded algebra. 
Our motivation comes from the belief that
there should be interesting 
super analogs of all of the categories just mentioned. In fact, 
they are already known in several cases. For example, 
analogs of the Brauer and
oriented Brauer categories are suggested by \cite{KT} and
\cite{JK}, respectively.
Also in \cite{BE}, we have defined a super analog of the
Kac-Moody 2-category, building on \cite{EL} which treated the case of
$\mathfrak{sl}_2$.
In thinking about such questions, one quickly runs into some basic
foundational issues. To start with, already in the literature, there are several competing
notions as to what should be called a ``super
monoidal category.'' The goal of the paper is to clarify 
these notions
and the connections between them; see also \cite{Usher} for further developments.

Let $\k$ be a field of characteristic different from 2.
A {\em superspace} is a $\Z/2$-graded vector space $V = V_\0 \oplus
V_\1$. 
We use the notation $|v|$ for the parity of a homogeneous vector $v$
in a superspace. Formulae involving this notation for 
inhomogeneous $v$ should be interpreted by extending additively from
the homogeneous case.
 
Let $\SVEC$ (resp.\ $\fdSVEC$) be the category of all \smallcat
superspaces (resp.\ finite dimensional superspaces) and (not
necessarily homogeneous) linear
maps. 
These categories
possess
some additional structure:
\begin{itemize}
\item
A linear map between superspaces $V$ and $W$
is {\em even} (resp.\ {\em odd})
if it preserves (resp. reverses) the parity of vectors.
Moreover, any linear map
$f:V \rightarrow W$ 
decomposes uniquely as a sum $f = f_\0 + f_\1$ with $f_\0$ even and
$f_\1$ odd.
This makes each morphism space $\Hom_{\SVEC}(V,W)$ into a superspace.
\item
The usual $\k$-linear tensor product of two superspaces is again a superspace with
$(V \otimes W)_\0 = V_\0 \otimes W_{\0} \oplus
V_{\1} \otimes W_{\1}$
and
$(V \otimes W)_\1 = V_\0 \otimes W_{\1} \oplus
V_{\1} \otimes W_{\0}$.
Also the tensor product $f \otimes g$ of two linear maps is 
the linear map defined from $(f \otimes g)(v \otimes w) := (-1)^{|g||v|}f(v)
\otimes g(w)$.
\end{itemize}
Let $\SVec$ be the 
subcategory of $\SVEC$ consisting of all superspaces but only the even
linear maps.
The restriction of the tensor product operation just defined 
gives a functor
$-\otimes-:\SVec \times \SVec \rightarrow \SVec$ 
making $\SVec$ into a monoidal category.
However, $\SVEC$ itself is {\em not} monoidal in the usual sense, because of the sign in the
following formula for composing
tensor products of linear maps:
\begin{equation}\label{sint}
(f \otimes g) \circ (h \otimes k) =
(-1)^{|g||h|}(f \circ h) \otimes (g \circ k).
\end{equation}
In fact, $\SVEC$ is what we'll call a {\em monoidal supercategory}.
We proceed to the formal definitions.

\begin{definition}\label{defsupercat}
(i) A {\em supercategory} means a $\SVec$-enriched category,
i.e. each morphism space is a superspace
and composition induces an even linear map.
(We refer to 
\cite[$\S$1.2]{kelly} for the basic language of enriched categories.)

(ii) A {\em superfunctor} $F:\A \rightarrow \B$
between
supercategories is a $\SVec$-enriched functor, i.e. the
function $\Hom_{\A}(\lambda,\mu) \rightarrow
\Hom_{\B}(F\lambda, F\mu), f \mapsto Ff$
is an even linear map
for all $\lambda, \mu \in \ob \A$.
(See \cite[$\S$1.2]{kelly} again.)

(iii) Given superfunctors $F, G:\A \rightarrow \B$,
a {\em supernatural transformation}
$x:F \Rightarrow G$ 
is a family of 
morphisms
$x_\lambda = x_{\lambda,\0}+x_{\lambda,\1} \in \Hom_{\mathcal
  \B}(F\lambda, G\lambda)$
for $\lambda\in\ob\A$, such that
$|x_{\lambda,p}| = p$ and
$x_{\mu,p} \circ Ff = (-1)^{p|f|}
Gf \circ x_{\lambda,p}$ 
for all
$p \in \Z/2$ and $f\in \Hom_{\A}(\lambda, \mu)$.
The supernatural transformation $x$ decomposes as a sum of 
{homogeneous} supernatural transformations as
$x=x_{\0}+x_{\1}$
where $(x_p)_\lambda := x_{\lambda,p}$, making
the space $\Hom(F,G)$ of all supernatural transformations
from $F$ to $G$ into a superspace.
(Even supernatural transformations are just the same as the $\SVec$-enriched
natural transformations of \cite[$\S$1.2]{kelly}.)

(iv)
A superfunctor $F:\A\rightarrow\B $ is a {\em superequivalence} 
if there is a superfunctor $G:\B\rightarrow \A$ 
such that $F \circ G$ and $G \circ F$ are isomorphic to
identities via 
even supernatural transformations.
To check that $F$ is
a superequivalence, it suffices to
show that it is full, faithful, and {\em evenly dense}, i.e. every
object of $\B$ should be isomorphic 
to an object in the image of $F$ via an even isomorphism.
 
(v)
For any supercategory $\A$, the {\em underlying category}
$\underline{\A}$ is the category with the same objects as $\A$ but only its
even morphisms.
If $F:\A \rightarrow \B$ is a superfunctor between supercategories, it
restricts to $\underline{F}:\underline{\A} \rightarrow
\underline{\B}$.
Also an even supernatural transformation $x:F \Rightarrow G$ 
is the same data as a natural transformation $x:\underline{F} \Rightarrow \underline{G}$.
(These definitions are special cases of ones in
\cite[$\S$1.3]{kelly}.)
\end{definition}

\begin{examples}\label{snuggly}
(i) We've already explained how to make
$\SVEC$ into a supercategory. The underlying category
is $\SVec$.
 
(ii) A {\em superalgebra} is a superspace $A = A_\0 \oplus A_\1$
equipped with an even linear map $m_A:A \otimes A \rightarrow A$
making $A$ into an associative, unital algebra; we denote
the image of $a \otimes b$ under this map simply by $ab$.
Any superalgebra $A$
can be viewed as a supercategory $\A$ with one object whose
endomorphism superalgebra is $A$.

(iii) 
Suppose we are given superalgebras $A$ and $B$.
Then there is a supercategory $A\lrSMod B$ consisting of
all \smallcat 
$(A,B)$-superbimodules and superbimodule homomorpisms.
Here, 
an {\em $(A,B)$-superbimodule} is a superspace $V$
plus an even linear map
$m_V:A \otimes V \otimes B\rightarrow V$
making $V$ into an $(A,B)$-bimodule in the usual sense;
we denote the image of $a \otimes v \otimes b$ 
under this map simply by $a v b$.
A {\em superbimodule homomorphism} $f:V \rightarrow W$ 
is a linear map such that $m_W\circ (1_A \otimes f \otimes 1_B) = f
\circ m_V$. In view of the definition of tensor product of linear
maps between superspaces, this means explicitly that
$f(avb) = (-1)^{|f||a|} a f(v) b$.

(iv)
For any two supercategories $\A$ and $\B$, 
there is a supercategory $\mathcal{H}om(\A, \B)$ consisting of all
superfunctors and supernatural transformations.
\end{examples}

The monoidal category $\SVec$ is symmetric with 
braiding 
$u \otimes v \mapsto (-1)^{|u||v|} v \otimes u$.
As in \cite[$\S$1.4]{kelly},
this allows us to introduce a product operation
$- \boxtimes - $ which makes the category
$\SCat$ of all \smallcat supercategories and
superfunctors into a monoidal category.
On objects (i.e. supercategories) $\A$ and $\B$, this operation is defined by letting
 $\A \boxtimes \B$ be the supercategory
whose objects are ordered pairs $(\lambda,\mu)$ of objects of $\A$
and $\B$, respectively, and 
$
\Hom_{\A \boxtimes \B}((\lambda,\mu), (\sigma,\tau)) :=
\Hom_{\A}(\lambda,\sigma) \otimes \Hom_{\B}(\mu,\tau).
$
Composition in $\A \boxtimes \B$
is defined using the symmetric braiding in $\SVec$, so that
$(f \otimes g) \circ (h \otimes k) =
(-1)^{|g||h|}(f \circ h) \otimes (g \circ k)$.
The unit object $\mathcal{I}$ is a distinguished
supercategory with one object whose endomorphism superalgebra is $\k$ 
(concentrated in even parity).
 The definition of $-\boxtimes -$ on morphisms (i.e. superfunctors) is
obvious, as are the coherence maps.

\begin{remark}
Example~\ref{snuggly}(iii) is a special case of Example~\ref{snuggly}(iv).
Let $\A$ and $\B$ be defined from superalgebras $A$ and $B$ as in
Example~\ref{snuggly}(ii).
Let $\B^{\sop}$ be the
supercategory with $\ob \B^{\sop} := \ob \B$, and new composition law
$a \bullet b := (-1)^{|a||b|} b \circ a$.
Then the supercategory $\mathcal{H}om(\A \boxtimes \B^{\sop}, \SVEC)$ 
is isomorphic to $A\lrSMod
B$ via the 
superfunctor which identifies $V:\A \boxtimes \B^{\sop} \rightarrow \SVEC$
with the superspace obtained by
evaluating at the only object, viewed as a superbimodule so 
$avb := (-1)^{|b||v|} V(a\otimes b)(v)$. The data of a
supernatural transformation
$f:V \rightarrow W$ is exactly the same as the data of a superbimodule
homomorphism.
\end{remark}

\begin{definition}\label{dash}
(i) A {\em monoidal supercategory} is a supercategory $\A$
equipped with a superfunctor $-\otimes -:\A \boxtimes \A
\rightarrow \A$, a unit object $\unit$, and 
even supernatural isomorphisms\footnote{By $(- \otimes -) \otimes -$ we mean the
superfunctor 
$(\A \boxtimes \A) \boxtimes \A \rightarrow \A$
obtained by applying $\otimes$ twice in the order indicated.
Similarly, $- \otimes (- \otimes -)$ is a superfunctor
$\A \boxtimes (\A \boxtimes \A) \rightarrow \A$, but we are viewing it
as a superfunctor $(\A \boxtimes \A) \boxtimes \A\rightarrow \A$ by using the
canonical isomorphism defined by the associator in $\SCat$.
Also, $\unit \otimes -:\A\rightarrow\A$ and $- \otimes \unit:\A
\rightarrow \A$ denote the
superfunctors defined by tensoring on the left and right by the
unit object, respectively.}
$a:(- \otimes -) \otimes -
\stackrel{\sim}{\Rightarrow} - \otimes (- \otimes -)$,
$l:\unit \otimes - \stackrel{\sim}{\Rightarrow} -$
and $r:- \otimes \unit \stackrel{\sim}{\Rightarrow} -$
called
{\em coherence maps},
which satisfy
axioms analogous to the ones of a monoidal category.
In any monoidal supercategory,
tensor products of morphisms 
compose according to the same rule (\ref{sint}) that we already observed
in $\SVEC$. We call this the {\em super interchange law}.

(ii)
Given monoidal supercategories $\A$ and $\B$, a {\em monoidal
  superfunctor} is a
superfunctor $F:\A \rightarrow \B$ plus 
coherence maps $c: (F\: -)
\otimes (F\: -)  \stackrel{\sim}{\Rightarrow}
F(- \otimes -)$ and $i:\unit_\B
\stackrel{\sim}{\rightarrow} F\unit_\A$ satisfying axioms 
analogous to the ones of a monoidal category;
we require that $c$ is an even supernatural isomorphism and that $i$
is an even isomorphism. (Note we implicitly assume all monoidal
(super)functors are strong throughout the article.)

(iii) Given monoidal superfunctors $F, G:\A \rightarrow \B$,
a {\em monoidal natural transformation}
is an even supernatural transformation $x:F \Rightarrow G$ 
such that 
\begin{align*}
x_{\lambda \otimes \mu} \circ (c_F)_{\lambda,\mu} &=
(c_G)_{\lambda,\mu} \circ (x_\lambda \otimes x_\mu),\\
x_{\unit_\A} \circ i_F &= i_G,
\end{align*}
in $\Hom_\B((F\lambda)
\otimes (F\mu), G(\lambda \otimes \mu))$
and $\Hom_\B(\unit_\B, G \unit_\A)$, respectively.
(There is no such thing as a monoidal {\em super}natural transformation.)
\end{definition}

A monoidal supercategory (resp. superfunctor) is {\em strict} if its
coherence maps are identities.
There is a version of Mac Lane's {\em Coherence Theorem} \cite{Mac} for monoidal
supercategories. It implies that any
monoidal supercategory $\A$ is {\em monoidally superequivalent} to a strict
monoidal supercategory $\B$,
i.e. there are monoidal superfunctors $F:\A \rightarrow \B$ and $G:\B
\rightarrow \A$ such that 
$G \circ F$ and $F \circ G$ are
isomorphic to identities via monoidal natural transformations;
equivalently, there is a monoidal superfunctor
$F:\A \rightarrow \B$
which defines a superequivalence between the underlying supercategories.

With a little care about signs, 
the string calculus mentioned earlier can be used to represent morphisms in a strict
monoidal supercategory $\A$. 
Thus, a morphism $f \in \Hom_\A(\lambda,\mu)$
is the picture
\begin{equation}\label{jazzy}
\mathord{
\begin{tikzpicture}[baseline = 0]
	\draw[-,thick,darkred] (0.08,-.4) to (0.08,-.13);
	\draw[-,thick,darkred] (0.08,.4) to (0.08,.13);
      \draw[thick,darkred] (0.08,0) circle (4pt);
   \node at (0.08,0) {\color{darkred}$\scriptstyle{f}$};
   \node at (0.08,-.53) {$\scriptstyle{\lambda}$};
   \node at (0.08,.53) {$\scriptstyle{\mu}$};
\end{tikzpicture}
}\,.
\end{equation}
Often we will omit the object labels $\lambda, \mu$ here.
Then the horizontal and vertical compositions $f \otimes g$ and $f
\circ g$ are obtained by horizontally
and vertically stacking:
$$
\mathord{
\begin{tikzpicture}[baseline = -2]
	\draw[-,thick,darkred] (0.08,-.4) to (0.08,-.13);
	\draw[-,thick,darkred] (0.08,.4) to (0.08,.13);
      \draw[thick,darkred] (0.08,0) circle (4pt);
   \node at (0.08,0) {\color{darkred}$\scriptstyle{f}$};
\end{tikzpicture}
}
\otimes
\mathord{
\begin{tikzpicture}[baseline = -2]
	\draw[-,thick,darkred] (0.08,-.4) to (0.08,-.13);
	\draw[-,thick,darkred] (0.08,.4) to (0.08,.13);
      \draw[thick,darkred] (0.08,0) circle (4pt);
   \node at (0.08,0) {\color{darkred}$\scriptstyle{g}$};
\end{tikzpicture}
}=
\mathord{
\begin{tikzpicture}[baseline = -2]
	\draw[-,thick,darkred] (0.08,-.4) to (0.08,-.13);
	\draw[-,thick,darkred] (0.08,.4) to (0.08,.13);
      \draw[thick,darkred] (0.08,0) circle (4pt);
   \node at (0.08,0) {\color{darkred}$\scriptstyle{g}$};
	\draw[-,thick,darkred] (-.8,-.4) to (-.8,-.13);
	\draw[-,thick,darkred] (-.8,.4) to (-.8,.13);
      \draw[thick,darkred] (-.8,0) circle (4pt);
   \node at (-.8,0) {\color{darkred}$\scriptstyle{f}$};
\end{tikzpicture}
}\,,
\qquad\qquad
\mathord{
\begin{tikzpicture}[baseline = -2]
	\draw[-,thick,darkred] (0.08,-.4) to (0.08,-.13);
	\draw[-,thick,darkred] (0.08,.4) to (0.08,.13);
      \draw[thick,darkred] (0.08,0) circle (4pt);
   \node at (0.08,0) {\color{darkred}$\scriptstyle{f}$};
\end{tikzpicture}
}
\circ
\mathord{
\begin{tikzpicture}[baseline = -2]
	\draw[-,thick,darkred] (0.08,-.4) to (0.08,-.13);
	\draw[-,thick,darkred] (0.08,.4) to (0.08,.13);
      \draw[thick,darkred] (0.08,0) circle (4pt);
   \node at (0.08,0) {\color{darkred}$\scriptstyle{g}$};
\end{tikzpicture}
}
=
\mathord{
\begin{tikzpicture}[baseline = 8]
	\draw[-,thick,darkred] (0.08,-.4) to (0.08,-.13);
	\draw[-,thick,darkred] (0.08,.57) to (0.08,.13);
	\draw[-,thick,darkred] (0.08,.83) to (0.08,1.1);
      \draw[thick,darkred] (0.08,0) circle (4pt);
      \draw[thick,darkred] (0.08,.7) circle (4pt);
   \node at (0.08,0) {\color{darkred}$\scriptstyle{g}$};
   \node at (0.08,.71) {\color{darkred}$\scriptstyle{f}$};
\end{tikzpicture}
}\:.
$$
More complicated pictures should be interpreted by {\em first}
composing horizontally {\em then} composing vertically. For example,
the following is $(f \otimes g) \circ (h \otimes k)$:
$$
\mathord{
\begin{tikzpicture}[baseline = 8]
	\draw[-,thick,darkred] (0.08,-.4) to (0.08,-.13);
	\draw[-,thick,darkred] (0.08,.57) to (0.08,.13);
	\draw[-,thick,darkred] (0.08,1.17) to (0.08,.83);
      \draw[thick,darkred] (0.08,0) circle (4pt);
      \draw[thick,darkred] (0.08,.7) circle (4pt);
   \node at (0.08,0) {\color{darkred}$\scriptstyle{h}$};
   \node at (0.08,.71) {\color{darkred}$\scriptstyle{f}$};
\end{tikzpicture}
}\qquad
\mathord{
\begin{tikzpicture}[baseline = 8]
	\draw[-,thick,darkred] (0.08,-.4) to (0.08,-.13);
	\draw[-,thick,darkred] (0.08,.57) to (0.08,.13);
	\draw[-,thick,darkred] (0.08,1.17) to (0.08,.83);
      \draw[thick,darkred] (0.08,0) circle (4pt);
      \draw[thick,darkred] (0.08,.7) circle (4pt);
   \node at (0.08,0) {\color{darkred}$\scriptstyle{k}$};
   \node at (0.08,.71) {\color{darkred}$\scriptstyle{g}$};
\end{tikzpicture}
}\:\:.
$$
Unlike in the purely even setting, this is {\em not} the same as $(f \circ h) \otimes (g
\circ k)$.
In fact, in pictures, the super interchange law tells us that
\begin{equation}
\mathord{
\begin{tikzpicture}[baseline = 0]
	\draw[-,thick,darkred] (0.08,-.4) to (0.08,-.23);
	\draw[-,thick,darkred] (0.08,.4) to (0.08,.03);
      \draw[thick,darkred] (0.08,-0.1) circle (4pt);
   \node at (0.08,-0.1) {\color{darkred}$\scriptstyle{g}$};
	\draw[-,thick,darkred] (-.8,-.4) to (-.8,-.03);
	\draw[-,thick,darkred] (-.8,.4) to (-.8,.23);
      \draw[thick,darkred] (-.8,0.1) circle (4pt);
   \node at (-.8,.1) {\color{darkred}$\scriptstyle{f}$};
\end{tikzpicture}
}
\quad=\quad
\mathord{
\begin{tikzpicture}[baseline = 0]
	\draw[-,thick,darkred] (0.08,-.4) to (0.08,-.13);
	\draw[-,thick,darkred] (0.08,.4) to (0.08,.13);
      \draw[thick,darkred] (0.08,0) circle (4pt);
   \node at (0.08,0) {\color{darkred}$\scriptstyle{g}$};
	\draw[-,thick,darkred] (-.8,-.4) to (-.8,-.13);
	\draw[-,thick,darkred] (-.8,.4) to (-.8,.13);
      \draw[thick,darkred] (-.8,0) circle (4pt);
   \node at (-.8,0) {\color{darkred}$\scriptstyle{f}$};
\end{tikzpicture}
}
\quad=\quad
(-1)^{|f||g|}\:
\mathord{
\begin{tikzpicture}[baseline = 0]
	\draw[-,thick,darkred] (0.08,-.4) to (0.08,-.03);
	\draw[-,thick,darkred] (0.08,.4) to (0.08,.23);
      \draw[thick,darkred] (0.08,0.1) circle (4pt);
   \node at (0.08,0.1) {\color{darkred}$\scriptstyle{g}$};
	\draw[-,thick,darkred] (-.8,-.4) to (-.8,-.23);
	\draw[-,thick,darkred] (-.8,.4) to (-.8,.03);
      \draw[thick,darkred] (-.8,-0.1) circle (4pt);
   \node at (-.8,-.1) {\color{darkred}$\scriptstyle{f}$};
\end{tikzpicture}
}\:.
\end{equation}

\begin{examples}
(i) The supercategory
$\SVEC$ is a monoidal supercategory with tensor functor as defined
above.
The unit object is $\k$.
More generally, for a superalgebra $A$, 
$A\lrSMod A$ is a monoidal supercategory with
tensor functor
defined by taking the usual tensor product of superbimodules over
$A$.
The unit object is the regular superbimodule $A$.

(ii)
For a supercategory $\A$, 
$\mathcal{E}nd(\A)$ is a 
strict monoidal supercategory, with $-\otimes-$ defined on
functors $F, G:\A \rightarrow \A$ by 
$F \otimes G := F \circ G$,
and on supernatural transformations $x:F \Rightarrow G$ and $y:H
\Rightarrow K$ so that
$(x \otimes y)_\lambda := x_{K \lambda} \circ F
y_{\lambda}$.
The unit object is the identity functor $I :\A \rightarrow \A$.
Later on, we will denote the horizontal compositions $F \otimes G$ and
$x \otimes y$
of two superfunctors or two supernatural transformations
simply by $FG$ and $xy$, respectively.
In more complicated horizontal compositions, we often adopt the standard shorthand of writing simply
$F$ in place of the identity morphism $1_F$, e.g. for $x:F \Rightarrow
G$, $y:H \Rightarrow
K$, the expressions $F y$ and $x H$ denote $1_F
\,y:FH \Rightarrow FK$
and $x 1_H:FH \Rightarrow GH$, respectively.

(iii)
Here is a purely diagrammatic example.
The
{\em odd Brauer supercategory} 
is the strict monoidal supercategory
$\mathcal{SB}$
with one generating object $\color{darkgreen}{\cdot}$, an {\em even} generating morphism $\mathord{
\begin{tikzpicture}[baseline = -0.8]
	\draw[-,thick,darkgreen] (0.18,-.2) to (-0.18,.3);
	\draw[thick,darkgreen,-] (-0.18,-.2) to (0.18,.3);
\end{tikzpicture}
}:
{\color{darkgreen}\cdot}\,\otimes\,{\color{darkgreen}\cdot}\,
\rightarrow 
{\color{darkgreen}\cdot}\,\otimes\,{\color{darkgreen}\cdot}\,$,
and two {\em odd} generating morphisms
$\mathord{
\begin{tikzpicture}[baseline = 0]
	\draw[-,thick,darkgreen] (0.3,0.25) to[out=-90, in=0] (0.1,-0.05);
	\draw[-,thick,darkgreen] (0.1,-0.05) to[out = 180, in = -90] (-0.1,0.25);
\end{tikzpicture}
}: \unit\rightarrow
{\color{darkgreen}\cdot}\,\otimes\,{\color{darkgreen}\cdot}\,$
and 
$\mathord{
\begin{tikzpicture}[baseline = 0]
	\draw[-,thick,darkgreen] (0.3,-0.1) to[out=90, in=0] (0.1,0.2);
	\draw[-,thick,darkgreen] (0.1,0.2) to[out = 180, in = 90] (-0.1,-.1);
\end{tikzpicture}
}:
{\color{darkgreen}\cdot}\,\otimes\,{\color{darkgreen}\cdot}\,
\rightarrow \unit$, 
subject to the following relations:
\begin{align*}
\mathord{
\begin{tikzpicture}[baseline = 0]
	\draw[-,thick,darkgreen] (0.28,0) to[out=90,in=-90] (-0.28,.6);
	\draw[-,thick,darkgreen] (-0.28,0) to[out=90,in=-90] (0.28,.6);
	\draw[-,thick,darkgreen] (0.28,-.6) to[out=90,in=-90] (-0.28,0);
	\draw[-,thick,darkgreen] (-0.28,-.6) to[out=90,in=-90] (0.28,0);
\end{tikzpicture}
}\,
&=\,
\mathord{
\begin{tikzpicture}[baseline = 0]
	\draw[-,thick,darkgreen] (0.18,-.4) to (0.18,.4);
	\draw[-,thick,darkgreen] (-0.18,-.4) to (-0.18,.4);
\end{tikzpicture}
}\:,
&
\mathord{
\begin{tikzpicture}[baseline = 0]
	\draw[-,thick,darkgreen] (0.45,.6) to (-0.45,-.6);
	\draw[-,thick,darkgreen] (0.45,-.6) to (-0.45,.6);
        \draw[-,thick,darkgreen] (0,-.6) to[out=90,in=-90] (-.45,0);
        \draw[-,thick,darkgreen] (-0.45,0) to[out=90,in=-90] (0,0.6);
\end{tikzpicture}
}
&=
\mathord{
\begin{tikzpicture}[baseline = 0]
	\draw[-,thick,darkgreen] (0.45,.6) to (-0.45,-.6);
	\draw[-,thick,darkgreen] (0.45,-.6) to (-0.45,.6);
        \draw[-,thick,darkgreen] (0,-.6) to[out=90,in=-90] (.45,0);
        \draw[-,thick,darkgreen] (0.45,0) to[out=90,in=-90] (0,0.6);
\end{tikzpicture}
}\:,
&\!\!\!\!\!\!\!\!\mathord{
\begin{tikzpicture}[baseline = 0]
  \draw[-,thick,darkgreen] (0.2,0) to (0.2,.5);
	\draw[-,thick,darkgreen] (0.2,0) to[out=-90, in=0] (0,-0.35);
	\draw[-,thick,darkgreen] (0,-0.35) to[out = 180, in = -90] (-0.2,0);
	\draw[-,thick,darkgreen] (-0.2,0) to[out=90, in=0] (-0.4,0.35);
	\draw[-,thick,darkgreen] (-0.4,0.35) to[out = 180, in =90] (-0.6,0);
  \draw[-,thick,darkgreen] (-0.6,0) to (-0.6,-.5);
\end{tikzpicture}
}
\,&=\,
\mathord{\begin{tikzpicture}[baseline=0]
  \draw[-,thick,darkgreen] (0,-0.4) to (0,.4);
\end{tikzpicture}
}\:,
\qquad
\mathord{
\begin{tikzpicture}[baseline = 0]
  \draw[-,thick,darkgreen] (0.2,0) to (0.2,-.5);
	\draw[-,thick,darkgreen] (0.2,0) to[out=90, in=0] (0,0.35);
	\draw[-,thick,darkgreen] (0,0.35) to[out = 180, in = 90] (-0.2,0);
	\draw[-,thick,darkgreen] (-0.2,0) to[out=-90, in=0] (-0.4,-0.35);
	\draw[-,thick,darkgreen] (-0.4,-0.35) to[out = 180, in =-90] (-0.6,0);
  \draw[-,thick,darkgreen] (-0.6,0) to (-0.6,.5);
\end{tikzpicture}
}\,
=\,
-\:\,\mathord{\begin{tikzpicture}[baseline=0]
  \draw[-,thick,darkgreen] (0,-0.4) to (0,.4);
\end{tikzpicture}
}\:,\\
\mathord{
\begin{tikzpicture}[baseline = 0]
\draw[-,thick,darkgreen](.6,-.3) to (.1,.4);
	\draw[-,thick,darkgreen] (0.6,0.4) to[out=-140, in=0] (0.1,-0.1);
	\draw[-,thick,darkgreen] (0.1,-0.1) to[out = 180, in = -90] (-0.2,0.4);
\end{tikzpicture}
}&=
\mathord{
\begin{tikzpicture}[baseline = 0]
\draw[-,thick,darkgreen](-.5,-.3) to (0,.4);
	\draw[-,thick,darkgreen] (0.3,0.4) to[out=-90, in=0] (0,-0.1);
	\draw[-,thick,darkgreen] (0,-0.1) to[out = 180, in = -40] (-0.5,0.4);
\end{tikzpicture}
}\,,\!\!\!\!\!\!\!\!
&
\mathord{
\begin{tikzpicture}[baseline = 0]
	\draw[-,thick,darkgreen] (0.28,-0.2) to[out=90,in=-90] (-0.28,.4);
	\draw[-,thick,darkgreen] (-0.28,-0.2) to[out=90,in=-90] (0.28,.4);
	\draw[-,thick,darkgreen] (0,-.5) to[out=180,in=-90] (-0.28,-0.2);
	\draw[-,thick,darkgreen] (0,-.5) to[out=0,in=-90] (0.28,-0.2);
\end{tikzpicture}
}
&=
\mathord{
\begin{tikzpicture}[baseline = 0]
	\draw[-,thick,darkgreen] (0,-.2) to[out=180,in=-90] (-0.25,0.3);
	\draw[-,thick,darkgreen] (0,-.2) to[out=0,in=-90] (0.25,0.3);
\end{tikzpicture}
}\,.
\end{align*}
This was introduced in \cite{KT} where it is called the marked Brauer category, motivated by Schur-Weyl duality for the
Lie superalgebra $\mathfrak{p}_n(\C)$.
Unlike the Brauer category defined earlier, there is no parameter
$\delta$. Indeed, using the relations and super interchange, one can check that
$$
\mathord{
\begin{tikzpicture}[baseline = 0]
	\draw[-,thick,darkgreen] (0.28,0.2) to[out=-90,in=90] (-0.28,-.4);
	\draw[-,thick,darkgreen] (-0.28,0.2) to[out=-90,in=90] (0.28,-.4);
	\draw[-,thick,darkgreen] (0,.5) to[out=-180,in=90] (-0.28,0.2);
	\draw[-,thick,darkgreen] (0,.5) to[out=0,in=90] (0.28,0.2);
\end{tikzpicture}
}
=
-\:\mathord{
\begin{tikzpicture}[baseline = -4]
	\draw[-,thick,darkgreen] (0,.2) to[out=180,in=90] (-0.25,-0.3);
	\draw[-,thick,darkgreen] (0,.2) to[out=0,in=90] (0.25,-0.3);
\end{tikzpicture}
}\,
\qquad\text{hence}\qquad
\mathord{
\begin{tikzpicture}[baseline = 0]
	\draw[-,thick,darkgreen] (0,-.25) to[out=180,in=-90] (-0.28,0.05);
	\draw[-,thick,darkgreen] (0,-.25) to[out=0,in=-90] (0.28,0.05);
	\draw[-,thick,darkgreen] (0,.35) to[out=180,in=90] (-0.28,0.05);
	\draw[-,thick,darkgreen] (0,.35) to[out=0,in=90] (0.28,0.05);
\end{tikzpicture}
}\,
=
\frac{1}{2}\:
\mathord{
\begin{tikzpicture}[baseline = -1.6mm]
	\draw[-,thick,darkgreen] (0,-0.7) to[out=180,in=-90] (-0.28,-.4);
	\draw[-,thick,darkgreen] (0,-0.7) to[out=0,in=-90] (0.28,-.4);
	\draw[-,thick,darkgreen] (0.28,0.2) to[out=-90,in=90] (-0.28,-.4);
	\draw[-,thick,darkgreen] (-0.28,0.2) to[out=-90,in=90] (0.28,-.4);
	\draw[-,thick,darkgreen] (0,.5) to[out=-180,in=90] (-0.28,0.2);
	\draw[-,thick,darkgreen] (0,.5) to[out=0,in=90] (0.28,0.2);
\end{tikzpicture}
}
-\frac{1}{2}\:
\mathord{
\begin{tikzpicture}[baseline = -1.6mm]
	\draw[-,thick,darkgreen] (0,-0.7) to[out=180,in=-90] (-0.28,-.4);
	\draw[-,thick,darkgreen] (0,-0.7) to[out=0,in=-90] (0.28,-.4);
	\draw[-,thick,darkgreen] (0.28,0.2) to[out=-90,in=90] (-0.28,-.4);
	\draw[-,thick,darkgreen] (-0.28,0.2) to[out=-90,in=90] (0.28,-.4);
	\draw[-,thick,darkgreen] (0,.5) to[out=-180,in=90] (-0.28,0.2);
	\draw[-,thick,darkgreen] (0,.5) to[out=0,in=90] (0.28,0.2);
\end{tikzpicture}
}
 = 0.
$$
\end{examples}

\vspace{2mm}
\noindent
1.3.
Now we switch the focus to one of the competing notions. Instead of
working with additive categories {\em enriched in} the monoidal category $\SVec$, one can work with
module categories {\em over} the monoidal category
$\fdSVec$ (e.g. see \cite[$\S$7.1]{EGNO}), or equivalently, 
additive $\k$-linear categories equipped with a strict
action of the cyclic group $\Z/2$ (e.g. see \cite[$\S$1.3]{EW2}). We adopt the following language for such structures:

\begin{definition}\label{jazz}
(i) A {\em $\Pi$-category} $(\A, \Pi, \xi)$ 
is a $\k$-linear category $\A$
plus a $\k$-linear endofunctor $\Pi:\A \rightarrow
\A$ and a natural isomorphism
$\xi:\Pi^2 \stackrel{\sim}{\Rightarrow} I$
such that $\xi \Pi = \Pi \xi$ in $\Hom(\Pi^3, \Pi)$.
Note then that $\Pi$ is a self-inverse equivalence.

(ii) Given $\Pi$-categories $(\A, \Pi_\A, \xi_\A)$ and $(\B, \Pi_\B, \xi_\B)$, a {\em $\Pi$-functor} $F:\A \rightarrow \B$
is a $\k$-linear functor plus the data of a natural isomorphism
$\beta_F:\Pi_\B F\stackrel{\sim}{\Rightarrow} F \Pi_\A $ such that
$\xi_\B F (\xi_\A)^{-1}=\beta_F \Pi_\A \circ \Pi_\B \beta_F$
in $\Hom((\Pi_\B)^2F, F (\Pi_\A)^2)$.
For example, 
the identity functor $I$ is a $\Pi$-functor
with $\beta_I = 1_\Pi$, and
$\Pi$ is a $\Pi$-functor with
$\beta_\Pi := - 1_{\Pi^2}$.
Note also that the composition of two $\Pi$-functors 
$F:\A \rightarrow \mathcal
B$
and $G:\B \rightarrow \C$
is itself a $\Pi$-functor
with $\beta_{G F} := G \beta_F\circ \beta_G F$.

(iii) Given $\Pi$-functors $F, G:\A \rightarrow \B$,
a {\em $\Pi$-natural transformation} is a natural transformation 
$x:F \Rightarrow G$ such that $x
\Pi_\A\circ \beta_F= \beta_G \circ \Pi_\B x$  in $\Hom(\Pi_\B F, G \Pi_\A)$.
\end{definition}

There is a close relationship between supercategories and
$\Pi$-categories.
To explain this formally, we need the following intermediate
notion. Actually, our experience suggests this is often the most
convenient place to work in practice.

\begin{definition}\label{defscat}
A {\em $\Pi$-supercategory} $(\A, \Pi, \zeta)$ is a supercategory $\A$ plus the
extra data of a superfunctor $\Pi:\A \rightarrow \A$
and an odd supernatural isomorphism $\zeta:\Pi\stackrel{\sim}{\Rightarrow}
I$.
Note then that $\xi := \zeta\zeta:\Pi^2 \stackrel{\sim}{\Rightarrow} I$ is an even
supernatural isomorphism, i.e. $\A$ is equipped with canonical even isomorphisms
$\xi_\lambda: \Pi^2
\lambda \stackrel{\sim}{\rightarrow} \lambda$ satisfying
\begin{equation}
\xi_\lambda = \zeta_\lambda \circ \Pi \zeta_\lambda = - \zeta_\lambda
\circ \zeta_{\Pi \lambda}
\end{equation}
for all $\lambda \in \ob \A$.
%This implies that $\Pi$ is a superequivalence.
Moreover, we have that $\xi \Pi = \Pi \xi$ in $\Hom(\Pi^3,
\Pi)$.
\end{definition}

To specify the extra data needed to make
a supercategory into a $\Pi$-supercategory,
one just needs to give objects $\Pi \lambda$ 
and odd isomorphisms $\zeta_\lambda : \Pi
\lambda\stackrel{\sim}{\rightarrow} \lambda$ for each $\lambda \in
\ob \A$. The effect of $\Pi$ on a morphism $f:\lambda
\rightarrow \mu$ is uniquely determined 
by the requirement that 
$\zeta_\mu \circ \Pi f = (-1)^{|f|} f \circ \zeta_\lambda$.
It is then automatic that
$\zeta = (\zeta_\lambda):\Pi \stackrel{\sim}{\Rightarrow} I$ is
an odd supernatural isomorphism.

\begin{example}\label{psf}
Given superalgebras $A$ and $B$,
$A \lrSMod B$ 
is a $\Pi$-supercategory; hence, taking $A = B = \k$, 
so is $\SVEC$. 
To
specify $\Pi$ and $\zeta$, 
we just need to define an $(A,B)$-supermodule
$\Pi V$ 
and
 an odd isomorphism
$\zeta_V:\Pi V \stackrel{\sim}{\rightarrow} V$ 
for each
$(A,B)$-superbimodule $V$.
We take
$\Pi V$ to be
the same underlying vector space as $V$ viewed
as a superspace with the
opposite $\Z/2$-grading $(\Pi V)_\0 := V_\1$ and $(\Pi V)_\1 := V_\0$.
The superbimodule structure on $\Pi V$ is 
defined in terms of the original action by $a \cdot v \cdot b := (-1)^{|a|} avb$.
This ensures that the identity function on the underlying vector space
defines an odd superbimodule
isomorphism $\zeta_V: \Pi V \stackrel{\sim}{\rightarrow} V$.
Everything else is forced; for example,
for a morphism $f:V \rightarrow W$ we must have that $\Pi f : \Pi V \rightarrow
\Pi W$ is the function $(-1)^{|f|} f$; also,
$\xi_V:\Pi^2 V \stackrel{\sim}{\rightarrow} V$ is minus the identity.
\end{example}

Now we can explain the connection between supercategories and
$\Pi$-categories.
Let $\SCat$ be the category of \smallcat supercategories and
superfunctors as above.
Also let
$\piSCat$ be the category of \smallcat $\Pi$-supercategories and
superfunctors, and $\piCat$ be the category of \smallcat
$\Pi$-categories and $\Pi$-functors.
There are functors 
\begin{equation}\label{golf1}
\SCat \stackrel{(1)}{\longrightarrow} \piSCat \stackrel{(2)}{\longrightarrow} \piCat.
\end{equation}
The functor (1) is defined in Definition~\ref{pienv} below; it sends supercategory $\A$ to its {\em $\Pi$-envelope}
$\A_\pi$.
The functor (2) sends $\Pi$-supercategory
$(\A, \Pi, \zeta)$ to the {\em underlying
$\Pi$-category}
$(\underline{\A}, \underline{\Pi}, \xi)$, where $\underline{\A}$ and $\underline{\Pi}$ are as in
Definition~\ref{defsupercat}(v), 
and $\xi := \zeta \zeta$; it sends
superfunctor $F:(\A,\Pi_\A,\zeta_\A) \rightarrow (\B, \Pi_\B, \zeta_\B)$
to the $\Pi$-functor $(\underline{F}, \beta_{F})$, where 
$\beta_F := -\zeta_\B F (\zeta_\A)^{-1} :\Pi_\B F
\stackrel{\sim}{\Rightarrow} F\Pi_\A$.

\begin{theorem}\label{hop}
The functors just defined have the following
properties:
\begin{itemize}
\item
The functor (1) is left 2-adjoint to the forgetful functor $\nu:\piSCat
\rightarrow \SCat$ in the sense that there is a superequivalence
$\mathcal{H}om(\A,\nu \B) \rightarrow \mathcal{H}om(\A_\pi, \B)$
for every supercategory $\A$ and $\Pi$-supercategory $\B$.
\item
The functor (2) is an equivalence of categories.
\end{itemize}
\end{theorem}

\begin{definition}\label{pienv}
The {\em $\Pi$-envelope}
$\A_\pi$ of supercategory $\A$
is the $\Pi$-supercategory with objects $\{\Pi^a \lambda\:|\:\lambda \in \ob\A, a \in \Z/2\}$,
i.e. we double the objects in $\A$.  Morphisms are
defined from 
$$
\Hom_{\A_\pi}(\Pi^a \lambda, \Pi^b \mu) := \Pi^{a+b}
\Hom_{\A}(\lambda,\mu),
$$
where the $\Pi$ on the right hand side is the parity-switching functor on
$\SVEC$ from Example~\ref{psf}.
We denote the morphism $\Pi^a \lambda \rightarrow \Pi^b \mu$ in
$\A_\pi$ coming from a homogeneous morphism $f:\lambda \rightarrow \mu$
in $\A$ under this identification by $f_a^b$.
Thus, if $|f|$ denotes the parity of $f$ in $\A$, then $f_a^b:\Pi^a \lambda \rightarrow
\Pi^{b} \mu$ is of parity $a+b+|f|$ in $\A_\pi$.
Composition in $\A_\pi$ is induced by composition in
$\A$, so $f^{c}_b \circ g^b_a :=
(f \circ g)^c_{a}.$
To make $\A_\pi$ into a $\Pi$-supercategory, we set
$\Pi(\Pi^a \lambda) := \Pi^{a+\1} \lambda$ and 
$\zeta_{\Pi^a \lambda}:= (1_\lambda)^a_{a+\1}:\Pi^{a+\1}
\lambda \rightarrow \Pi^a \lambda$.
If $F:\A \rightarrow \B$ is a superfunctor, it extends to
$F_\pi:\A_\pi \rightarrow \B_\pi$ sending $\Pi^a \lambda \mapsto \Pi^a
(F\lambda)$
and $f_a^b \mapsto (F f)_a^b$.
\end{definition}

\begin{remark}\rm
In \cite{Man}, one finds already the notion of a {\em superadditive
  category}. In our language, this is an additive
$\Pi$-supercategory.
The {\em superadditive envelope} of a supercategory $\A$ may be
constructed by first taking the $\Pi$-envelope, then taking the usual
additive envelope after that.
\end{remark}

\vspace{2mm}
\noindent
1.4.
We can now introduce {\em
  monoidal $\Pi$-categories} and {\em monoidal
  $\Pi$-supercategories}. It is best to start with
monoidal $\Pi$-supercategories, since this definition is on the
surface. Then we'll recover the correct definition of monoidal
$\Pi$-category on passing to the underlying category.

\begin{definition}\label{thnder}
A {\em monoidal $\Pi$-supercategory} $(\A, \pi, \zeta)$ is a monoidal supercategory $\A$
with the additional data of a distinguished object $\pi$ and an odd isomorphism
$\zeta: \pi \stackrel{\sim}{\rightarrow} \unit$
from $\pi$ to the unit object $\unit$.
\end{definition}

Any monoidal $\Pi$-supercategory $(\A,\pi,\zeta)$ 
is a $\Pi$-supercategory in the sense of Definition~\ref{defscat}
with parity-switching functor $\Pi := \pi \otimes-:\A
\rightarrow \A$
and 
$\zeta_\lambda := l_\lambda \circ \zeta \otimes 1_\lambda:\Pi \lambda \stackrel{\sim}{\rightarrow} \lambda$.
One could also choose to define $\Pi$ to be the functor $-\otimes
\pi$, but that is isomorphic to our choice
because there is an even supernatural
isomorphism $\beta:\pi \otimes -\stackrel{\sim}{\Rightarrow} - \otimes \pi $
with $\beta_\lambda$ defined as the composite
\begin{equation*}
\pi \otimes \lambda
\stackrel{\zeta
  \otimes 1_\lambda}{\longrightarrow} 
\unit \otimes \lambda
\stackrel{l_\lambda}{\longrightarrow} 
\lambda
\stackrel{r_\lambda^{-1}}{\longrightarrow} \lambda \otimes \unit
\stackrel{1_\lambda \otimes \zeta^{-1}}{\longrightarrow} 
\lambda \otimes \pi.
\end{equation*}
We observe moreover that the pair $(\pi, \beta)$ is an object in the
{\em Drinfeld
center} of $\A$, i.e. we have that
\begin{align}
l_\pi \circ \beta_\unit
&=  r_\pi,\\
a_{\lambda,\mu,\pi}
\circ 
\beta_{\lambda \otimes \mu}
\circ  
a_{\pi,\lambda,\mu}
 &= 
(1_\lambda \otimes \beta_\mu) 
\circ  
a_{\lambda,\pi,\mu}
\circ 
(\beta_\lambda \otimes 1_\mu),
\end{align}
 for all objects $\lambda,\mu \in
\ob \A$.
Moreover, $\beta_\pi = - 1_{\pi\otimes\pi}$.
There is also an even isomorphism $\xi  := (l_\unit=r_{\unit}) \circ \zeta \otimes \zeta:\pi \otimes \pi \rightarrow \unit$
such that
\begin{equation}\label{bugg}
(1_\lambda \otimes \xi^{-1})
\circ 
r_\lambda^{-1} 
\circ
l_\lambda \circ 
 (\xi \otimes 1_\lambda)
= 
a_{\lambda,\pi,\pi}
\circ (\beta_\lambda \otimes 1_\pi)
\circ a_{\pi,\lambda,\pi}^{-1}
\circ 
 (1_\pi \otimes \beta_\lambda) 
\circ a_{\pi,\pi,\lambda}
\end{equation}
in $\Hom_{\A}((\pi \otimes \pi) \otimes \lambda, \lambda \otimes (\pi
\otimes \pi))$.

\begin{examples}\label{washing}
(i) We've already explained how $A \lrSMod A$ is both a monoidal
supercategory and a $\Pi$-supercategory.
In fact, it is a monoidal $\Pi$-supercategory with $\pi := \Pi A$
and $\zeta:\Pi A \stackrel{\sim}{\rightarrow} A$ being the identity function.
In particular, this makes $\SVEC$ into a monoidal $\Pi$-supercategory.

(ii) If $(\A, \Pi, \zeta)$ is any $\Pi$-supercategory, then
$(\mathcal{E}nd(\A), \Pi, \zeta)$ is a
strict monoidal $\Pi$-supercategory.
\end{examples}

\begin{definition}\label{green}
(i) 
A {\em monoidal $\Pi$-category} $(\A, \pi, \beta,\xi)$ is a
$\k$-linear monoidal category $\A$
plus the extra data of an object $(\pi, \beta)$ in its Drinfeld center
with $\beta_\pi = - 1_{\pi \otimes \pi}$, and an isomorphism
$\xi:\pi \otimes \pi \stackrel{\sim}{\rightarrow} \unit$ satisfying
(\ref{bugg}).

(ii) 
A {\em monoidal $\Pi$-functor} between 
monoidal $\Pi$-categories $(\A, \pi_\A, \beta_\A,
\xi_\A)$
and $(\B, \pi_\B, \beta_\B, \xi_\B)$ is a $\k$-linear monoidal
functor
$F:\A \rightarrow \B$ with its usual coherence maps $c$ and $i$,
plus an additional coherence map
$j:\pi_\B \stackrel{\sim}{\rightarrow} F\pi_\A$ which is an isomorphism
compatible with the $\beta$'s 
and the $\xi$'s in the sense that
\begin{align*}
F
(\beta_\A)_\lambda
\circ
c_{\pi_\A,\lambda} \circ (j \otimes 1_{F\lambda}) 
&=
c_{\lambda, \pi_\A} \circ 
(1_{F\lambda} \otimes j)
\circ
(\beta_\B)_{F\lambda},\\
i \circ \xi_\B &= F \xi_A \circ c_{\pi_\A,\pi_\A} \circ (j \otimes j),
\end{align*}
in $\Hom(\pi_\B \otimes F\lambda, F(\lambda\otimes \pi_\A))$ and
$\Hom(\pi_\B\otimes\pi_\B,F \unit_\A)$, respectively.

(iii) A {\em monoidal $\Pi$-natural transformation}
$x:F \Rightarrow G$ between monoidal $\Pi$-functors $F,G:\A
\rightarrow \B$ is a monoidal natural transformation as usual, such
that 
$x_{\pi_\A} \circ j_F = j_G$
in $\Hom_\B(\pi_\B, G \pi_\A)$.
%THE FOLLOWING FOLLOWS FROM NATURALITY OF \beta_\B:
%x_\lambda \otimes 1_{\pi_\B} \circ (\beta_\B)_{F\lambda}=
%(\beta_\B)_{G\lambda}\circ 1_{\pi_\B}\otimes x_\lambda in 
%\Hom_\B(\pi_\B\otimes F\lambda,G\lambda \otimes \pi_\B).
\end{definition}

There are categories $\SMon$, $\piSMon$ and $\piMon$ consisting of all
\smallcat
monoidal supercategories, monoidal $\Pi$-supercategories and monoidal
$\Pi$-categories, respectively.
Morphisms in $\SMon$ and $\piSMon$ are monoidal superfunctors
as in Definition~\ref{dash}(ii). Morphisms in $\piMon$ are monoidal
  $\Pi$-functors in the sense of Definition~\ref{green}(ii).
Now, just like in (\ref{golf1}), there are functors
\begin{equation}\label{golf2}
\SMon \stackrel{(1)}{\longrightarrow} \piSMon \stackrel{(2)}{\longrightarrow} \piMon.
\end{equation}
The functor (1) is defined by the $\Pi$-envelope construction
explained in
Definition~\ref{env2} below.
The functor (2) sends 
monoidal $\Pi$-supercategory $(\A, \pi, \zeta)$
to the underlying category $\underline{\A}$
with the obvious monoidal structure,
 made into
a monoidal $\Pi$-category $(\underline{\A},
\pi, \beta, \xi)$ as explained before Definition~\ref{green}.
It sends a monoidal superfunctor $F$ between monoidal
$\Pi$-supercategories $\A$ and $\B$ to
$\underline{F}:\underline{\A} \rightarrow
\underline{\B}$,
made into a monoidal $\Pi$-functor by setting
$j := 
(F \zeta_\A)^{-1} \circ i \circ \zeta_\B :\pi_\B
\stackrel{\sim}{\rightarrow} F \pi_\A$.

\begin{theorem}\label{bop}
The functors just defined satisfy
analogous properties to Theorem~\ref{hop}: (1) is left 2-adjoint to the
forgetful functor
and (2) is an equivalence.
\end{theorem}

\begin{definition}\label{env2}
The {\em $\Pi$-envelope} of a monoidal supercategory $\A$
is the monoidal $\Pi$-supercategory
$(\A_\pi, \pi, \zeta)$ 
where $\A_\pi$ is as in Definition~\ref{pienv},
$\pi := \Pi^\1 \unit$, $\zeta := (1_\unit)^\0_\1$, and tensor products
of objects and morphisms are defined from
\begin{align*}
(\Pi^a \lambda) \otimes (\Pi^b \mu) &:= \Pi^{a+b} (\lambda \otimes \mu),\\
f^b_a \otimes g^d_c &:= (-1)^{a|g|+|f|d+ad+ac} (f \otimes g)^{b+d}_{a+c}.
\end{align*}
The unit object of $\A_\pi$ is $\Pi^\0 \unit$. The coherence maps $a,
l$ and $r$ extend to $\A_\pi$ in an obvious way.
Also if $F:\A \rightarrow \B$ is a monoidal superfunctor then 
the superfunctor $F_\pi:\A_\pi \rightarrow \B_\pi$ from
Definition~\ref{pienv} is naturally monoidal too.
\end{definition}

In the strict case, one can work with $\A_\pi$ diagrammatically as
follows. For $f$ as in (\ref{jazzy}), 
we represent $f^b_a \in
\Hom_{\A_\pi}(\Pi^a \lambda, \Pi^b \mu)$ by the diagram
$$
\mathord{
\begin{tikzpicture}[baseline = 0]
	\draw[-,thick,darkred] (0.08,-.4) to (0.08,-.13);
	\draw[-,thick,darkred] (0.08,.4) to (0.08,.13);
      \draw[thick,darkred] (0.08,0) circle (4pt);
   \node at (0.08,0) {\color{darkred}$\scriptstyle{f}$};
   \node at (0.08,-.53) {$\scriptstyle{\lambda}$};
   \node at (0.08,.53) {$\scriptstyle{\mu}$};
\draw[-,thin,red](.4,-.4) to (-.24,-.4);
\draw[-,thin,red](.4,.4) to (-.24,.4);
\node at (.5,.4) {$\color{red}\scriptstyle b$};
\node at (.5,-.4) {$\color{red}\scriptstyle a$};
\end{tikzpicture}
}
$$
Then the rules for horizontal and vertical composition become:
$$
\mathord{
\begin{tikzpicture}[baseline =-1.3]
	\draw[-,thick,darkred] (0.08,-.4) to (0.08,-.13);
	\draw[-,thick,darkred] (0.08,.4) to (0.08,.13);
      \draw[thick,darkred] (0.08,0) circle (4pt);
   \node at (0.08,0) {\color{darkred}$\scriptstyle{f}$};
\draw[-,thin,red](.4,-.4) to (-.24,-.4);
\draw[-,thin,red](.4,.4) to (-.24,.4);
\node at (.5,.4) {$\color{red}\scriptstyle b$};
\node at (.5,-.4) {$\color{red}\scriptstyle a$};
\end{tikzpicture}
}
\,\otimes
\mathord{
\begin{tikzpicture}[baseline = -1.3]
	\draw[-,thick,darkred] (0.08,-.4) to (0.08,-.13);
	\draw[-,thick,darkred] (0.08,.4) to (0.08,.13);
      \draw[thick,darkred] (0.08,0) circle (4pt);
   \node at (0.08,0) {\color{darkred}$\scriptstyle{g}$};
\draw[-,thin,red](.4,-.4) to (-.24,-.4);
\draw[-,thin,red](.4,.4) to (-.24,.4);
\node at (.5,.4) {$\color{red}\scriptstyle d$};
\node at (.5,-.4) {$\color{red}\scriptstyle c$};
\end{tikzpicture}
}
=
(-1)^{a|g|+|f|d+ad+ac}
\mathord{
\begin{tikzpicture}[baseline = -2]
	\draw[-,thick,darkred] (0.08,-.4) to (0.08,-.13);
	\draw[-,thick,darkred] (0.08,.4) to (0.08,.13);
      \draw[thick,darkred] (0.08,0) circle (4pt);
   \node at (0.08,0) {\color{darkred}$\scriptstyle{g}$};
	\draw[-,thick,darkred] (-.8,-.4) to (-.8,-.13);
	\draw[-,thick,darkred] (-.8,.4) to (-.8,.13);
      \draw[thick,darkred] (-.8,0) circle (4pt);
   \node at (-.8,0) {\color{darkred}$\scriptstyle{f}$};
\draw[-,thin,red](.45,-.4) to (-1.1,-.4);
\draw[-,thin,red](.45,.4) to (-1.1,.4);
\node at (.74,.4) {$\color{red}\scriptstyle b+d$};
\node at (.74,-.4) {$\color{red}\scriptstyle a+c$};
\end{tikzpicture}
},
\qquad\!
\mathord{
\begin{tikzpicture}[baseline = -1.3]
	\draw[-,thick,darkred] (0.08,-.4) to (0.08,-.13);
	\draw[-,thick,darkred] (0.08,.4) to (0.08,.13);
      \draw[thick,darkred] (0.08,0) circle (4pt);
   \node at (0.08,0) {\color{darkred}$\scriptstyle{f}$};
\draw[-,thin,red](.4,-.4) to (-.24,-.4);
\draw[-,thin,red](.4,.4) to (-.24,.4);
\node at (.5,.4) {$\color{red}\scriptstyle c$};
\node at (.5,-.4) {$\color{red}\scriptstyle b$};
\end{tikzpicture}
}
\:\circ
\mathord{
\begin{tikzpicture}[baseline =-1.3]
	\draw[-,thick,darkred] (0.08,-.4) to (0.08,-.13);
	\draw[-,thick,darkred] (0.08,.4) to (0.08,.13);
      \draw[thick,darkred] (0.08,0) circle (4pt);
   \node at (0.08,0) {\color{darkred}$\scriptstyle{h}$};
\draw[-,thin,red](.4,-.4) to (-.24,-.4);
\draw[-,thin,red](.4,.4) to (-.24,.4);
\node at (.5,.4) {$\color{red}\scriptstyle b$};
\node at (.5,-.4) {$\color{red}\scriptstyle a$};
\end{tikzpicture}
}
=\!
\mathord{
\begin{tikzpicture}[baseline = 8]
	\draw[-,thick,darkred] (0.08,-.4) to (0.08,-.13);
	\draw[-,thick,darkred] (0.08,.57) to (0.08,.13);
	\draw[-,thick,darkred] (0.08,.83) to (0.08,1.1);
      \draw[thick,darkred] (0.08,0) circle (4pt);
      \draw[thick,darkred] (0.08,.7) circle (4pt);
   \node at (0.08,0) {\color{darkred}$\scriptstyle{h}$};
   \node at (0.08,.71) {\color{darkred}$\scriptstyle{f}$};
\draw[-,thin,red](.4,-.4) to (-.24,-.4);
\draw[-,thin,red](.4,1.1) to (-.24,1.1);
\node at (.5,1.1) {$\color{red}\scriptstyle c$};
\node at (.5,-.4) {$\color{red}\scriptstyle a$};
\end{tikzpicture}
}.
$$
In order to appreciate the need for the sign in this definition of
horizontal composition, the reader might want to verify the super
interchange law in $\A_\pi$.

\vspace{2mm}
\noindent
1.5.
Let us make a few remarks about Grothendieck
groups/rings.
Recall for a category $\A$ that its {\em additive Karoubi envelope}
$\Kar(\A)$ is the
idempotent completion of the additive envelope of $\A$.
The {\em Grothendieck group} 
$K_0(\Kar(\A))$ is the
Abelian group generated by isomorphism classes of objects of $\Kar(\A)$,
subject to the relations $[V] + [W] = [V  \oplus W]$.
 In case $\A$ is a monoidal category, the monoidal structure on $\A$
extends canonically to $\Kar(\A)$, hence we get a ring
structure on 
$K_0(\Kar(\A))$ with $[V] \cdot [W] = [V \otimes W]$.

For a supercategory $\A$, we 
propose 
that 
the role of additive Karoubi envelope should be played by
the $\Pi$-category
$\SKar(\A) := \Kar(\underline{\A}_\pi)$, i.e. one first passes to the
$\Pi$-envelope, then to the underlying category, and then one
takes additive Karoubi envelope as usual.
The Grothendieck group $K_0(\SKar(\A))$
comes equipped with a distinguished involution $\pi$ defined from
$\pi([V]) := [\Pi V]$, making it into a module over the ring
$$
\Z^\pi := \Z[\pi] / (\pi^2-1).
$$
In case $\A$ is a monoidal supercategory,
$\SKar(\A)$ is a monoidal
$\Pi$-category. The tensor product induces a multiplication on
$K_0(\SKar(\A))$, making it into a $\Z^\pi$-algebra.

\begin{example}\label{day}
(i) Suppose $A$ is a superalgebra viewed as a supercategory $\A$ with
one object. Then $\SKar(\A)$
is equivalent to the category of finitely generated projective
$A$-supermodules and 
even $A$-supermodule homomorphisms.
Hence, $K_0(\SKar(\A))$ is the usual split Grothendieck group of
the superalgebra $A$.
 
(ii)
Recall that $\mathcal{I}$, the unit object of the monoidal category $\SCat$,
is a supercategory with one object whose endomorphism superalgebra is
$\k$. There is a unique way to define a tensor product making $\mathcal{I}$ into
a strict monoidal supercategory.
Its super Karoubi envelope
$\SKar(\mathcal{I})$ is monoidally equivalent to $\fdSVec$.
Hence, it is
a semisimple Abelian category with just two isomorphism classes of irreducible objects
represented
by  $\k$ and $\Pi \k$, and
$K_0(\SKar(\mathcal{I})) \cong K_0(\fdSVec) \cong \Z^\pi$.

(iii)
Here is an example which may be of independent interest.
For $\delta \in \k$,
the {\em odd Temperley-Lieb supercategory}
is the strict monoidal supercategory
$\mathcal{STL}(\delta)$
with one generating object 
$\color{darkpurple}{\cdot}$
and two {\em odd} generating morphisms
$\mathord{
\begin{tikzpicture}[baseline = 0]
	\draw[-,thick,darkpurple] (0.3,0.25) to[out=-90, in=0] (0.1,-0.05);
	\draw[-,thick,darkpurple] (0.1,-0.05) to[out = 180, in = -90] (-0.1,0.25);
\end{tikzpicture}
}: \unit\rightarrow
{\color{darkpurple}\cdot}\,\otimes\,{\color{darkpurple}\cdot}\,$
and 
$\mathord{
\begin{tikzpicture}[baseline = 0]
	\draw[-,thick,darkpurple] (0.3,-0.1) to[out=90, in=0] (0.1,0.2);
	\draw[-,thick,darkpurple] (0.1,0.2) to[out = 180, in = 90] (-0.1,-.1);
\end{tikzpicture}
}:
{\color{darkpurple}\cdot}\,\otimes\,{\color{darkpurple}\cdot}\,
\rightarrow \unit$, 
subject to the following relations:
\begin{align*}
\mathord{
\begin{tikzpicture}[baseline = 0]
  \draw[-,thick,darkpurple] (0.2,0) to (0.2,.5);
	\draw[-,thick,darkpurple] (0.2,0) to[out=-90, in=0] (0,-0.35);
	\draw[-,thick,darkpurple] (0,-0.35) to[out = 180, in = -90] (-0.2,0);
	\draw[-,thick,darkpurple] (-0.2,0) to[out=90, in=0] (-0.4,0.35);
	\draw[-,thick,darkpurple] (-0.4,0.35) to[out = 180, in =90] (-0.6,0);
  \draw[-,thick,darkpurple] (-0.6,0) to (-0.6,-.5);
\end{tikzpicture}
}
\,&=\,
\mathord{\begin{tikzpicture}[baseline=0]
  \draw[-,thick,darkpurple] (0,-0.4) to (0,.4);
\end{tikzpicture}
}\:,
\qquad
\mathord{
\begin{tikzpicture}[baseline = 0]
  \draw[-,thick,darkpurple] (0.2,0) to (0.2,-.5);
	\draw[-,thick,darkpurple] (0.2,0) to[out=90, in=0] (0,0.35);
	\draw[-,thick,darkpurple] (0,0.35) to[out = 180, in = 90] (-0.2,0);
	\draw[-,thick,darkpurple] (-0.2,0) to[out=-90, in=0] (-0.4,-0.35);
	\draw[-,thick,darkpurple] (-0.4,-0.35) to[out = 180, in =-90] (-0.6,0);
  \draw[-,thick,darkpurple] (-0.6,0) to (-0.6,.5);
\end{tikzpicture}
}\,
=\,
-\:\,\mathord{\begin{tikzpicture}[baseline=0]
  \draw[-,thick,darkpurple] (0,-0.4) to (0,.4);
\end{tikzpicture}
}\:,\qquad
\mathord{
\begin{tikzpicture}[baseline = 0]
	\draw[-,thick,darkpurple] (0,-.25) to[out=180,in=-90] (-0.28,0.05);
	\draw[-,thick,darkpurple] (0,-.25) to[out=0,in=-90] (0.28,0.05);
	\draw[-,thick,darkpurple] (0,.35) to[out=180,in=90] (-0.28,0.05);
	\draw[-,thick,darkpurple] (0,.35) to[out=0,in=90] (0.28,0.05);
\end{tikzpicture}
}\,
=\delta.
\end{align*}
The following theorem will be proved in the appendix.

\begin{theorem}\label{stl}
Assume that $\delta = -(q-q^{-1})$ for $q \in \k^\times$ that is
not a root of unity.
Then $\SKar(\mathcal{STL}(\delta))$ is
a semisimple Abelian category.
Moreover, as a based ring with canonical basis coming from the
isomorphism classes of irreducible objects,
$K_0(\SKar(\mathcal{STL}(\delta)))$ is isomorphic to the subring of $\Z^\pi[x,x^{-1}]$
spanned over $\Z$ by $\left\{[n+1]_{x,\pi}, \pi [n+1]_{x,\pi}\:\big|\:n \in \N\right\}$,
where
\begin{equation}\label{qint}
[n+1]_{x,\pi}:= 
%\frac{x^n - (\pi x)^{-n}}{x-\pi x^{-1}}=
x^{n} + \pi x^{n-2}+\cdots+\pi^{n} x^{-n}.
\end{equation}
\end{theorem}

When $\k$ is of characteristic zero, we will explain this result by constructing a monoidal equivalence between 
$\SKar(\mathcal{STL}(\delta))$ and
the category of 
finite-dimensional representations of the quantum
superalgebra $U_q(\mathfrak{osp}_{1|2})$ as defined by Clark and Wang
\cite{CW}. 
We note that
\begin{equation}\label{qint2}
[n+1]_{x,\pi} [m+1]_{x,\pi} = 
\sum_{r=0}^{\min(m,n)} \pi^r [n+m-2r+1]_{x,\pi},
\end{equation}
which may be interpreted as the analog of Clebsch-Gordon for
$U_q(\mathfrak{osp}_{1|2})$.
Also
\begin{equation}
\sum_{n=0}^\infty [n]_{x,\pi} t^n = \frac{1}{1 - [2]_{x,\pi} t + \pi t^2},
\end{equation}
which is a $\pi$-deformed version of the generating function for Chebyshev
polynomials of the second kind.
It follows that
$K_0(\SKar(\mathcal{STL}(\delta)))$ is a polynomial algebra over
$\Z^\pi$ generated by $[2]_{x,\pi}$, which is the isomorphism class of
the generating object $\color{darkpurple}{\cdot}$.
\end{example}

\vspace{2mm}
\noindent
1.7.
In the remainder of the article, we will work in the more general
setting of 2-categories.
Recalling that a monoidal category is essentially the same as a
2-category with one object, the reader should have no trouble
recovering the definitions made in this introduction from the more
general ones formulated later on.

In Section 2, we will discuss {\em 2-supercategories}, which (in the
strict case) are
categories enriched in $\SCat$; the basic
example is the 2-supercategory of supercategories, superfunctors and
supernatural transformations.
Then in Section 3, we introduce {\em $\Pi$-2-supercategories};
the basic example is the $\Pi$-2-supercategory of
$\Pi$-supercategories, superfunctors and supernatural transformations.
Section 4 develops the appropriate generalization of the notion of
{\em $\Pi$-envelope} to 2-supercategories, 
in particular establishing the properties of the
functors (1) above.
In Section 5, we discuss {\em $\Pi$-2-categories}; the
basic example is the $\Pi$-2-category of $\Pi$-categories,
$\Pi$-functors and $\Pi$-natural transformations.
Then we prove that the functors (2) above are equivalences; more
generally, we show that the categories of
$\Pi$-2-categories and $\Pi$-2-supercategories
are equivalent.

The approach to $\Z/2$-graded categories developed by this point can
also be applied in almost exactly the same way to $\Z$-graded
categories.
We give a brief account of this in the final Section
6. Actually, we will combine the two gradings into a single $\Z \oplus
\Z/2$-grading, and develop a theory of {\em graded supercategories}. 
Although we won't discuss it further here, there are two natural ways to suppress the $\Z/2$-grading
(thereby leaving the domain of superalgebra): 
one can either view $\Z$-gradings as $\Z\oplus \Z/2$-gradings with
the $\Z/2$-grading being trivial, i.e. concentrated in parity $\0$; or one
can view $\Z$-gradings as $\Z \oplus \Z/2$-gradings with
the $\Z/2$-grading being induced by the $\Z$-grading, i.e.
all elements of degree $n\in\Z$ are of parity $n\pmod{2}$.
The first of these variations is already extensively used in representation
theory, e.g. see the last paragraph of \cite[$\S$2.2.1]{Rou} or \cite[$\S$5.2]{BLW}.

We would like to say finally
that many of the general definitions in this article can be found in some
equivalent form in many
places in the literature. We were influenced especially
by the work of Kang, Kashiwara and Oh in \cite[Section 7]{KKO2}; see also
\cite[Section 2]{EL}.
Our choice of terminology is different. We include here
a brief dictionary
for readers familiar with \cite{KKO2} and \cite{EL}; note also that in
\cite{KKO2} additivity is assumed
from the outset.
$$
\begin{array}{l|l}
\text{Our language}&\text{Language of \cite{KKO2, EL}}\\
\hline
\text{Supercategory}&\text{$1$-supercategory \cite[Def. 7.7]{KKO2}}\\
\text{Superfunctor}&\text{Superfunctor \cite[Def. 7.7]{KKO2}}\\
\text{Supernatural transformation}&\text{Even and odd morphisms
  \cite[Def. 7.8]{KKO2}}\\
\text{2-supercategory}&\text{$2$-supercategory
  \cite[Def. 7.12]{KKO2}}\\
\text{$\Pi$-category}&\text{Supercategory \cite[Def. 7.1]{KKO2},
  \!\cite[Def. 2.13]{EL}}\\
\text{$\Pi$-functor}&\text{Superfunctor \cite[Def. 7.1]{KKO2},
  \!\cite[Def. 2.13]{EL}}\\
\text{$\Pi$-natural transformation}&\text{Supernatural transformation
  \cite[Def. 2.16]{EL}}\\
\text{$\Pi$-2-category}&\text{Super-2-category \cite[Def. 2.17]{EL}}
\end{array}
$$
There is a similar linguistic clash in our development of the graded
theory in Section 6:
by a {\em graded category}, we mean a category enriched in graded
vector spaces. It is more common in the literature for a
graded category to mean a category equipped with a distinguished 
autoequivalence.
When working with the latter structure, we will denote this autoequivalence by $Q$, and call
it a {\em $Q$-category}.

\vspace{2mm}
\noindent
{\em Acknowledgements.}
The first author would like to thank Jon Kujawa for convincing him to
take categories enriched in super vector spaces seriously in the first place. 
We also benefitted greatly from conversations with Victor Ostrik and Ben Elias.

\section{Supercategories}

In the main body of the article, $\k$ will denote some fixed
commutative ground ring.
By {\em superspace}, we mean now a $\Z/2$-graded $\k$-module $V = V_\0 \oplus
V_\1$; as usual when working over a commutative ring, we make no
distinction between left modules and right modules, indeed, we'll often view $\k$-modules
as $(\k,\k)$-bimodules whose left and right actions are 
related by $cv = vc$.
By a linear map, we mean a $\k$-module homomorphism.
Viewing $\k$ as a superalgebra concentrated in even parity, these are
the same as $\k$-supermodules and $\k$-supermodule homomorphisms\footnote{In Sections 2--4, one can actually work even more generally over any 
commutative superalgebra $\k = \k_\0\oplus\k_\1$, 
interpreting a superspace as a $(\k,\k)$-superbimodule
whose left and right actions are related by $cv =
(-1)^{|c||v|}vc$.}.

We have the $\Pi$-supercategory $\SVEC$ of all superspaces\footnote{One should be careful about
set-theoretic issues here 
by fixing a 
Grothendieck universe and taking only {\em small} superspaces. 
We won't be doing anything high enough for this to cause
difficulties, so will ignore issues of this nature.}
 and linear
maps defined just like in the introduction.
The underlying category
$\SVec$ consisting of superspaces and even linear maps
is a symmetric monoidal category with braiding
defined as in the introduction.

Recall also the definitions of {\em superfunctor} and
{\em supernatural transformation} from Definition~\ref{defsupercat}.
Let $\SCat$ be the category of all \smallcat supercategories and
superfunctors. 
We make it into a monoidal category with tensor functor denoted
$\boxtimes$ as explained after Example~\ref{snuggly}.

\begin{definition}\label{defsuper2cat}
A {\em strict 2-supercategory} is a category enriched in the
monoidal category
$\SCat$ just defined.
Thus, for objects $\lambda,\mu$ in a strict 2-supercategory
$\AA$,
there is given a \smallcat
supercategory $\mathcal{H}om_{\AA}(\lambda,\mu)$ of
morphisms from $\lambda$ to $\mu$, whose objects $F,G$
are the {\em 1-morphisms} of $\AA$, and whose morphisms
$x:F \rightarrow G$ are
the {\em 2-morphisms} of $\AA$.
We use the shorthand $\Hom_{\AA}(F,G)$ for the superspace
$\Hom_{\mathcal{H}om_{\AA}(\lambda,\mu)}(F,G)$ of all such
2-morphisms.
\end{definition}

The string calculus explained for monoidal supercategories in the
introduction
can also be used for strict 2-supercategories:
given $1$-morphisms $F, G:\lambda \rightarrow \mu$,
one represents a 2-morphism $x:F\Rightarrow G$
by the picture
$$
\mathord{
\begin{tikzpicture}[baseline = 0]
	\draw[-,thick,darkred] (0.08,-.4) to (0.08,-.13);
	\draw[-,thick,darkred] (0.08,.4) to (0.08,.13);
      \draw[thick,darkred] (0.08,0) circle (4pt);
   \node at (0.08,0) {\color{darkred}$\scriptstyle{x}$};
   \node at (0.08,-.53) {$\scriptstyle{F}$};
   \node at (0.52,0) {$\scriptstyle{\lambda}.$};
   \node at (-0.32,0) {$\scriptstyle{\mu}$};
   \node at (0.08,.53) {$\scriptstyle{G}$};
\end{tikzpicture}
}
$$
The composition $y \circ x$ of $x$ with another 2-morphism
$y \in \Hom_{\AA}(G,H)$ is obtained by vertically stacking
pictures:
$$
\mathord{
\begin{tikzpicture}[baseline = 0]
	\draw[-,thick,darkred] (0.08,-.4) to (0.08,-.13);
	\draw[-,thick,darkred] (0.08,.4) to (0.08,.13);
	\draw[-,thick,darkred] (0.08,.77) to (0.08,.6);
	\draw[-,thick,darkred] (0.08,1.37) to (0.08,1.03);
      \draw[thick,darkred] (0.08,0) circle (4pt);
      \draw[thick,darkred] (0.08,.9) circle (4pt);
   \node at (0.08,0) {\color{darkred}$\scriptstyle{x}$};
   \node at (0.08,.91) {\color{darkred}$\scriptstyle{y}$};
   \node at (0.08,-.53) {$\scriptstyle{F}$};
   \node at (0.62,.5) {$\scriptstyle{\lambda}.$};
   \node at (-0.42,0.5) {$\scriptstyle{\mu}$};
   \node at (0.08,1.5) {$\scriptstyle{H}$};
   \node at (0.08,.5) {$\scriptstyle{G}$};
\end{tikzpicture}
}
$$
The composition law in $\AA$ gives a 
coherent family of superfunctors
$$
T_{\nu,\mu,\lambda}:\mathcal{H}om_{\AA}(\mu,\nu) \boxtimes
\mathcal{H}om_{\AA}(\lambda,\mu) \rightarrow \mathcal{H}om_{\AA}(\lambda,\nu)
$$ 
for all objects $\lambda,\mu,\nu \in \AA$. Given 2-morphisms 
$x:F \rightarrow H,
y:G \rightarrow K$ between 1-morphisms $F, H:\lambda \rightarrow\mu,
G,K:\mu \rightarrow\nu$,
we denote 
$T_{\nu,\mu,\lambda}(y\otimes x):T_{\nu,\mu,\lambda}(G, F) \rightarrow T_{\nu,\mu,\lambda}(K, H)$
simply by $yx:GF\rightarrow KH$,
and represent it by horizontally stacking pictures:
$$
\mathord{
\begin{tikzpicture}[baseline = 0]
	\draw[-,thick,darkred] (0.08,-.4) to (0.08,-.13);
	\draw[-,thick,darkred] (0.08,.4) to (0.08,.13);
      \draw[thick,darkred] (0.08,0) circle (4pt);
   \node at (0.08,0) {\color{darkred}$\scriptstyle{x}$};
   \node at (0.08,-.53) {$\scriptstyle{F}$};
   \node at (0.58,0) {$\scriptstyle{\lambda}.$};
   \node at (-0.37,0) {$\scriptstyle{\mu}$};
   \node at (0.08,.53) {$\scriptstyle{H}$};
	\draw[-,thick,darkred] (-.8,-.4) to (-.8,-.13);
	\draw[-,thick,darkred] (-.8,.4) to (-.8,.13);
      \draw[thick,darkred] (-.8,0) circle (4pt);
   \node at (-.8,0) {\color{darkred}$\scriptstyle{y}$};
   \node at (-.8,-.53) {$\scriptstyle{G}$};
   \node at (-1.22,0) {$\scriptstyle{\nu}$};
   \node at (-.8,.53) {$\scriptstyle{K}$};
\end{tikzpicture}
}
$$
When confusion seems unlikely, we will use the same notation for a $1$-morphism $F$
as for its identity $2$-morphism. 
With this convention, 
we have that $yH \circ Gx = yx = (-1)^{|x||y|} Kx \circ yF$, or in pictures:
$$
\mathord{
\begin{tikzpicture}[baseline = 0]
   \node at (0.08,-.53) {$\scriptstyle{F}$};
   \node at (0.58,0) {$\scriptstyle{\lambda}$};
   \node at (-0.37,0) {$\scriptstyle{\mu}$};
   \node at (0.08,.53) {$\scriptstyle{H}$};
   \node at (-.8,-.53) {$\scriptstyle{G}$};
   \node at (-1.22,0) {$\scriptstyle{\nu}$};
   \node at (-.8,.53) {$\scriptstyle{K}$};
	\draw[-,thick,darkred] (0.08,-.4) to (0.08,-.23);
	\draw[-,thick,darkred] (0.08,.4) to (0.08,.03);
      \draw[thick,darkred] (0.08,-0.1) circle (4pt);
   \node at (0.08,-0.1) {\color{darkred}$\scriptstyle{x}$};
	\draw[-,thick,darkred] (-.8,-.4) to (-.8,-.03);
	\draw[-,thick,darkred] (-.8,.4) to (-.8,.23);
      \draw[thick,darkred] (-.8,0.1) circle (4pt);
   \node at (-.8,.1) {\color{darkred}$\scriptstyle{y}$};
\end{tikzpicture}
}
\quad=\quad
\mathord{
\begin{tikzpicture}[baseline = 0]
   \node at (0.08,-.53) {$\scriptstyle{F}$};
   \node at (0.58,0) {$\scriptstyle{\lambda}$};
   \node at (-0.37,0) {$\scriptstyle{\mu}$};
   \node at (0.08,.53) {$\scriptstyle{H}$};
   \node at (-.8,-.53) {$\scriptstyle{G}$};
   \node at (-1.22,0) {$\scriptstyle{\nu}$};
   \node at (-.8,.53) {$\scriptstyle{K}$};
	\draw[-,thick,darkred] (0.08,-.4) to (0.08,-.13);
	\draw[-,thick,darkred] (0.08,.4) to (0.08,.13);
      \draw[thick,darkred] (0.08,0) circle (4pt);
   \node at (0.08,0) {\color{darkred}$\scriptstyle{x}$};
	\draw[-,thick,darkred] (-.8,-.4) to (-.8,-.13);
	\draw[-,thick,darkred] (-.8,.4) to (-.8,.13);
      \draw[thick,darkred] (-.8,0) circle (4pt);
   \node at (-.8,0) {\color{darkred}$\scriptstyle{y}$};
\end{tikzpicture}
}
\quad=\quad
(-1)^{|x||y|}\:
\mathord{
\begin{tikzpicture}[baseline = 0]
   \node at (0.08,-.53) {$\scriptstyle{F}$};
   \node at (0.58,0) {$\scriptstyle{\lambda}$};
   \node at (-0.37,0) {$\scriptstyle{\mu}$};
   \node at (0.08,.53) {$\scriptstyle{H}$};
   \node at (-.8,-.53) {$\scriptstyle{G}$};
   \node at (-1.22,0) {$\scriptstyle{\nu}$};
   \node at (-.8,.53) {$\scriptstyle{K}$};
	\draw[-,thick,darkred] (0.08,-.4) to (0.08,-.03);
	\draw[-,thick,darkred] (0.08,.4) to (0.08,.23);
      \draw[thick,darkred] (0.08,0.1) circle (4pt);
   \node at (0.08,0.1) {\color{darkred}$\scriptstyle{x}$};
	\draw[-,thick,darkred] (-.8,-.4) to (-.8,-.23);
	\draw[-,thick,darkred] (-.8,.4) to (-.8,.03);
      \draw[thick,darkred] (-.8,-0.1) circle (4pt);
   \node at (-.8,-.1) {\color{darkred}$\scriptstyle{y}$};
\end{tikzpicture}
}.
$$
This identity is a special case of the {\em super interchange law} in a
strict 2-supercategory, which is proved by the following calculation:
\begin{align*}
(v u) \circ (y x)
&= T_{\nu,\mu,\lambda}(v \otimes u) \circ T_{\nu,\mu,\lambda}(y\otimes
x)
= T_{\nu,\mu,\lambda}((v \otimes u) \circ (y \otimes x))\\
&= (-1)^{|u||y|} T_{\nu,\mu,\lambda}((v \circ y) \otimes (u \circ
x))
=
 (-1)^{|u||y|} (v \circ y) (u \circ
x).
\end{align*}
The presence of the sign 
here means that
a strict 2-supercategory is {\em not} a 2-category in the
usual sense.

For example, we can make $\SCat$ into a strict 2-supercategory $\SCAT$ by
declaring that its morphism categories
are the
supercategories $\mathcal{H}om(\A, \B)$
consisting of all superfunctors from $\A$ to $\B$,
with morphisms being all supernatural transformations.
The horizontal composition $GF$ of two superfunctors $F:\mathcal{A}
\rightarrow \mathcal{B}$ and $G:\B \rightarrow \C$ is
defined by
$G F := G \circ F$. The horizontal composition $yx:GF \Rightarrow KH$
of supernatural transformations
$x:F \Rightarrow H$ and $y:G \Rightarrow K$ 
is given by $(y x)_\lambda := y_{H\lambda} \circ Gx_{\lambda}$ for each object $\lambda$ of $\A$.
We leave it to the reader to verify that the super interchange law
holds; this works because of the signs built into the definition of
supernatural transformation.

So far, we have only defined the notion of {\em strict}
2-supercategory.
There is also a ``weak'' notion, which we call simply {\em
  2-supercategory}, 
in which the horizontal composition is
only assumed to be associative and unital up to some even supernatural
isomorphisms.
The following are the superizations of the definitions in
the purely even setting (e.g. see the definition of bicategory in
\cite{Leinster}, or \cite[$\S$2.2.2]{Rou}), replacing the usual Cartesian product $\times$ of categories with
the product $\boxtimes$.

\begin{definition}\label{ms}
(i)
A {\em 2-supercategory} $\AA$ consists of:
\begin{itemize}
\item A set of objects $\ob \AA$.
\item
A supercategory $\mathcal{H}om_{\AA}(\mu,\lambda)$ for each
$\lambda,\mu \in \ob\AA$, whose objects and morphisms are called
1-morphisms and 2-morphisms, respectively.
We refer to the composition of 2-morphisms in these supercategories as
{\em vertical composition}.
\item
A family of $1$-morphisms
$\unit_\lambda:\lambda\rightarrow \lambda$ for each $\lambda \in
\ob\AA$.
\item Superfunctors
$T_{\nu,\mu,\lambda}:\mathcal{H}om_{\AA}(\mu,\nu) \boxtimes
\mathcal{H}om_{\AA}(\lambda,\mu) \rightarrow \mathcal{H}om_{\AA}(\lambda,\nu)$ 
for all $\lambda,\mu,\nu \in \ob\AA$. 
We usually denote $T_{\nu,\mu,\lambda}$ simply by $-\:-$, and call it
{\em horizontal composition}.
\item
Even supernatural
isomorphisms
$a:(-\:-)\:- \stackrel{\sim}{\Rightarrow}
-\:(-\:-)$,
$l:\unit_\lambda\:- \stackrel{\sim}{\Rightarrow} -$
and $r:-\:\unit_\lambda\stackrel{\sim}{\Rightarrow} -$ in all
situations that such horizontal compositions makes sense.
\end{itemize}
Then we require that the following diagrams of supernatural transformations commute:
$$
\begin{tikzpicture}[commutative diagrams/every diagram]
\node  (P0) at (90:2.3cm) {$((-\: -)\: -)\: -$};
\node (P1) at (90+72:2cm) {$(-\: (-\: -))\: -$};
\node (P2) at (90+2*72:2cm) {\makebox[5ex][r]{$-\: ((-\: -)\: -)$}};
\node (P3) at (90+3*72:2cm) {\makebox[5ex][l]{$-\: (-\: (-\: -))$}};
\node (P4) at (90+4*72:2cm) {$(-\: -)\: (-\: -)$};
\path
[commutative diagrams/.cd, every arrow, every label]
(P0) edge node[swap] {$a\:-$} (P1)
(P1) edge node[swap] {$a$} (P2)
(P2) edge node[swap] {$-\:a$} (P3)
(P4) edge node {$a$} (P3)
(P0) edge node {$a$} (P4);
\end{tikzpicture},
\quad
\begin{tikzcd}
(-\:\unit_\la) \:-
\arrow[dd,swap,"a"]\arrow[rd,"r\:-"]\\
&
-\:-\\
-\:(\unit_\la\:-) \arrow[ru,"-\:l",swap]
\end{tikzcd}.
$$
A 1-morphism $F:\lambda\rightarrow \mu$ in a 2-supercategory is called
a
{\em superequivalence} if there is a 1-morphism $G$ in the other direction
such that $GF \cong \unit_\lambda$ and $F G \cong \unit_\mu$ via even 2-isomorphisms.

(ii) A {\em 2-superfunctor} $\mathbb{R}:\AA \rightarrow \BB$ between $2$-supercategories is
the following data:
\begin{itemize}
\item A function $\mathbb{R}:\ob \AA \rightarrow \ob \BB$.
\item Superfunctors
$\mathbb{R}:\mathcal{H}om_{\AA}(\lambda,\mu) \rightarrow
\mathcal{H}om_{\BB}(\mathbb{R}\lambda,\mathbb{R}\mu)$ for $\lambda,\mu \in \ob \AA$.
\item Even supernatural isomorphisms
$c:(\mathbb{R}\:-)\: (\mathbb{R}\:-) \stackrel{\sim}{\Rightarrow} \mathbb{R}(-\:-)$.
\item Even 2-isomorphisms
  $i:\unit_{\mathbb{R}\lambda}\stackrel{\sim}{\Rightarrow} \mathbb{R} \unit_\lambda$
for all $\lambda \in \ob \AA$.
\end{itemize}
Then we require that the following diagrams commute:
$$
\begin{tikzcd}
&(\mathbb{R}(-\:-))\:(\mathbb{R}\:-)
\arrow[dr,"c"] \\
((\mathbb{R}\:-)\: (\mathbb{R}\:-))\: (\mathbb{R}\:-)
\arrow[dd,"a",swap]\arrow[ur,"c\:(\mathbb{R} -)"]&
&\mathbb{R}((-\:-)\:-)\arrow[dd,"\mathbb{R} a"]\\\\
(\mathbb{R}\:-)
\:((\mathbb{R}\:-)\:(\mathbb{R}\:-))\arrow[dr,"(\mathbb{R}
-)\:c",swap]&&\mathbb{R}(-\:(-\:-))\\
&(\mathbb{R}\:-)\:(\mathbb{R}(-\:-))\arrow[ur,"c",swap]
\end{tikzcd},
$$
$$
\quad
\begin{tikzcd}
\arrow[d,"r",swap](\mathbb{R}\:-)\:\unit_{\mathbb{R}\lambda}\arrow[r,"(\mathbb{R}
-)\, i"]&(\mathbb{R}\:-)\:(\mathbb{R}\unit_\lambda)\arrow[d,"c"]\\
\mathbb{R}\:-&\arrow[l,"\mathbb{R}\, r"]\mathbb{R}(- \:\unit_\lambda)
\end{tikzcd},
\begin{tikzcd}
\arrow[d,swap,"l"]\unit_{\mathbb{R}\mu}\:(\mathbb{R}\:-)\arrow[r,"i\:(\mathbb{R}
-)"]&(\mathbb{R}\unit_\mu)\:(\mathbb{R}\:-)
\arrow[d,"c"]\\
\mathbb{R}\:-&\arrow[l,"\mathbb{R} \,l"]\mathbb{R}(\unit_\mu\:-)
\end{tikzcd}.
$$
There is a natural way to compose two 2-superfunctors. Also each 2-supercategory $\AA$
possesses an
identity 2-superfunctor, which will be denoted $\mathbb{I}$.
Hence, we get a category 2-$\SCat$ consisting of 2-supercategories and
2-superfunctors.

(iii)
Given 
2-superfunctors $\mathbb{R}, \mathbb{S}:\AA \rightarrow \BB$
for 2-supercategories $\AA$ and $\BB$,
a {\em 2-natural transformation}\footnote{In $n$Lab this is an {\em oplax
    natural transformation}.}
$(X,x):\mathbb{R} \Rightarrow \mathbb{S}$ 
is the following data:
\begin{itemize}
\item 1-morphisms $X_\lambda:\mathbb{R}\lambda \rightarrow \mathbb{S}
  \lambda$ in $\BB$ for each
  $\lambda \in \ob \AA$.
\item Even supernatural transformations $x_{\mu,\lambda}: 
X_\mu (\mathbb{R}\, -)
\Rightarrow
(\mathbb{S}\, -) X_\lambda$ 
(which are superfunctors
$\mathcal{H}om_\AA(\lambda,\mu) \rightarrow
\mathcal{H}om_{\BB}(\mathbb{R} \lambda, \mathbb{S} \mu)$)
for all $\lambda,\mu\in\ob\AA$
\end{itemize}
We require that the following diagrams commute for all 
$F:\lambda\rightarrow\mu$ and $G:\mu\rightarrow\nu$:
$$
\begin{tikzcd}
X_\nu ((\mathbb{R} G)(\mathbb{R} F))
\arrow[d,"X_\nu c",swap]
\arrow[r,"a^{-1}"] &
(X_\nu (\mathbb{R} G)) (\mathbb{R} F)
\arrow[rr,"(x_{\nu,\mu})_G(\mathbb{R} F)"]&&((\mathbb{S} G) X_\mu )(\mathbb{R} F)
\arrow[d,"a"]
\\
X_\nu \mathbb{R} (GF)\arrow[d,"(x_{\nu,\lambda})_{GF}",swap]&&&(\mathbb{S} G)
 (X_\mu (\mathbb{R} F))\arrow[d,"(\mathbb{S} G) (x_{\mu,\lambda})_F"]\\
\mathbb{S}(GF) X_\lambda&\arrow[l,"c X_\lambda"] ((\mathbb{S}
G)(\mathbb{S} F)) X_\lambda&&\arrow[ll,"a^{-1}"](\mathbb{S} G)
((\mathbb{S} F) X_\lambda)
\end{tikzcd},
$$
$$
\begin{tikzcd}
\unit_{\mathbb{S} \lambda} X_\lambda \arrow[d,"i
X_\lambda",swap]\arrow[r,"l"] &X_\lambda \arrow[r,"r^{-1}"]&X_\lambda \unit_{\mathbb{R}\lambda}\arrow[d,"X_\lambda i"]\\
(\mathbb{S} \unit_\lambda) X_\lambda&& \arrow[ll,"(x_{\lambda,\lambda})_{\unit_\lambda}",swap]X_\lambda
(\mathbb{R} \unit_\lambda)
\end{tikzcd}.
$$
A 2-natural transformation $(X,x)$ is {\em strong}\footnote{Or a
  {\em pseudonatural transformation} in $n$Lab.}
if each $x_{\mu,\lambda}$ is an isomorphism.
There is a 2-category
2-$\SCAT$ consisting of all 2-supercategories, 2-superfunctors and
2-natural transformations.

(iv) Suppose that $(X,x), (Y,y):\mathbb{R} \rightarrow \mathbb{S}$ are 
2-natural transformations for 
2-superfunctors $\mathbb{R},\mathbb{S}: \AA\rightarrow\BB$. A {\em supermodification} 
$\alpha:(X,x) 
{\:\Rightarrow\!\!\!\!\!\!\!-\!\!\!\!\!-\:\,} (Y,y)$ 
is a family of 2-morphisms
$\alpha_\lambda = \alpha_{\lambda,\0}+\alpha_{\lambda,\1}: X_\lambda \Rightarrow Y_\lambda$ for all $\lambda \in \ob \AA$,
such that the diagram
$$
\begin{tikzcd}
X_\mu (\mathbb{R} F)
\arrow[r,"(x_{\mu,\lambda})_F"]\arrow[d,swap,"\alpha_\mu (\mathbb{R}
F)"]&(\mathbb{S} F)X_\lambda \arrow[d,"(\mathbb{S} F) \alpha_\lambda"]\\
Y_\mu (\mathbb{R} F) \arrow[r,swap,"(y_{\mu,\lambda})_F"]&(\mathbb{S}
F) Y_\lambda
\end{tikzcd}
$$
commutes for all 1-morphisms $F:\lambda\rightarrow \mu$ in $\AA$.
We have that $\alpha = \alpha_\0 + \alpha_\1$ where $(\alpha_p)_\lambda :=
\alpha_{\lambda,p}$.
This makes the space $\Hom((X,x), (Y,y))$ of supermodifications
$\alpha:(X,x)
 {\:\Rightarrow\!\!\!\!\!\!\!-\!\!\!\!\!-\:\,}
(Y,y)$ into a superspace.
There is a supercategory $\mathcal{H}om(\mathbb{R}, \mathbb{S})$ consisting of
all 2-natural transformations
and supermodifications.
There is a
2-supercategory $\mathfrak{Hom}(\AA,\BB)$ consisting of
2-superfunctors, 2-natural transformations and supermodifications; it is strict if
$\BB$ is strict.
These are the morphism 2-supercategories in the strict 3-supercategory
of 2-supercategories.
Since we won't do anything with this here, we omit the details.
\end{definition}

We note that a strict 2-supercategory in the sense of
Definition~\ref{defsuper2cat} 
is the same thing as a 2-supercategory whose
coherence maps $a, l$ and $r$ are identities. In the strict case, the unit
objects $\unit_\lambda$ are uniquely determined, so do not need to be
given as part of the data.
A {\em strict 2-superfunctor} is a 2-superfunctor whose coherence maps
$c$ and $i$ are identities.
There exist 2-superfunctors between strict 2-supercategories which are
themselves not strict.

Recall for superalgebras $A$ and $B$ that 
$B\lrSMod A$ denotes the
supercategory
of $(B,A)$-superbimodules; see Example~\ref{snuggly}(iii).
Given another superalgebra $C$,
the usual tensor product over $B$ gives a superfunctor 
$$
-\otimes_B -:C\lrSMod B \boxtimes B \lrSMod A
\rightarrow C\lrSMod A.
$$
The
2-supercategory $\mathfrak{SBim}$ of
{\em superbimodules} has objects that are
superalgebras, the
morphism supercategories are defined from $\mathcal{H}om_{\mathfrak{SBim}}(A,B) :=
B\lrSMod A$, and horizontal composition comes from the tensor product operation
just mentioned.
It gives a basic example of a 2-supercategory which is not strict.

Two 2-supercategories $\AA$ and $\BB$ are 
{\em 2-superequivalent} if there 
are 2-superfunctors 
$\mathbb{R}:\AA \rightarrow \BB$ 
and $\mathbb{S}:\BB\rightarrow \AA$
such that
$\mathbb{S} \circ \mathbb{R}$ 
and $\mathbb{R} \circ \mathbb{S}$
are superequivalent to the identities in 
$\mathfrak{Hom}(\AA,\AA)$
and
$\mathfrak{Hom}(\BB,\BB)$, respectively.
Equivalently,
there is a 2-superfunctor
$\mathbb{R}:\AA\rightarrow\BB$ that induces a 
superequivalence
$\mathcal{H}om_{\AA}(\lambda,\mu)
\rightarrow \mathcal{H}om_{\BB}(\mathbb{R} \lambda, \mathbb{R}\mu)$
for all $\lambda,\mu\in\ob\AA$,
and  every $\nu \in \ob \BB$ is 
superequivalent to an object of the form
$\mathbb{R} \lambda$
for some $\lambda \in \ob \AA$.

The {\em Coherence Theorem} for 2-supercategories implies that any
2-supercategory is 2-superequivalent to a strict $2$-supercategory.
The
proof can be obtained by mimicking the argument in the purely even
case from
\cite{Leinster}.
In view of this result, 
we will sometimes assume for simplicity
that we are working in the strict case.

\begin{definition}\label{drinfeld}
Let $\AA$ be a 2-supercategory.
The {\em Drinfeld center} of $\AA$
is the monoidal supercategory 
of all strong 2-natural transformations 
$\mathbb{I} \Rightarrow \mathbb{I}$
and 
supermodifications.
Thus, an object $(X,x)$ of the Drinfeld center is a coherent family of 1-morphisms
$X_\lambda:\lambda\rightarrow \lambda$
and even supernatural isomorphisms
$x_{\mu,\lambda}: X_\mu \:- \stackrel{\sim}{\Rightarrow} -\:
X_\lambda$
for $\lambda,\mu \in \ob \A$; a morphism
$\alpha:(X,x) 
{\:\Rightarrow\!\!\!\!\!\!\!-\!\!\!\!\!-\:\,} (Y,y)$ is coherent
family of 2-morphisms $\alpha_\lambda:X_\lambda \Rightarrow
Y_\lambda$.
The tensor product 
$(X \otimes Y, x \otimes y)$
of objects $(X,x)$ and $(Y,y)$ is
defined from $(X \otimes Y)_\lambda := X_\lambda Y_\lambda$, $(x \otimes y)_{\mu,\lambda} := x_{\mu,\lambda} y_{\mu,\lambda}$;
the tensor product $\alpha \otimes \beta$
of morphisms $\alpha:(X,x) \rightarrow (U,u)$ and
$\beta:(Y,y)\rightarrow (V,v)$ is
defined from $(\alpha \otimes \beta)_{\lambda} := \alpha_\lambda \beta_\lambda$.
If $\AA$ is strict then its Drinfeld center is strict too.
\end{definition}

We remark that the Drinfeld center of a 2-supercategory is a {\em braided} monoidal supercategory,
although we omit the definition of such a structure.
(See \cite{MS} for more about Drinfeld center in the purely even setting.)

\section{$\Pi$-Supercategories}

According to Definition~\ref{defscat}, a {\em $\Pi$-supercategory}
is a supercategory with the additional data of a
parity-switching functor $\Pi$ and an odd supernatural isomorphism
$\zeta:\Pi\stackrel{\sim}{\Rightarrow} I$. It is an easy structure to work with
as there are no additional axioms,
unlike the situation for the $\Pi$-categories of Definition~\ref{jazz}. 
The same goes for superfunctors and
supernatural transformations
between $\Pi$-supercategories: there
are no additional compatibility constraints with respect to $\Pi$.

\begin{definition}\label{georgia}
A {\em $\Pi$-2-supercategory} $(\AA, \pi, \zeta)$
is a 2-supercategory $\AA$ plus families $\pi = (\pi_\lambda)$
and $\zeta = (\zeta_\lambda)$ of $1$-morphisms
$\pi_\lambda:\lambda\rightarrow \lambda$
and odd 2-isomorphisms
$\zeta_\lambda \in \Hom_{\AA}(\pi_\lambda, \unit_\lambda)$
for each object $\lambda \in \AA$.
It is {\em strict} if $\AA$ is strict.
\end{definition}

Let $\piSCat$ be the category of all $\Pi$-supercategories and
superfunctors.
Let $\piSCAT$ be the strict 2-supercategory of all $\Pi$-supercategories,
superfunctors and supernatural transformations.
The latter gives the archetypal example of a strict {$\Pi$-2-supercategory}:
the additional data of $\pi = (\pi_\A)$ and $\zeta =
(\zeta_\A)$ are defined
by letting $\pi_\A$  be the
parity-switching functor $\Pi_\A:\A \rightarrow \A$ on the
$\Pi$-supercategory $\A$,
and taking
$\zeta_\A:\pi_\A \stackrel{\sim}{\rightarrow} \unit_\A$
to be the given odd supernatural isomorphism 
$\Pi_\A \stackrel{\sim}{\Rightarrow} I_\A$.

The basic example of a $\Pi$-2-supercategory that is not strict
is the 2-supercategory $\mathfrak{SBim}$ defined at the end of the
previous section. Recall the objects are superalgebras, the 1-morphisms
are superbimodules, the 2-morphisms are superbimodule
homomorphisms, and horizontal composition is given by tensor product.
Also, for each object (i.e. superalgebra) $A$, the unit 1-morphism $\unit_A$ is the regular
superbimodule $A$.
The extra data $\pi$ and $\zeta$ needed to make $\mathfrak{SBim}$ into a
$\Pi$-2-supercategory
are given by declaring that $\pi_A :=
\Pi A$ (i.e. we apply the parity-switching functor to the regular
superbimodule),
and each $\zeta_A:\pi_A \stackrel{\sim}{\Rightarrow} \unit_A$
comes from the superbimodule homomorphism $\Pi A \rightarrow A$ that is the
identity function on the underlying set.

Each morphism supercategory $\mathcal{H}om_{\AA}(\lambda,\mu)$
in a $\Pi$-2-supercategory $\AA$ admits a parity-switching functor
$\Pi$ making it into 
a $\Pi$-supercategory, namely,
the endofunctor
$\pi_\mu -$ arising by horizontally composing on the left
by $\pi_\mu$. Alternatively, one could take the endofunctor $- \pi_\lambda$
defined by horizontally composing on the right by $\pi_\lambda$. 
These two choices are isomorphic according to our first lemma.

\begin{lemma}\label{fish}
Let $(\AA, \pi, \zeta)$ be a $\Pi$-2-supercategory.
For objects $\lambda,\mu$, 
there is an even supernatural isomorphism
$$
\beta_{\mu,\lambda}: 
\pi_\mu -
\stackrel{\sim}{\Rightarrow}
- \:\pi_\lambda.
$$
Assuming $\AA$ is strict for simplicity, this is defined by
$(\beta_{\mu,\lambda})_F := 
- \zeta_\mu F
\zeta_\lambda^{-1}$ for each $1$-morphism $F:\lambda\rightarrow\mu$.
Setting $\beta := (\beta_{\mu,\lambda})$,
the pair $(\pi, \beta)$ is an object in the {Drinfeld center} of
$\AA$ as in Definition~\ref{drinfeld}, i.e. 
the following hold (still assuming strictness):
\begin{itemize}
\item[(i)]
$(\beta_{\nu,\lambda})_{GF} = 
G (\beta_{\mu,\lambda})_F 
\circ (\beta_{\nu,\mu})_G F$
for $1$-morphisms $F:\lambda \rightarrow \mu$ and $G:\mu \rightarrow \nu$;
\item[(ii)] $(\beta_{\lambda,\lambda})_{\unit_\lambda} = 1_{\pi_\lambda}$.
\end{itemize}
Moreover:
\begin{itemize}
\item[(iii)]
$\pi_\lambda \zeta_\lambda = - \zeta_\lambda \pi_\lambda$ hence
$(\beta_{\lambda,\lambda})_{\pi_\lambda} = -1_{\pi_\lambda^2}$;
\item[(iv)]
$\xi_\lambda := \zeta_\lambda\zeta_\lambda:\pi_\lambda^2\Rightarrow
\unit_\lambda$
is an even 2-isomorphism
such that $\xi_\mu F \xi_\lambda^{-1} = 
(\beta_{\mu,\lambda})_F \pi_\lambda \circ
\pi_\mu (\beta_{\mu,\lambda})_F$
in $\Hom_{\C}(\pi_\mu^2 F, F\pi_\lambda^2)$
for all 
$F:\lambda\rightarrow \mu$.
\end{itemize}
\end{lemma}

\begin{proof}
To show that $\beta_{\mu,\lambda}$ is an even supernatural
isomorphism, we need to show for any 2-morphism $x:F
\Rightarrow G$ between 1-morphisms $F, G:\lambda\rightarrow \mu$ that
\begin{equation}\label{house}
x \pi_\lambda \circ (\beta_{\mu,\lambda})_F=
(\beta_{\mu,\lambda})_G\circ \pi_\mu x.
\end{equation}
This follows from the following calculation with the super
interchange law:
\begin{align*}
x \pi_\lambda\circ \zeta_\mu G
\zeta^{-1}_\lambda = (-1)^{|x|}\zeta_\mu x
\zeta_\lambda^{-1}
= \zeta_\mu F \zeta_\lambda^{-1}\circ 
\pi_\mu x.
\end{align*}
For (i), we must show that
$G \zeta_\mu F \zeta_\lambda^{-1}
\circ
\zeta_\nu G \zeta_\mu^{-1} F
=
-\zeta_\nu GF \zeta_\lambda^{-1}$, which is clear by the super
interchange law again.
For (ii), we have that $-\zeta_\lambda \zeta_\lambda^{-1} =
\zeta_{\lambda}^{-1} \circ \zeta_\lambda = 1_{\pi_\lambda}$.
For (iii), $\zeta_\lambda \zeta_\lambda = \zeta_\lambda \circ
\pi_\lambda \zeta_\lambda = - \zeta_\lambda \circ \zeta_\lambda
\pi_\lambda$.
Cancelling $\zeta_\lambda$ on the left, we deduce that 
$\pi_\lambda \zeta_\lambda = - \zeta_\lambda \pi_\lambda$,
hence $-\zeta_\lambda \pi_\lambda \zeta_\lambda^{-1} = 
\pi_\lambda \zeta_\lambda^{-1}\circ
\zeta_\lambda \pi_\lambda
= -
1_{\pi_\lambda^2}$.
Finally (iv) follows from the calculation:
\begin{align*} 
F \xi_\lambda \circ (\beta_{\mu,\lambda})_F \pi_\lambda\circ \pi_\mu
  (\beta_{\mu,\lambda})_F
&=
F \zeta_\lambda \zeta_\lambda \circ
\zeta_\mu F \zeta_\lambda^{-1}\pi_\lambda\circ \pi_\mu
\zeta_\mu F \zeta_\lambda^{-1}
\\
&=-\zeta_\mu F\pi_\lambda \zeta_\lambda\circ \pi_\mu
\zeta_\mu F \zeta_\lambda^{-1}
= \zeta_\mu\zeta_\mu F = \xi_\mu F.
\end{align*}
\end{proof}

Applying Lemma~\ref{fish} to the strict $\Pi$-2-supercategory $\piSCAT$, we
obtain the following.

\begin{corollary}\label{pi}
Let $(\A, \Pi_\A, \zeta_\A)$ and $(\B, \Pi_\B, \zeta_\B)$ be
$\Pi$-supercategories.
As in Definition~\ref{defscat},
there are even supernatural isomorphisms
$\xi_\A:\Pi_\A^2 \stackrel{\sim}{\Rightarrow} I_\A$ and
$\xi_\B:\Pi_\B^2 \stackrel{\sim}{\Rightarrow} I_\B$ both defined by setting
$\xi := \zeta \zeta$.
\begin{itemize}
\item[(i)]
We have that $\Pi \zeta = -\zeta \Pi$ in $\Hom(\Pi^2, \Pi)$, hence
$\Pi \xi = \xi \Pi $ in $\Hom(\Pi^3, \Pi)$.
\item[(ii)]
 For a superfunctor $F:\A \rightarrow \B$,
define
$
\beta_F := -\zeta_\B F (\zeta_\A)^{-1}:\Pi_\B F \Rightarrow F \Pi_\A.
$ 
This is an even
supernatural isomorphism such that
$\xi_\B F (\xi_\A)^{-1}=\beta_F \Pi_\A\circ\Pi_\B \beta_F$ in $\Hom((\Pi_\B)^2 F,F (\Pi_\A)^2)$.
\item[(iii)]
If $x:F \Rightarrow G$ is a supernatural transformation between
superfunctors $F, G:\A \rightarrow \B$
then $\beta_G \circ \Pi_\B x = x \Pi_\A \circ \beta_F$ in $\Hom(\Pi_\B
F, G\Pi_\A)$.
\item[(iv)] 
For superfunctors $F: \A \rightarrow \B$ and
$G:\B \rightarrow \C$, 
we have that $\beta_{G  F} = G \beta_F \circ \beta_G F$.
Also $\beta_{I} = 1_\Pi$ and $\beta_{\Pi} = - 1_{\Pi^2}$.
\end{itemize}
\end{corollary}

When working with $\Pi$-2-supercategories, notions of
2-superfunctors, 2-natural transformations and supermodifications
are just as defined for 2-supercategories in Definition~\ref{ms}: there
are no additional compatibility constraints.
Let $\piTSCat$ be the category of all $\Pi$-2-supercategories and
2-superfunctors,
and $\piTSCAT$ be the strict 2-category of all $\Pi$-2-supercategories,
2-superfunctors and 2-natural transformations.

\section{Envelopes}

In this subsection, we 
prove the statements about the functors (1) in Theorems~\ref{hop} and
\ref{bop}.
We will also construct  $\Pi$-envelopes of 2-supercategories.
We start at the level of supercategories. Recall the
functor
$-_\pi:\SCat \rightarrow \piSCat$ 
from Definition~\ref{pienv}, which sends supercategory 
$\A$ to its {\em $\Pi$-envelope} $(\A_\pi, \Pi, \zeta)$, 
and superfunctor $F$ to $F_\pi$.
In fact, this is part of the data of a strict 2-superfunctor
\begin{equation}\label{music}
-_\pi:\SCAT \rightarrow \piSCAT,
\end{equation}
sending a supernatural transformation
$x:F \Rightarrow G$ to 
$x_\pi:F_\pi \Rightarrow G_\pi$
defined from $(x_\pi)_{\Pi^a \lambda} := (-1)^{|x|a} (x_\lambda)_a^a$.

For any supercategory $\A$, there is a canonical superfunctor $J:\A \rightarrow \A_\pi$
which sends $\lambda\mapsto \Pi^\0 \lambda$ and $f\mapsto f^\0_\0$.
This is full and faithful. It is also dense: each object
$\Pi^\0 \lambda$ of $\A_\pi$ is obviously in the image, while
$\Pi^\1 \lambda$ is isomorphic to $\Pi^\0 \lambda$ via the odd
isomorphism $(1_\lambda)^\0_\1$.
This means that $\A$ and $\A_\pi$ are equivalent as
abstract categories. However they need not be superequivalent
as $J$ need not be evenly dense:

\begin{lemma}\label{e}
The canonical superfunctor $J:\A \rightarrow \A_\pi$
is a superequivalence if and only if $\A$ 
is {\em $\Pi$-complete}, meaning that every object of
$\A$ is the target of an odd isomorphism.
\end{lemma}

\begin{proof}
The ``only if'' direction is clear as every object $\Pi^a \lambda$ of $\A_\pi$ is the
target of the odd isomorphism $(1_\lambda)_{a+\1}^a:\Pi^{a+\1}\lambda \rightarrow \Pi^a \lambda$. 
Conversely, assume that $\A$ is $\Pi$-complete.
To show that $J$ is a superequivalence, it suffices to check that it
is evenly dense.
Let $\lambda$ be an object of $\A$
and $f:\mu \rightarrow \lambda$ be an odd isomorphism in $\A$.
Then $f^\1_\0:\Pi^\0 \mu \rightarrow \Pi^\1 \lambda$ is an even
isomorphism in $\A_\pi$. Hence, $\Pi^\1 \lambda$ is 
isomorphic via an even isomorphism to something in the image of
$J$, as of course is $\Pi^\0 \lambda$.
\end{proof}

Here is the universal property 
of $\Pi$-envelopes.

\begin{lemma}\label{uniprop}
Suppose $\A$ is a supercategory and $(\B, \Pi, \zeta)$ is a
$\Pi$-supercategory.
\begin{itemize}
\item[(i)]
Given a superfunctor $F:\A \rightarrow \B$, there is a canonical
superfunctor $\tilde F:\A_\pi \rightarrow \B$ such that
$F = \tilde F  J$.
\item[(ii)]
Given a supernatural transformation $x:F \Rightarrow G$
between superfunctors $F, G:\A \rightarrow \B$, there is a unique
supernatural transformation $\tilde x:\tilde F \Rightarrow \tilde G$
such that $x = \tilde x J$.
\end{itemize}
\end{lemma}

\begin{proof}
(i)
For $\lambda \in \ob \A$, we set
$\tilde F (\Pi^a \lambda) := F \lambda$
if $a = \0$ or $\Pi (F \lambda)$ if $a = \1$.
For a morphism $f:\lambda \rightarrow \mu$ in $\A$,
let $\tilde F(f_a^b):\tilde F(\Pi^a \lambda) \rightarrow \tilde
F(\Pi^b \mu)$ be $(\zeta_{F\mu}^b)^{-1} \circ Ff \circ
\zeta_{F\lambda}^a$,
where $\zeta_{F\lambda}^a$ denotes $1_{F\lambda}$ if $a = \0$ or
$\zeta_{F\lambda}$ if $a = \1$, and $\zeta_{F\mu}^b$ is interpreted similarly.

(ii) 
We are given
that 
${\tilde x}_{\Pi^\0 \lambda} = x_\lambda$ for each $\lambda \in \ob\A$.
Also, by the definition in (i), we have that
$\tilde F \zeta_{\Pi^\0 \lambda} =
\zeta_{F\lambda}$
for each $\lambda \in \ob \A$.
Hence, to ensure the supernaturality property on the morphism
$\zeta_{\Pi^\0 \lambda}:\Pi^\1 \lambda \rightarrow \Pi^\0 \lambda$, 
we must have that
${\tilde x}_{\Pi^\1 \lambda} = (-1)^{|x|} (\zeta_{G \lambda})^{-1}
\circ x_\lambda \circ \zeta_{F\lambda}$.
Thus, in general, we have that
\begin{equation}\label{snow}
{\tilde x}_{\Pi^a \lambda} = (-1)^{a|x|} (\zeta_{G\lambda}^a)^{-1}
\circ x_\lambda \circ (\zeta_{F\lambda}^b).
\end{equation}
It just remains to check that this
is indeed a supernatural
transformation, i.e. it satisfies supernaturality on all other
morphisms in $\A_\pi$.
Take a homogeneous $f:\lambda \rightarrow \mu$ in $\A$ and consider
$f_a^b:\Pi^a \lambda \rightarrow \Pi^b \mu$. We must show that
$$
(\zeta_{G\mu}^b)^{-1} \circ Gf \circ (\zeta_{G\lambda}^a) \circ 
{\tilde x}_{\Pi^a \lambda} = (-1)^{|x|(|f|+a+b)} 
{\tilde x}_{\Pi^b \mu} \circ (\zeta_{F\mu}^b)^{-1} \circ Ff \circ
(\zeta_{F\lambda}^a).
$$
This follows on substituting in the definitions of the $\tilde x$'s
from (\ref{snow}) and using that $Gf \circ x_\lambda = (-1)^{|x||f|} x_\mu
\circ Ff$.
\end{proof}

Most of the time, Lemmas~\ref{e}--\ref{uniprop} are all that one needs when
working with $\Pi$-envelopes in practice. The following gives a more
formal statement, enough to establish the claim made about the
functor (1) in Theorem~\ref{hop} from the introduction.
To state it, we let $\nu:\piSCAT\rightarrow\SCAT$ be the obvious
forgetful 2-superfunctor.

\begin{theorem}\label{2adj}
For all supercategories $\A$ and $\Pi$-supercategories $\B$,
there is a functorial superequivalence
$\mathcal{H}om(\A, \nu \B)
\rightarrow
\mathcal{H}om(\A_\pi, \B)$,
sending superfunctor $F$ to $\tilde F$ and supernatural transformation
$x$ to $\tilde x$, both as defined in Lemma~\ref{uniprop}.
Hence, the strict 2-superfunctor $-_\pi$ is left 2-adjoint to $\nu$.
\end{theorem}

\begin{proof}
We must show that the given superfunctor is fully faithful and evenly
dense.
The fully faithfulness follows from 
Lemma~\ref{uniprop}(ii).
To see that it is evenly dense, take a superfunctor 
$F:\A_\pi \rightarrow \B$. Consider the composite functor 
$FJ:\A \rightarrow \nu \B$.
Then there is an even supernatural isomorphism
$\widetilde{F J} \stackrel{\sim}{\Rightarrow} F$, which is defined by the
following even isomorphisms
$\widetilde{FJ}(\Pi^a \lambda) \stackrel{\sim}{\rightarrow} F (\Pi^a \lambda)$
for each $\lambda \in \ob \A$ and $a \in \Z/2$:
if $a = \0$, then
$\widetilde{FJ}(\Pi^\0 \lambda) = F (\Pi^\0 \lambda)$, and we just
take the identity map;
if $a = \1$, then
$\widetilde{FJ}(\Pi^\1 \lambda) = \Pi F (\Pi^\0 \lambda)$,
so we need to produce an 
isomorphism $\Pi F
(\Pi^\0 \lambda)
\stackrel{\sim}{\rightarrow} 
F (\Pi^\1 \lambda)$, which we get from
Corollary~\ref{pi}(ii).
We leave it to the reader to check the naturality.
\end{proof}

We turn our attention to $2$-supercategories.

\begin{definition}\label{sixo}
The {\em $\Pi$-envelope} of a 2-supercategory $\AA$
is the $\Pi$-2-supercategory $(\AA_\pi, \pi, \zeta)$ 
with morphism supercategories 
that are the $\Pi$-envelopes of the morphism supercategories
in $\AA$:
\begin{itemize}
\item 
The object set for $\AA_\pi$ is the same as for $\AA$.
\item 
The set of 1-morphisms $\lambda \rightarrow \mu$
in $\AA_\pi$ is
$$
\{\Pi^a F\:|\:\text{for all 1-morphisms $F:\lambda\rightarrow\mu$ in
$\AA$ and $a \in
\Z/2$}\}.
$$
\item
The 
horizontal composition of 1-morphisms
$\Pi^a F:\lambda\rightarrow\mu$ and $\Pi^b G:\mu \rightarrow \nu$ 
is defined by $(\Pi^b G)(\Pi^a F) := \Pi^{a+b} (GF)$.
\item
The superspace of
2-morphisms
$\Pi^a F \Rightarrow \Pi^b G$ in $\AA_\pi$ is defined from
$$
\Hom_{\AA_\pi}(\Pi^a F, \Pi^b G)
:=
\Pi^{a+b} \Hom_{\AA}(F, G).
$$
We denote the 2-morphism $\Pi^a F \Rightarrow \Pi^b G$ coming from
$x:F \Rightarrow G$ under this identification by $x_a^b$. If
$x$ is homogeneous of parity $|x|$ then $x_a^b$ is
homogeneous of parity $|x|+a+b$.
\item
The vertical composition of $x^b_a:\Pi^a F \Rightarrow \Pi^b G$ and
$y^c_b:\Pi^b G \Rightarrow \Pi^c H$
is defined from 
\begin{equation}\label{shot1}
y^c_b \circ x^b_a := (y \circ x)^{c}_{a}:\Pi^a F \Rightarrow \Pi^c H.
\end{equation}
\item
The horizontal composition of $x^c_a:\Pi^a F \Rightarrow \Pi^c H$
and $y^d_b:\Pi^b G \Rightarrow \Pi^d K$
is defined by 
\begin{equation}\label{shot2}
y^d_b x^c_a := (-1)^{b|x|+|y|c+bc+ab}(yx)^{c+d}_{a+b}:\Pi^{a+b}(GF) \Rightarrow \Pi^{c+d}(KH).
\end{equation}
\item
The units $\unit_\lambda$ in $\AA_\pi$ are the 1-morphisms $\Pi^\0 \unit_\lambda$.
Also define $\pi = (\pi_\lambda)$ by $\pi_\lambda := \Pi^{\1} \unit_\lambda$
and $\zeta = (\zeta_\lambda)$ by $\zeta_\lambda := (1_{\unit_\lambda})^\0_\1:\pi_\lambda
\stackrel{\sim}{\Rightarrow} \unit_\lambda$;
in particular, $\xi_\lambda := \zeta_\lambda
\zeta_\lambda:\pi_\lambda^2 \stackrel{\sim}{\Rightarrow} \unit_\lambda$ is minus the
identity.
\item
The structure maps $a, l$ and $r$ in $\AA_\pi$ are induced by the ones
in $\AA$ in the obvious way, but there are some signs to be checked to
see that this makes sense. For example, for the associator,
one needs to note that the signs in the following two expressions agree:
\begin{align*}
(z_c^f y_b^e) x_a^d &= 
(-1)^{c|y|+|z|e+ce+bc+(b+c)|x|+(|y|+|z|)d + (b+c)d+a(b+c)}
((zy)x)_{a+b+c}^{d+e+f},\\
z_c^f (y_b^e x_a^d) &= 
(-1)^{b|x|+|y|d+bd+ab+c(|x|+|y|)+|z|(d+e)+c(d+e)+(a+b)c} (z(yx))_{a+b+c}^{d+e+f}.
\end{align*}
\end{itemize}
The main check needed to verify that this is indeed a
$\Pi$-2-supercategory is made in the following lemma:
\end{definition}

\begin{lemma}\label{shot}
The horizontal and vertical compositions 
from (\ref{shot1})--(\ref{shot2}) 
satisfy the
super interchange law.
\end{lemma}

\begin{proof}
We need to show that 
$$
(v^f_d u^e_c) \circ (y^d_b x_a^c) =(-1)^{(c+e+|u|)(b+d+|y|)}(v^f_d
\circ y^d_b) (u^e_c \circ x^c_a).
$$
The left hand side equals
\begin{multline*}
(-1)^{d|u|+|v|e+de+cd+b|x|+|y|c+bc+ab}
 (vu)^{e+f}_{c+d} \circ (yx)^{c+d}_{a+b}
=\\
(-1)^{d|u|+|v|e+de+cd+b|x|+|y|c+bc+ab}
 ((vu) \circ (yx) )^{e+f}_{a+b}
 =\\
(-1)^{d|u|+|v|e+de+cd+b|x|+|y|c+bc+ab+|u||y|}
 ((v \circ y) (u \circ x))^{e+f}_{a+b},
\end{multline*}
using the super interchange law in $\AA$ for the last equality.
The right hand side equals
\begin{multline*}
(-1)^{(c+e+|u|)(b+d+|y|)}
 (v \circ y)^f_{b}
(u\circ x)^e_{a}
=\\
(-1)^{(c+e+|u|)(b+d+|y|)+b(|u|+|x|)+(|v|+|y|)e+be+ab}
 ((v \circ y) (u \circ x))^{e+f}_{a+b}.
\end{multline*}
We leave it to the reader to check that the signs here are indeed equal.
\end{proof}

For any 2-supercategory $\AA$, there is a canonical strict 2-superfunctor
$\mathbb{J}:\AA \rightarrow \AA_\pi$; it is the identity on objects,
it sends the 1-morphism $F:\lambda\rightarrow \mu$ to $\Pi^\0 F$,
and the 2-morphism $x:F \Rightarrow G$ to $x_\0^\0:\Pi^\0 F \Rightarrow
\Pi^\0 G$. The analog of Lemma~\ref{e} is as follows:

\begin{lemma}\label{e2}
For a 2-supercategory $\AA$, the canonical 2-superfunctor
$\mathbb{J}:\AA \rightarrow \AA_\pi$ is a 2-superequivalence if and only
if  $\AA$ is {\em $\Pi$-complete}, meaning that it
possesses 1-morphisms $\pi_\lambda:\lambda\rightarrow \lambda$
and odd 2-isomorphisms $\pi_\lambda \cong \unit_\lambda$
for every $\lambda \in \ob \AA$.
\end{lemma}

\begin{proof}
Applying Lemma~\ref{e} to the morphism supercategories, we get that
$\mathbb{J}$ is a 2-superequivalence if and only if every 1-morphism
in $\AA$ is the target of an odd 2-isomorphism. It is clearly
sufficient to verify this condition just for the unit
1-morphisms $\unit_\lambda$ in $\AA$.
\end{proof}

Taking $\Pi$-envelopes actually defines a strict 2-functor
\begin{equation}\label{massage}
-_\pi:\TSCAT \rightarrow \piTSCAT.
\end{equation}
We still need to specify this on 2-superfunctors and 2-natural
transformations:
\begin{itemize}
\item
Suppose that 
$\mathbb{R}:\AA \rightarrow \BB$ is a 2-superfunctor
with coherence maps
$c:(\mathbb{R} -)(\mathbb{R} -)
\stackrel{\sim}{\Rightarrow} \mathbb{R}(-\,-)$ and $i:\unit_{\mathbb{R} \lambda} \stackrel{\sim}{\Rightarrow}
\mathbb{R} \unit_\lambda$ for each $\lambda \in \ob \AA$.
Then we let
$\mathbb{R}_\pi:\AA_\pi \rightarrow \BB_\pi$ be the 2-superfunctor
equal to $\mathbb{R}$ on objects, and given by the rules $\Pi^a F \mapsto \Pi^a
(\mathbb{R} F)$ on 1-morphisms and $x_a^b \mapsto (\mathbb{R} x)_a^b$
on 2-morphisms.
Its coherence maps $c_\pi:(\mathbb{R}_\pi -)(\mathbb{R}_\pi -)
\stackrel{\sim}{\Rightarrow} \mathbb{R}_\pi(-\,-)$ and
$i_\pi:
\unit_{\mathbb{R}_\pi \lambda} \stackrel{\sim}{\Rightarrow}
\mathbb{R}_\pi \unit_\lambda$ for $\mathbb{R}_\pi$
are defined by
$(c_\pi)_{\Pi^a F, \Pi^b G} := (c_{F,G})_{a+b}^{a+b}$ and
$i_\pi := i_{\0}^{\0}$.
\item
If $(X,x):\mathbb{R} \Rightarrow \mathbb{S}$ is a 2-natural
transformation,
we let $(X_\pi, x_\pi):\mathbb{R}_\pi \Rightarrow \mathbb{S}_\pi$ be
the 2-natural transformation
defined from
$(X_\pi)_\lambda := \Pi^\0 X_\lambda$ and
$((x_\pi)_{\mu,\lambda})_{\Pi^a F} := ((x_{\mu,\lambda})_F)^a_a$.
\end{itemize}

\begin{lemma}\label{uniprop2}
Suppose $\AA$ is a 2-supercategory and $(\BB, \pi, \zeta)$ is a
$\Pi$-2-supercategory.
\begin{itemize}
\item[(i)]
Given a graded 
2-superfunctor $\mathbb{R}:\AA \rightarrow \BB$, there is a canonical
graded 2-superfunctor $\tilde{\mathbb{R}}:\AA_\pi \rightarrow \BB$ such that
$\mathbb{R} = \tilde{\mathbb{R}}  \mathbb{J}$.
\item[(ii)]
Given a 2-natural transformation $(X,x):\mathbb{R} \Rightarrow
\mathbb{S}$ between 2-superfunctors $\mathbb{R},\mathbb{S}:\AA
\rightarrow \BB$, there is a unique
2-natural transformation $(\tilde X, \tilde x):\tilde{\mathbb{R}} \Rightarrow \tilde{\mathbb{S}}$
such that 
${\tilde X}_\lambda = X_\lambda$ 
and $x_{\mu,\lambda} = \tilde x_{\mu,\lambda} \mathbb{J}$
for all
$\lambda,\mu \in \ob \AA$.
\end{itemize}
\end{lemma}

\begin{proof}
To simplify notation throughout this proof, we will assume that $\BB$
is strict.

(i)
On objects, we take 
$\tilde{\mathbb{R}} \lambda := \mathbb{R} \lambda$.
To specify its effect on 1- and 2-morphisms, we first introduce some
notation:
for $a \in \Z/2$ and $\lambda \in \ob \BB$, 
let $\zeta_\lambda^a : \pi_\lambda^a \Rightarrow \unit_\lambda$ denote the
2-morphism
$1_{\unit_\lambda} \in \Hom_{\BB}(\unit_\lambda,\unit_\lambda)$
if $a = \0$ or the 2-morphism
$\zeta_\lambda\in \Hom_{\BB}(\pi_\lambda, \unit_\lambda)$ if $a = \1$.
Then, for a 1-morphism $F:\lambda \rightarrow \mu$ in $\AA$ and $a\in \Z/2$,
we set $\tilde{\mathbb{R}} (\Pi^a F) := \pi_{\mathbb{R}\mu}^a (\mathbb{R}F)$.
Also, if $x:F \Rightarrow G$ is a 2-morphism in $\AA$
between 1-morphisms $F, G:\lambda\rightarrow\mu$,
we define $\tilde{\mathbb{R}}(x_a^b):\tilde{\mathbb{R}}(\Pi^a F)
\Rightarrow 
\tilde{\mathbb{R}}(\Pi^b G)$ to be the following composition:
$$
\begin{CD}
\pi_{\mathbb{R}\mu}^a (\mathbb{R}F)
&@>\zeta_{\mathbb{R}\mu}^a(\mathbb{R}F)>>
&\mathbb{R}F
&
@>\mathbb{R}x>>
&\mathbb{R}G
&@>(\zeta_{\mathbb{R}\mu}^b)^{-1}(\mathbb{R}G)>>
&\pi_{\mathbb{R}\mu}^b (\mathbb{R}G).
\end{CD}
$$
In other words, by the super interchange law, we have that
\begin{equation}\label{snowy}
\tilde{\mathbb{R}}(x_a^b)
=(-1)^{a |x|} (\zeta_{\mathbb{R}\mu}^b)^{-1} \zeta_{\mathbb{R}\mu}^a
(\mathbb{R} x).
\end{equation}
Recalling (\ref{shot1}), it is easy to see from this definition that
$\tilde{\mathbb{R}}(y_b^c \circ x_a^b) = 
\tilde{\mathbb{R}}(y_b^c) \circ \tilde{\mathbb{R}}(x_a^b)$.
Thus, we have specified the first two pieces of data from Definition~\ref{ms}(ii)
that are required to define the
2-superfunctor $\tilde{\mathbb{R}}$.

For the other two pieces of required 
data, let $c:(\mathbb{R} -) (\mathbb{R}-) \stackrel{\sim}{\Rightarrow}
\mathbb{R}(-\,-)$ and $i:\unit_{\mathbb{R}\lambda} \stackrel{\sim}{\Rightarrow}
\mathbb{R} \unit_\lambda$ be the coherence maps for $\mathbb{R}$.
The coherence map $\tilde\imath$ for $\tilde{\mathbb{R}}$ is just the same
as $i$.
We define the other coherence map $\tilde c$ for $\tilde{\mathbb{R}}$
by letting ${\tilde c}_{\Pi^b G, \Pi^a F}:\tilde{\mathbb{R}}(\Pi^b G)\,
\tilde{\mathbb{R}}(\Pi^a F) \Rightarrow \tilde{\mathbb{R}}(\Pi^{a+b}
(GF))$ 
be the following composition (for $F:\lambda \rightarrow \mu$ and $G:\mu
\rightarrow \nu$):
$$
\begin{CD}
\pi^b_{\mathbb{R}\nu} (\mathbb{R} G) \pi^a_{\mathbb{R}\mu} (\mathbb{R}
F)
&@>\pi^b_{\mathbb{R}\nu}
(\beta^a_{\mathbb{R}\nu,\mathbb{R}\mu})^{-1}_{\mathbb{R}G} (\mathbb{R} F)>>&
\pi^b_{\mathbb{R}\nu} \pi^a_{\mathbb{R}\nu} (\mathbb{R} G) (\mathbb{R}
F)
&@>m_{b,a} c_{G,F}>>&
\pi^{a+b}_{\mathbb{R} \nu} \mathbb{R}(GF).
\end{CD}
$$
Here, 
$\beta_{\mathbb{R}\nu,\mathbb{R}\mu}^a:\pi^a_{\mathbb{R}\nu} -
\stackrel{\sim}{\Rightarrow}
- \pi^a_{\mathbb{R}\mu}$ is the identity if $a = \0$ or the even
supernatural isomorphism $\beta_{\mathbb{R}\nu,\mathbb{R}\mu}$ 
from Lemma~\ref{fish} if $a = \1$, and
$m_{b,a}:\pi^b_{\mathbb{R}\nu} \pi^a_{\mathbb{R}\nu} \stackrel{\sim}{\Rightarrow}
\pi^{a+b}_{\mathbb{R}\nu}$
is the identity 
if $ab = \0$, or
the 2-isomorphism $-\xi_{\mathbb{R}\nu} = -\zeta_{\mathbb{R}\nu}
\zeta_{\mathbb{R}\nu}$ from 
Lemma~\ref{fish}(iv) if $ab = \1$.

The key point now is to check the naturality of $\tilde c$.
Take $x:F \Rightarrow H$ and $y:G \Rightarrow K$. We must show that
the following diagram commutes for all $a,b,c,d \in \Z/2$:
$$
\begin{CD}
\pi^b_{\mathbb{R}\nu}(\mathbb{R}G) \pi^a_{\mathbb{R}\mu}(\mathbb{R} F)
&@>\tilde
c_{\Pi^b G, \Pi^a F} >>&\pi^{a+b}_{\mathbb{R}\nu} \mathbb{R}(GF)\\
@V\tilde{\mathbb R}(y^d_b)
\tilde{\mathbb{R}}(x^c_a)VV&&@VV\tilde{\mathbb{R}}(y_b^d x_a^c)V\\
\pi^d_{\mathbb{R}\nu}(\mathbb{R}K) \pi^c_{\mathbb{R}\mu} (\mathbb{R} H)&@>\tilde
c_{\Pi^d K, \Pi^c H} >>&\pi^{c+d}_{\mathbb{R}\nu} \mathbb{R}(KH).
\end{CD}
$$
Recalling (\ref{shot2}) and (\ref{snowy}), 
the composite of the top and right hand maps is equal to
$$
(-1)^{a|x|+(a+b+c)|y|+ab+bc}
\left((\zeta_\nu^{c+d})^{-1} \zeta_\nu^{a+b} \mathbb{R}(yx)\right)
\circ 
\left(m_{b,a} c_{G,F}\right) \circ \left(\pi_{\mathbb{R}\nu}^b
(\beta_{\mathbb{R}\nu, \mathbb{R}\mu}^a)_{\mathbb{R}G}^{-1}
(\mathbb{R} F)\right).
$$
Also the composite of the bottom and left hand maps is 
$$
(-1)^{a|x|+b|y|} \left( m_{d,c} c_{K,H}\right) \circ \left(
  \pi^d_{\mathbb{R}\nu} (\beta^c_{\mathbb{R} \nu,
    \mathbb{R}\mu})_{\mathbb{R}K}^{-1}
(\mathbb{R}H) \right)\circ
\left( 
(\zeta_\nu^d)^{-1} \zeta_\nu^b (\mathbb{R} y) (\zeta_\mu^c)^{-1}
\zeta_\mu^a (\mathbb{R} x)
\right).
$$
To see that these two are indeed equal, use the following
commutative diagrams:
\begin{align*}
\begin{CD}
\pi^b_{\mathbb{R}\nu} \pi^a_{\mathbb{R}\nu}
&@>m_{b,a}>>&\pi^{a+b}_{\mathbb{R}\nu}\\
@V(\zeta_{\mathbb{R}\nu}^d)^{-1} \zeta_{\mathbb{R}\nu}^b (\zeta_{\mathbb{R}\nu}^c)^{-1} \zeta_{\mathbb{R}\nu}^a
VV&&@VV(-1)^{ab+bc} (\zeta_{\mathbb{R}\nu}^{c+d})^{-1} \zeta_{\mathbb{R}\nu}^{a+b} V\\
\pi^d_{\mathbb{R}\nu} \pi^c_{\mathbb{R}\nu}
&@>>m_{d,c}>&\pi^{c+d}_{\mathbb{R}\nu},
\end{CD}\\
\begin{CD}
\pi^a_{\mathbb{R}\nu} (\mathbb{R}
G)&@>(\beta^a_{\mathbb{R}\nu,\mathbb{R}\mu})_{\mathbb{R} G}>>& (\mathbb{R} G)
\pi^a_{\mathbb{R}\mu}\\
@V(\zeta_{\mathbb{R}\nu}^c)^{-1} \zeta_{\mathbb{R}\nu}^a (\mathbb{R} y) VV&&@VV (-1)^{(a+c)|y|}(\mathbb{R} y)
(\zeta_{\mathbb{R}\mu}^c)^{-1} \zeta_{\mathbb{R}\mu}^a V\\
\pi^c_{\mathbb{R}\nu} (\mathbb{R} K)&@>>(\beta_{\mathbb{R}\nu,
  \mathbb{R}\mu}^c)_{\mathbb{R} K}>& (\mathbb{R} K)
\pi^c_{\mathbb{R}\mu}.
\end{CD}
\end{align*}
To establish the latter two diagrams, note by the definitions of
$m_{b,a}$ and $\beta^a_{\mathbb{R}\nu, \mathbb{R}\mu}$ that
$\zeta_{\mathbb{R}\nu}^b  \zeta_{\mathbb{R}\nu}^a = (-1)^{ab}
\zeta_{\mathbb{R}\nu}^{a+b} \circ m_{b,a}$ and
$\left((\mathbb{R} y) \zeta^a_{\mathbb{R}\mu}\right) \circ (\beta^a_{\mathbb{R}\nu,
  \mathbb{R}\mu})_{\mathbb{R}G}
= (-1)^{a|y|} \zeta^a_{\mathbb{R}\nu}(\mathbb{R} y)$, then use the
super interchange law.

We leave it to the reader to verify that the coherence axioms hold,
i.e. the two diagrams of Definition~\ref{ms}(ii) commute. This
depends crucially on Lemma~\ref{fish}.

(ii)
Take a 1-morphism $F:\lambda\rightarrow\mu$ in $\AA$.
We are given that $({\tilde x}_{\mu,\lambda})_{\Pi^\0 F} =
(x_{\mu,\lambda})_F$.
In order for $\tilde x_{\mu,\lambda}$ to satisfy naturality on
the 2-morphism $(1_F)_{\1}^{\0}:\Pi^\1 F \Rightarrow \Pi^\0 F$,
we are also forced to have
$({\tilde x}_{\mu,\lambda})_{\Pi^\1 F} = \left(\zeta_{\mathbb{S}\mu}
(\mathbb{S} F) X_\lambda\right)^{-1}
\circ (x_{\mu,\lambda})_F \circ \left(X_\mu \zeta_{\mathbb{R}\mu}
(\mathbb{R} F)\right)$.
Thus, in general, we have that
\begin{align}\label{steak}
({\tilde x}_{\mu,\lambda})_{\Pi^a F} &= \left(\zeta^a_{\mathbb{S}\mu}
(\mathbb{S} F)X_\lambda\right)^{-1}
\circ (x_{\mu,\lambda})_F \circ \left(X_\mu \zeta^a_{\mathbb{R}\mu}
(\mathbb{R} F)\right)\notag\\
&
= \left(\pi^a_{\mathbb{S}\mu}
(x_{\mu,\lambda})_F\right) \circ 
\left((\beta_{\mathbb{S}\mu, \mathbb{R}\mu}^a)^{-1} (\mathbb{R}
  F)\right).
\end{align}
To check naturality in general, take some homogeneous $x:F \Rightarrow G$, and consider
$x_a^b:\Pi^a F \Rightarrow \Pi^b G$. We know that
$(x_{\mu,\lambda})_G \circ \left(X_\mu (\mathbb{R}x)\right) = \left((\mathbb{S} x)
X_\lambda\right) \circ (x_{\mu,\lambda})_F$, and need to prove that
$(\tilde x_{\mu,\lambda})_{\Pi^b G} \circ \left(\tilde X_\mu (\tilde{\mathbb{R}}x_a^b)\right) = \left((\tilde{\mathbb{S}} x_a^b)
\tilde X_\lambda\right) \circ (\tilde{x}_{\mu,\lambda})_{\Pi^a F}$.
On expanding all the definitions, this reduces to checking the
following identity:
\begin{multline*}
\left(\zeta_{\mathbb{S}\mu}^b (\mathbb{S} G) X_\lambda\right)^{-1} 
\circ (x_{\mu,\lambda})_G \circ \left(X_\mu \zeta^b_{\mathbb{R}\mu}
(\mathbb{R} G)\right)  \circ \left(X_\mu (\zeta^b_{\mathbb{R}\mu})^{-1}
  \zeta^a_{\mathbb{R}\mu} (\mathbb{R} x)\right) =\\
\left((\zeta_{\mathbb{S}\mu}^b)^{-1} \zeta_{\mathbb{S}\mu}^a (\mathbb{S} x)
X_\lambda\right)
\circ \left(\zeta^a_{\mathbb{S}\mu} (\mathbb{S} F) X_\lambda\right)^{-1} \circ
(x_{\mu,\lambda})_F \circ \left(X_\mu \zeta^a_{\mathbb{R}\mu}(\mathbb{R} F)\right),
\end{multline*}
which is quite straightforward.

It remains to verify that $(\tilde X, \tilde x)$ satisfies the two axioms for 2-natural transformations from
Definition~\ref{ms}(iii). We leave this to the reader again; one needs
to use Lemma~\ref{fish} repeatedly.
\end{proof}

\begin{example}
Assume that $\k$ is a field, and
recall the monoidal supercategory $\mathcal{I}$ with one object from
Example~\ref{day}(ii).
Its $\Pi$-envelope $\mathcal{I}_\pi$ is a monoidal $\Pi$-supercategory
with two objects $\Pi^\0$ and $\Pi^\1$.
Each morphism space $\Hom_{\mathcal{I}_\pi}(\Pi^a, \Pi^b)$ is
one-dimensional
with basis $1_a^b$.
The tensor product satisfies $\Pi^b \otimes \Pi^a = \Pi^{a+b}$
and $1_b^d \otimes 1_a^c = (-1)^{(a+c)b} 1_{a+b}^{c+d}$.
We also have the monoidal $\Pi$-supercategory $\SVEC$ from Example~\ref{washing}(i).
By Lemma~\ref{uniprop2}(i), 
the canonical superfunctor $F:\mathcal{I} \rightarrow \SVEC$
sending the only object to $\k$ extends to
a monoidal superfunctor $\tilde F:\mathcal{I}_\pi \rightarrow \SVEC$.
This sends $\Pi^a\mapsto\Pi^a \k$ and
$1_a^b \mapsto (\operatorname{id}_a^b:\Pi^a \k \rightarrow \Pi^b
\k, 1 \mapsto 1)$; its coherence maps are $\Pi^b \k \otimes \Pi^a \k \rightarrow
\Pi^{a+b} \k,
1 \otimes 1 \mapsto 1$.
The signs are consistent because
the linear map $\id_b^d\otimes \id_a^c:\Pi^b \k \otimes \Pi^a \k
\rightarrow \Pi^d \k \otimes \Pi^c \k$ sends
$1 \otimes 1 \mapsto
(-1)^{(a+c)b} 1 \otimes 1$.
\end{example}

Using Lemma~\ref{uniprop2}, one can prove the following. In the statement, $\nu$ denotes the
obvious forgetful functor (actually, here it is a 2-functor).

\begin{theorem}\label{golly}
For all 2-supercategories $\AA$ and $\Pi$-2-supercategories $\BB$,
there is a functorial equivalence
$\mathcal{H}om(\AA, \nu \BB)
\rightarrow
\mathcal{H}om(\AA_\pi, \BB)$,
sending 2-superfunctor $\mathbb{R}$ to $\tilde{\mathbb{R}}$ and 2-natural transformation
$(X,x)$ to $(\tilde X, \tilde x)$, both as defined in Lemma~\ref{uniprop2}.
Hence, the strict 2-functor $-_\pi$ is left 2-adjoint to $\nu$.
\end{theorem}

On specializing to 2-supercategories with one object, this implies the result about the
functor (1) made in the statement of Theorem~\ref{bop} from the introduction.

\begin{remark}\label{phone}
In fact, we should really go one level higher here, viewing $-_\pi$ as a strict
3-superfunctor from the strict 3-supercategory of 2-supercategories to
the strict 3-supercategory of $\Pi$-2-supercategories,
by associating a supermodification
$\alpha_\pi:(X_\pi,x_\pi)
{\:\Rightarrow\!\!\!\!\!\!\!-\!\!\!\!\!-\:\,} (Y_\pi,y_\pi)$ 
to each supermodification
$\alpha:(X,x)
{\:\Rightarrow\!\!\!\!\!\!\!-\!\!\!\!\!-\:\,} (Y,y)$ defined from
$(\alpha_\pi)_\lambda := 
(\alpha_\lambda)_\0^\0$.
We leave it to the reader to formulate an appropriate part (iii)
to 
Lemma~\ref{uniprop2} explaining how to extend $\alpha:(X,x)
{\:\Rightarrow\!\!\!\!\!\!\!-\!\!\!\!\!-\:\,} (Y,y)$
to
$\tilde\alpha:(\tilde X, \tilde x)
{\:\Rightarrow\!\!\!\!\!\!\!-\!\!\!\!\!-\:\,} (\tilde Y ,\tilde y)$.
Then Theorem~\ref{golly} becomes a
2-superequivalence
$$
\mathfrak{Hom}(\AA, \nu \BB)
\rightarrow
\mathfrak{Hom}(\AA_\pi, \BB).
$$
In particular, 
it follows that
there is an induced monoidal superfunctor
from the Drinfeld center of a 2-supercategory $\AA$ to the Drinfeld center of
its $\Pi$-envelope $\AA_\pi$; the latter is a monoidal $\Pi$-supercategory in the sense
of Definition~\ref{thnder}.
\end{remark}

\section{$\Pi$-Categories}

We continue to assume that $\k$ is a commutative ground
ring.
Let $\piCat$ be the category of all \smallcat $\Pi$-categories
and $\Pi$-functors in the sense of 
Definition~\ref{jazz}.
Recall also that we denote the underlying category of a
supercategory $\A$ by $\underline{\A}$; see Definition~\ref{defsupercat}(v).
If $(\A, \Pi_\A, \zeta_\A)$ 
is a $\Pi$-supercategory and we set $\xi_\A := \zeta_\A \zeta_\A$,
then $(\underline{\A}, \underline{\Pi}_\A, \xi_\A)$ is a
$\Pi$-category thanks to Corollary~\ref{pi}(i).
Given another $\Pi$-supercategory
$(\B, \Pi_\B, \zeta_\B)$ and a superfunctor $F:\A
\rightarrow \B$,
Corollary~\ref{pi}(ii) explains how to construct the additional natural isomorphism
$\beta_F$
needed to make the underlying functor $\underline{F}$ into a
$\Pi$-functor from $(\underline{\A}, \underline{\Pi}_\A, \xi_\A)$ to
$(\underline{\B}, \underline{\Pi}_\B,\xi_\B)$.
Using also Corollary~\ref{pi}(iv), this shows that there is a functor
\begin{align}\label{upgrade}
E_1:\piSCat \rightarrow \piCat
\qquad
&(\A, \Pi, \zeta) \mapsto (\underline{\A},
\underline{\Pi}, \xi),
\quad
F \mapsto (\underline{F}, \beta_F).
\end{align}
This is the functor (2) in (\ref{golf1}).

In order to complete the proof of Theorem~\ref{hop}, 
we must show that the functor $E_1$ is an equivalence, so that
a $\Pi$-supercategory $(\A, \Pi, \zeta)$ can be recovered up
to superequivalence from
its underlying category
$(\underline{\A}, \underline{\Pi}, \xi)$.
To establish this, we define a functor
in the other direction:
\begin{align}\label{D}
D_1:\piCat \rightarrow \piSCat
\qquad
&(\A, \Pi, \xi) \mapsto (\widehat{\A}, \widehat{\Pi},
\zeta),
\quad
(F, \beta_F) \mapsto \widehat{F}.
\end{align}
This sends $\Pi$-category $(\A, \Pi, \xi)$
to the {\em associated $\Pi$-supercategory} $(\widehat{\A}, \widehat{\Pi},
\zeta)$, which is the supercategory with the same objects as $\A$ and
morphisms 
$
\Hom_{\widehat{\A}}(\lambda,\mu)_{\0} :=
\Hom_{\A}(\lambda,\mu)$,
$\Hom_{\widehat{\A}}(\lambda,\mu)_{\1} := \Hom_{\A}(\lambda, \Pi
\mu)$.
Composition in $\widehat{\A}$
is induced by the composition in
$\A$: 
if $\hat f:\lambda \rightarrow \mu$ and
$\hat g:\mu \rightarrow \nu$ are homogeneous morphisms
in $\widehat{\A}$ then
\begin{itemize}
\item
if $\hat f$ and $\hat g$ are both even, so $\hat f= f$ and $\hat g =
g$ for morphisms
$f:\lambda\rightarrow \mu$ and $g:\mu \rightarrow \nu$ in
$\A$, then we set
 $\hat g \circ  \hat f := g \circ f$;
\item
if $\hat f$ is even and $\hat g$ is odd, so 
$\hat f = f$ and $\hat g =  g$ for
$f:\lambda\rightarrow \mu$ and $g:\mu\rightarrow \Pi \nu$ in
$\A$, then we again set
 $\hat g \circ  \hat f := g \circ f$;
\item
if $\hat f$ is odd and $\hat g$ is even, so $\hat f = f$ and $\hat g =
g$ for $f:\lambda \rightarrow \Pi \mu$ and $g:\mu \rightarrow \nu$,
we set $\hat g \circ  \hat f := (\Pi g) \circ f$;
\item
if $\hat f$ and $\hat g$ are both odd, so $\hat f = f$ and $\hat g =
g$ for $f:\lambda \rightarrow \Pi \mu$ and $g:\mu \rightarrow \Pi
\nu$, 
we set $\hat g\circ \hat f := \xi_\nu \circ (\Pi g) \circ f$.
\end{itemize}
The check that $(\hat h \circ \hat g) \circ
\hat f = \hat h \circ (\hat g \circ \hat f)$ for odd $\hat f, \hat g,
\hat h$ depends on the axiom $\xi \Pi = \Pi \xi$.
To make $\widehat{\A}$ into a $\Pi$-supercategory, we define
$\widehat{\Pi}:\widehat{\A} \rightarrow \widehat{\A}$
to be the superfunctor that is equal to $\Pi$ on objects, while
$\widehat{\Pi} \hat f := \Pi f$ if $\hat f$ is even coming from
$f:\lambda \rightarrow \mu$ in $\A$, and
$\widehat{\Pi} \hat f := -\Pi f$ if $\hat f$ is odd coming from
$f:\lambda \rightarrow \Pi \mu$ in $\A$.
The odd natural isomorphism $\zeta:\widehat{\Pi} \rightarrow
I$
is defined on object $\lambda$ by $\zeta_\lambda :=
1_{\Pi\lambda}$,
i.e. it is the identity morphism $\Pi \lambda \rightarrow \Pi \lambda$ in $\A$ viewed as an odd morphism $\Pi \lambda \rightarrow \lambda$ in
$\widehat{\A}$.
Finally, 
if $(F,\beta_F): \A \rightarrow \B$ 
is a $\Pi$-functor,
we get induced a superfunctor $\widehat{F}:\widehat{\A}
\rightarrow \widehat{\B}$
between the associated supercategories as follows:
it is the same as $F$ on objects; on a homogeneous morphism $\hat f:\lambda \rightarrow
\mu$ in $\widehat{\A}$
we have that $\widehat{F} \hat f := Ff$ if $\hat f$ is even coming
from $f:\lambda\rightarrow \mu$ in $\A$, or
$\widehat{F} \hat f := (\beta_F)^{-1}_\mu\circ Ff$ if $\hat f$ is odd coming from
$f:\lambda\rightarrow\Pi_\A \mu$ in $\A$.
The check that $\widehat{F} (\hat g \circ \hat f) = (\widehat{F} \hat
g) \circ (\widehat{F} \hat f)$ for odd $\hat f, \hat g$ depends on the axiom 
$F \xi_\A = \xi_\B F\circ \Pi_\B \beta_F^{-1} \circ \beta_F^{-1} \Pi_\A$.

\begin{lemma}\label{maineq}
The functors 
$D_1:\piCat \rightarrow \piSCat$
and $E_1:\piSCat \rightarrow \piCat$ are mutually inverse equivalences
of categories.
\end{lemma}

\begin{proof}
We have simply that $E_1 \circ D_1 = I_{\piCat}$. It remains to show that $D_1 \circ E_1\cong I_{\piSCat}$.
To see this, we have to define a natural isomorphism $T:D_1 \circ E_1
\stackrel{\sim}{\Rightarrow} I_{\piSCat}$.
So for each $\Pi$-supercategory $(\A, \Pi, \zeta)$, we need
to produce 
an isomorphism of supercategories $T_\A:\widehat{{(}\underline{\A}{)}}
\stackrel{\sim}{\rightarrow}
\A$.
We take $T_{\A}$ to be the identity on objects (which are the
same in $\widehat{{(}\underline{\A}{)}}$ as in $\A)$.
On a morphism $\hat f:\lambda \rightarrow \mu$ in
$\widehat{{(}\underline{\A}{)}}$,
we let $T_{\A}(\hat f) := f$ if $\hat f$ is even
coming from an even morphism $f:\lambda\rightarrow\mu$ in $\A$, or
$\zeta_\mu \circ f$ if $\hat f$ is odd coming from an even morphism $f:\lambda
\rightarrow \Pi \mu$ in $\A$.

To check that $T_{\A}$ is a functor, we need to show that
$T_{\A}(\hat g \circ \hat f) = T_{\A}(\hat g) \circ
T_{\A}(\hat f)$ for $\hat f:\lambda\rightarrow \mu$ and
$\hat g:\mu \rightarrow \nu$:
\begin{itemize}
\item
This is
clear if both $\hat f$ and $\hat g$ are even.
\item
If $\hat f$ is even and $\hat g$ is odd,
so $\hat f = f$ and $\hat g = g$ for even $f:\lambda \rightarrow \mu$
and $g:\mu \rightarrow \Pi \nu$ in $\A$,
we have that $T_{\A}(\hat g \circ \hat f) = \zeta_\nu \circ g
\circ f = T_{\A}(\hat g) \circ T_{\A}(\hat f)$.
\item
If $f$ is
odd and $g$ is even, so $\hat f = f$ and $\hat g = g$ for
$f:\lambda\rightarrow \Pi \mu$ and $g:\mu \rightarrow \nu$, then
$T_{\A}(\hat g \circ \hat f) = \zeta_\nu \circ (\Pi g) \circ f$,
while
$T_{\A}(\hat g) \circ T_{\A}(\hat f) = g \circ \zeta_\mu
\circ f$. These are equal as $\zeta_\nu \circ \Pi g = g \circ
\zeta_\mu$ by the supernaturality of $\zeta$.
\item
If both are odd, so $\hat f = f$ and $\hat g = g$ for
$f:\lambda \rightarrow \Pi \mu$ and $g:\mu \rightarrow \Pi \nu$, 
then $T_{\A}(\hat g \circ \hat f) =
\xi_\nu \circ (\Pi g)
\circ f$. By the super interchange law, $\xi_\nu = - \zeta_\nu \circ
\zeta_{\Pi\nu}$, while supernaturality of $\zeta$ gives that 
$\zeta_{\Pi \nu} \circ \Pi g = -g \circ \zeta_\mu$.
Hence, $T_{\A}(\hat g \circ \hat f)$ equals $\zeta_\nu \circ
g \circ \zeta_\mu \circ f=
T_{\A}(\hat g) \circ T_{\mathcal
  A}(\hat f)$.
\end{itemize}

To see that $T_{\A}$ is an isomorphism, we just need to see
that it is bijective on morphisms. This is clear on even morphisms,
and follows on odd morphisms because the function
$\Hom_{\widehat{(\underline{\A})}}(\lambda,\mu)_{\1} = 
\Hom_{\A}(\lambda, \Pi \mu)_{\0}
\rightarrow
\Hom_{\A}(\lambda,\mu)_{\1}, f \mapsto 
\zeta_\mu \circ f$ is invertible with inverse $f \mapsto
\zeta_\mu^{-1} \circ f$. 

Finally we must check the naturality of
$T$:
there is an equality 
of superfunctors
$F  T_\A = T_\B \widehat{(\underline{F})}:
\widehat{(\underline{A})} \rightarrow \B$
for any superfunctor $F:\A \rightarrow
\B$
between $\Pi$-supercategories $\A$ and $\B$.
This is clear on objects and even morphisms.
Consider an odd morphism $\hat f:\lambda\rightarrow\mu$ in
$\widehat{(\underline{\A})}$ coming from an even morphism
$f:\lambda\rightarrow \Pi_\A\mu$ in $\A$.
We have that $F (T_\A \hat f) = F (\zeta_\A)_\mu \circ F f$
and $T_\B (\widehat{(\underline{F})} \hat f) = (\zeta_\B)_{F\mu} \circ  (\beta_F)_\mu^{-1}
\circ Ff$. These are equal because $\zeta_\B F = F
\zeta_\A \circ \beta_F$ by the definition in Corollary~\ref{pi}(ii)
and the super interchange law.
\end{proof}

We need one more general notion, which we spell out below
under the simplifying assumption that our 2-categories are strict. The reader should have no trouble
interpreting this in the non-strict case; see Definition~\ref{green} from the
introduction where this is done when there is only one object.
Our general conventions regarding 2-categories are analogous to 
the ones for 2-supercategories in Definition~\ref{ms}.

\begin{definition}\label{clean}
(i) A {\em $\Pi$-2-category} $(\AA, \pi, \beta, \xi)$
is a $\k$-linear $2$-category $\AA$
plus
a family $\pi = (\pi_\lambda)$ of $1$-morphisms
$\pi_\lambda:\lambda \rightarrow \lambda$,
a family $\beta = (\beta_{\mu,\lambda})$  
of natural isomorphisms $\beta_{\mu,\lambda}: \pi_\mu\, -
\stackrel{\sim}{\Rightarrow} -\,\pi_\lambda$,
and a family
$\xi = (\xi_\lambda)$ of $2$-isomorphisms
$\xi_\lambda:\pi_\lambda^2 \stackrel{\sim}{\Rightarrow}
\unit_\lambda$, such that (assuming $\AA$ is strict):
\begin{itemize}
\item
the pair $(\pi, \beta)$ is an object in the Drinfeld center of
$\AA$, i.e. the properties from
Lemma~\ref{fish}(i)--(ii) hold;
\item $(\beta_{\lambda,\lambda})_{\pi_\lambda} =
-1_{\pi_\lambda^2}$;
\item
$\xi_\mu F \xi_\lambda^{-1} = 
(\beta_{\mu,\lambda})_F \pi_\lambda 
\circ
\pi_\mu (\beta_{\mu,\lambda})_F$
in $\Hom_{\AA}(\pi_\mu^2 F,F \pi_\lambda^2)$
for all $1$-morphisms
$F:\lambda\rightarrow \mu$.
\end{itemize}
Using the second two of these properties, we get that
$\xi_\mu \pi_\mu = \pi_\mu \xi_\mu$ in $\Hom_{\AA}(\pi_\mu^3,
\pi_\mu)$.
Hence, each of the morphism categories
$\mathcal{H}om_{\AA}(\lambda,\mu)$
in a $\Pi$-2-category is itself a $\Pi$-category, with $\Pi := \pi_\mu
-$ and $\xi := \xi_\mu -$.

(ii) A {\em $\Pi$-2-functor}
between two 
$\Pi$-2-categories $\AA$ and $\BB$
is a $\k$-linear 
2-functor
$\mathbb{R}:\AA \rightarrow \BB$ with its usual coherence maps $c$ and
$i$, plus an additional family of 2-isomorphisms
$j:\pi_{\mathbb{R} \lambda} \stackrel{\sim}{\Rightarrow} \mathbb{R}
\pi_\lambda$
for each $\la\in\ob\AA$, such that the following commute (assuming
$\AA$ and $\BB$ are strict):
$$
\begin{tikzcd}
&(\mathbb{R} \pi_\mu)\:(\mathbb{R}\:-)
\arrow[dr,"c"] \\
\pi_{\mathbb{R}\mu} (\mathbb{R}\:-)
\arrow[dd,"\beta_{\mathbb{R}\mu,\mathbb{R}\lambda}",swap]\arrow[ur,"j\:(\mathbb{R} -)"]&
&\mathbb{R}(\pi_\mu\:-)\arrow[dd,"\mathbb{R} \beta_{\mu,\lambda}"]\\\\
(\mathbb{R}\:-) \pi_{\mathbb{R}\lambda}\arrow[dr,"(\mathbb{R}\,-)
j",swap]&&\mathbb{R}(-\:\pi_\lambda)\\
&(\mathbb{R}\:-) \mathbb{R}\pi_\lambda 
\arrow[ur,"c",swap]
\end{tikzcd},
$$
$$
\:
\begin{tikzcd}
\pi^2_{\mathbb{R} \lambda}
\arrow[r,"jj"]\arrow[d,"\xi_{\mathbb{R}\lambda}",swap] 
&(\mathbb{R}
\pi_\lambda)^2
\arrow[r,"c"]
&\mathbb{R}( \pi_\lambda^2)
\arrow[d,"\mathbb{R}\xi_\lambda"]\\
\unit_{\mathbb{R}\lambda}
\arrow[rr,"i",swap]
&&
\mathbb{R} \unit_\lambda
\end{tikzcd}.
$$
A $\Pi$-2-functor is {\em strict} if its coherence maps $c, i$ and $j$
are identities.

(iii) A {\em $\Pi$-2-natural transformation}
$(X,x):\mathbb{R} \Rightarrow \mathbb{S}$
between two $\Pi$-2-functors $\mathbb{R}, \mathbb{S}:\AA \rightarrow
\BB$
is a 2-natural transformation as usual, with one additional 
coherence axiom:
$$
\begin{tikzcd}
\pi_{\mathbb{S}\lambda} X_\lambda
\arrow[d,swap,"j X_\lambda"]
\arrow[rr,"(\beta_{\mathbb{S}\lambda,\mathbb{R}\lambda})_{X_\lambda}"]
&&X_\lambda \pi_{\mathbb{R}\lambda}
\arrow[d,"X_\lambda j"]\\
(\mathbb{S} \pi_\lambda ) X_\lambda
&&
\arrow[ll,"(x_{\lambda,\lambda})_{\pi_\lambda}",swap]
 X_\lambda (\mathbb{R} \pi_\lambda)
\end{tikzcd}.
$$
\end{definition}

The basic example of a strict $\Pi$-2-category is $\piCAT$: objects
are
\smallcat $\Pi$-categories, $1$-morphisms are $\Pi$-functors, and
$2$-morphisms are $\Pi$-natural transformations.
We define the additional data $\pi, \beta$ and $\xi$ 
so that
$\pi_{\A} := \Pi_\A$ for each $\Pi$-category $\A$,
and $\xi$ and $\beta$ come from the natural transformations of 
Definition~\ref{jazz}(i)--(ii).

For a $2$-supercategory $\AA$, the {\em underlying $2$-category} $\underline{\AA}$ is
the $2$-category with the same objects as $\AA$, morphism categories that
are the underlying categories of the morphism supercategories in
$\AA$, and horizontal composition that is the restriction of
the one in $\AA$.
If $(\AA, \pi, \zeta)$ is a (strict) $\Pi$-2-supercategory,
Lemma~\ref{fish} shows how
to define $\beta$ and $\xi$ making $(\underline{\AA}, \pi,
\beta, \xi)$ into a $\Pi$-2-category.
In particular, starting from the $\Pi$-2-supercategory $\piSCAT$, we see
that
$\Pi$-$\underline{\SCAT}$ is a $\Pi$-2-category.

Now we upgrade the functors $E_1$ and $D_1$ from
(\ref{upgrade})--(\ref{D}) to strict
$\Pi$-2-functors
\begin{equation}\label{lun}
\mathbb{E}_1:
{\Pi\text{-}\underline{\SCAT}}
\rightarrow \piCAT,
\qquad
\mathbb{D}_1:
\piCAT
\rightarrow
{\Pi\text{-}\underline{\SCAT}}.
\end{equation}
These agree with $E_1$ and $D_1$ on objects and 1-morphisms.
On 2-morphisms, $\mathbb{E}_1$ sends an even supernatural transformation
$x:F \Rightarrow G$ to $\underline{x}:\underline{F}\Rightarrow \underline{G}$ defined from ${\underline{x}}_\lambda :=
x_\lambda$, which is a $\Pi$-natural transformation thanks to
Corollary~\ref{pi}(iii).
In the other direction, 
$\mathbb{D}_1$ sends a $\Pi$-natural transformation $y:F \Rightarrow G$
to $\hat y:\widehat{F}\Rightarrow\widehat{G}$ 
defined from ${\hat y}_\lambda := y_\lambda$.
In order to check that $\hat y$ is an even supernatural transformation,
the subtle point is to show that ${\hat y}_\mu \circ \widehat{F} \hat f =
\widehat{G} \hat f \circ {\hat y}_\lambda$ 
for an odd morphism $\hat f:\lambda \rightarrow \mu$ 
coming from $f:\lambda
\rightarrow \Pi_\A \mu$ in $\A$,
i.e. $\Pi_\B y_\mu \circ (\beta_F)_\mu^{-1} \circ Ff = (\beta_G)_\mu^{-1}
\circ Gf \circ y_\lambda$. This follows from the property
$\beta_G \circ \Pi_\B y = y \Pi_\A\circ \beta_F$ from
Definition~\ref{jazz}(iii), plus the fact that $y_{\Pi_\A \mu} \circ F f =
Gf \circ y_\lambda$ by the naturality of $y$.
The following strengthens Lemma~\ref{maineq} by taking natural
transformations into account:

\begin{theorem}\label{pro}
The strict $\Pi$-2-functors $\mathbb{D}_1$ and $\mathbb{E}_1$ from
(\ref{lun}) give
mutually inverse $\Pi$-2-equivalences between
$\piCAT$ and $\Pi$-$\underline{\SCAT}$.
\end{theorem}

\begin{proof}
We have that $\mathbb{E}_1 \circ \mathbb{D}_1 = \mathbb{I}_{\piCAT}$.
Conversely, we show that $\mathbb{D}_1 \circ \mathbb{E}_1$ is isomorphic
(not merely equivalent!)
to
$\mathbb{I}_{\Pi\text{-}\underline{\SCAT}}$
in the 2-category $\Pi\text{-}\underline{\SCAT}$
by producing a $\Pi$-2-natural isomorphism $$(T,t):\mathbb{D}_1 \circ
\mathbb{E}_1 \stackrel{\sim}{\Rightarrow}
\mathbb{I}_{\Pi\text{-}\underline{\SCAT}}.
$$
Thus, we need to supply supercategory isomorphisms
$T_\A:\widehat{(\underline{\A})} \stackrel{\sim}{\rightarrow} \A$
and even supernatural isomorphisms
$(t_{\B,\A})_F : T_\B \widehat{(\underline{F})}
\stackrel{\sim}{\Rightarrow} F T_\A$
for all $\Pi$-supercategories and superfunctors
$F:\A \rightarrow \B$.
The isomorphisms $T_\A$ have already been defined in the proof of
Lemma~\ref{maineq}.
Also, in the last paragraph of that proof,
we observed that $T_\B \widehat{(\underline{F})} = F T_\A$. So we can
simply take each $(t_{\B,\A})_F$ to be the identity.
To see that $t_{\B,\A}$ is natural, one needs to observe that
$x T_\A = T_\B \widehat{(\underline{x})}$ for all even supernatural
isomorphisms $x:F \Rightarrow G$. The only other non-trivial check
required is for the coherence axiom of Definition~\ref{clean}(iii). For
this, we must show that $(\beta_{\A,
  \widehat{(\underline{\A})}})_{T_\A}$ is the identity 
for each $\Pi$-supercategory $\A$.
This amounts to
checking that the natural transformations
$\zeta_\A T_\A$ and $T_\A
\zeta_{\widehat{(\underline{\A})}}$ are equal. By definition, on an
object $\lambda$,
$\zeta_{\widehat{(\underline{\A})}}$ is the odd morphism
$\hat{1}_{\Pi_\A \lambda}:\Pi_\A \lambda \rightarrow \lambda$ 
in $\widehat{(\underline{\A})}$ associated to the identity morphism
$1_{\Pi_\A \lambda}$.
Hence, 
according to the definition from the first
paragraph of the proof of Lemma~\ref{maineq},
$(T_\A
\zeta_{\widehat{(\underline{\A})}})_\lambda = 
T_\A \hat{1}_{\Pi_\A \lambda}
=(\zeta_\A)_\lambda = (\zeta_\A T_\A)_\lambda,$
as required.
\end{proof}

Recall that $\piTSCat$ is the category of $\Pi$-2-supercategories and
2-superfunctors. Also let $\piTCat$ denote the category of
$\Pi$-2-categories and $\Pi$-2-functors.
There is a functor
\begin{equation}
E_2:\piTSCat \rightarrow \piTCat,\qquad
(\AA, \pi, \zeta) \mapsto (\underline{\AA}, \pi, \beta, \xi),
\quad
\mathbb{R} \mapsto \underline{\mathbb{R}}.
\end{equation}
We've already defined the effect of this on $\Pi$-2-supercategories.
On a 2-superfunctor $\mathbb{R}:\AA \rightarrow \BB$, we define $\underline{\mathbb{R}}$
to be
the same function as $\mathbb{R}$ on objects and the underlying
functor to $\mathbb{R}$ on morphism categories. The coherence maps $c$
and $i$ restrict in an obvious way to give coherence maps for
$\underline{\mathbb{R}}$. We also need the additional coherence map
$j:\pi_{\mathbb{R}\lambda} \stackrel{\sim}{\Rightarrow} \mathbb{R}
\pi_\lambda$, which is defined so that the following diagram commutes:
$$
\begin{tikzcd}
\pi_{\mathbb{R}\lambda}
\arrow[d,swap,"\zeta_{\mathbb{R} \lambda}"]
\arrow[r,"j"]
&\mathbb{R} \pi_\lambda
\arrow[d,"\mathbb{R} \zeta_\lambda"]\\
\unit_{\mathbb{R}\lambda}\arrow[r,"i",swap]
&
\mathbb{R} \unit_{\lambda}
\end{tikzcd}.
$$
Now one has to check that the two axioms from Definition~\ref{clean}(ii)
are satisfied.
The first of these is a consequence of the second two diagrams from
Definition~\ref{ms}(ii) plus the definition of $\beta$.
For the second one, we have by the super interchange law that
\begin{align*}
(
(\mathbb{R} \zeta_\lambda)
(\mathbb{R} \zeta_\lambda)
)
\circ jj 
= 
((\mathbb{R} \zeta_\lambda)\circ  j)
((\mathbb{R} \zeta_\lambda)\circ  j)
&= 
(i \circ \zeta_{\mathbb{R} \lambda})
(i \circ \zeta_{\mathbb{R} \lambda})
=
ii \circ (\zeta_{\mathbb{R} \lambda}\zeta_{\mathbb{R} \lambda}).
\end{align*}
Also, by the naturality of $c$, we have that
$\mathbb{R} (\zeta_\lambda \zeta_\lambda) \circ c = c \circ ((\mathbb{R}
\zeta_\lambda) (\mathbb{R} \zeta_\lambda))$.
Putting these together gives 
$\mathbb{R} (\zeta_\lambda \zeta_\lambda) \circ c \circ jj
=
c \circ ii \circ
(\zeta_{\mathbb{R}\lambda}\zeta_{\mathbb{R}\lambda})$, and the
conclusion follows easily.

In the other direction, we define a functor
\begin{equation}
D_2:\piTCat
\rightarrow \piTSCat,
\qquad
(\AA,\pi,\beta,\xi) \mapsto
(\widehat{\AA}, \pi, \zeta),
\quad
\mathbb{R} \mapsto \widehat{\mathbb{R}}
\end{equation}
as follows.
The 2-supercategory $\widehat{\AA}$
has the same objects as $\AA$. Its morphism supercategories
$\mathcal{H}om_{\widehat{\AA}}(\lambda,\mu)$
arise as associated $\Pi$-supercategories to the morphism categories
$\mathcal{H}om_{\AA}(\lambda,\mu)$.
Thus the 1-morphisms in $\widehat{\AA}$ are the same as in
$\AA$, while
for 1-morphisms $F, G:\lambda \rightarrow \mu$ we have that
$\Hom_{\widehat{\AA}}(F,G)_{\0} := 
\Hom_{\AA}(F,G)$ and
$\Hom_{\widehat{\AA}}(F,G)_{\1} := 
\Hom_{\AA}(F,\pi_\mu G)$. To describe horizontal and
vertical composition in $\widehat{\AA}$, we assume to simplify the exposition that $\AA$ is strict.
Then vertical composition in
$\widehat{\AA}$ is induced
by that of $\AA$ (using $\xi$ when composing two odd
$2$-morphisms).
Horizontal composition of 1-morphisms in $\widehat{\AA}$ is the same as in $\AA$;
the horizontal composition $\hat y \hat x$ of homogeneous
$2$-morphisms $\hat x:F \Rightarrow H$ and
$\hat y:G \Rightarrow K$ for $F, H:\lambda \rightarrow \mu$ and $G,
K:\mu\rightarrow\nu$
in $\widehat{\AA}$ is defined as follows:
\begin{itemize}
\item
if they are both even, so $\hat  x= x$ and $\hat y = y$ for morphisms
$x:F\Rightarrow H$ and $y:G \Rightarrow K$ in $\AA$,
we define $\hat y \hat x$ to be the horizontal
composition $yx:GF \Rightarrow
KH$ in $\AA$; 
\item
if $\hat x$ is even and $\hat y$ is odd, so $\hat x = x$ and $\hat y =
y$ for $x:F \Rightarrow H$ and $y:G \Rightarrow \pi_\nu K$ in
$\AA$, we let $\hat y \hat x$ be
the horizontal composition $yx:GF\Rightarrow \pi_\nu K H$ in $\AA$ viewed as an odd 2-morphism $GF \Rightarrow KH$ in
$\widehat{\AA}$;
\item
if $\hat y$ is even and $\hat x$ is odd, so $\hat y = y$ and $\hat x = x$ for
$x:F \Rightarrow \pi_\mu H$ and $y:G \Rightarrow K$,
we let $\hat y \hat x$ be
$(\beta_{\nu,\mu})^{-1}_K H \circ y x: G F \Rightarrow K \pi_\mu H
\Rightarrow \pi_\nu K H$;
\item
if both are odd, so $\hat x = x$ and $\hat y = y$ for 
$x:F \Rightarrow
\pi_\mu H$ and $y:G \Rightarrow \pi_\nu K$,
 we let $\hat y \hat x$ be
$-\xi_\nu KH \circ \pi_\nu (\beta_{\nu,\mu})^{-1}_{K} H \circ yx:GF \Rightarrow  \pi_\nu K \pi_\mu H \Rightarrow \pi_\nu^2 KH
\Rightarrow KH$.
\end{itemize}
We leave it as an instructive exercise for the reader to check the
super interchange law using the axioms from Definition~\ref{clean}(i); 
see also \cite[(2.44)--(2.45)]{EL} for helpful pictures. To make
$\widehat{\AA}$ into a $\Pi$-2-supercategory,
we already have the required data $\pi = (\pi_\lambda)$, and
we get $\zeta = (\zeta_\lambda)$ by defining
$\zeta_\lambda:\pi_\lambda \Rightarrow \unit_\lambda$ to be 
$1_{\pi_\lambda}$ viewed as an odd $2$-isomorphism in $\widehat{\AA}$.  

To complete the definition of $D_2$, 
we still need to define the 2-superfunctor
$\widehat{\mathbb{R}}:\widehat{\AA} \rightarrow \widehat{\BB}$ given a
$\Pi$-2-functor
$\mathbb{R}:\AA \rightarrow \BB$. For simplicity, we assume
that $\AA$ and $\BB$ are strict.
Then $\widehat{\mathbb{R}}$ is the same as $\mathbb{R}$ on objects and 1-morphisms. 
On an even 2-morphism $\hat x:F \Rightarrow G$, coming from $x:F
\Rightarrow G$ in $\AA$, we let $\widehat{\mathbb{R}} \hat x$ be the even
2-morphism
associated to $\mathbb{R} x: \mathbb{R} F \Rightarrow \mathbb{R} G$.
On an odd 2-morphism $\hat x:F \Rightarrow G$, coming from $x:F
\Rightarrow \pi_\mu G$, we let $\widehat{\mathbb{R}} \hat x$ be the 
odd 2-morphism associated to the composition
$j^{-1} (\mathbb{R} G) \circ c^{-1} \circ \mathbb{R} x: \mathbb{R} F \Rightarrow \mathbb{R} (\pi_\mu G)
\Rightarrow (\mathbb{R} \pi_\mu) (\mathbb{R} G)
\Rightarrow \pi_{\mathbb{R} \mu} (\mathbb{R} G)$.
We take the coherence maps $c$ and $i$ for $\widehat{\mathbb{R}}$ that
are defined
by the same data as $c$ and $i$ for $\mathbb{R}$.
As usual, there are various checks to be made:
\begin{itemize}
\item To see that $\widehat{\mathbb{R}}$ is a well-defined functor on
  morphism supercategories, one needs to check that
  $\widehat{\mathbb{R}} (\hat y \circ \hat x)
= \widehat{\mathbb{R}} \hat y \circ \widehat{\mathbb{R}} \hat
x$
for
 $\hat x:F \Rightarrow G, \hat y:G \Rightarrow H$ and
$F,G,H:\lambda\rightarrow \mu$.
This is immediate if $\hat x$ is even. If $\hat x$ is odd, it comes
from some 2-morphism $x:F \Rightarrow \pi_\mu G$ in $\AA$.
Suppose $\hat y$ is even, coming from $y :G \Rightarrow H$ in $\AA$.
Then we need to show that $$\qquad
j^{-1}(\mathbb{R} H) \circ c^{-1} \circ \mathbb{R} (\pi_\mu y)
\circ \mathbb{R} x = \pi_{\mathbb{R}\mu}(\mathbb{R} y) \circ j^{-1} (\mathbb{R}G)
\circ c^{-1} \circ \mathbb{R} x.$$
This follows by the commutativity of the following hexagon of
2-morphisms in $\BB$:
$$
\begin{tikzcd}
&
(\mathbb{R} \pi_\mu) (\mathbb{R} G)
\arrow[ddd,"(\mathbb{R} \pi_\mu) (\mathbb{R} y)",swap]
\arrow[dr,"c"]\\
\pi_{\mathbb{R}\mu} (\mathbb{R} G)
\arrow[d,swap,"\pi_{\mathbb{R}\mu} (\mathbb{R} y)"]
\arrow[ur,"j (\mathbb{R} G)"]
&
&
\mathbb{R} (\pi_\mu G)
\arrow[d,"\mathbb{R} (\pi_\mu y)"]
\\
\pi_{\mathbb{R} \mu} (\mathbb{R} H)
\arrow[dr,"j (\mathbb{R} H)",swap]
&&\mathbb{R}(\pi_\mu H)\\
&(\mathbb{R} \pi_\mu)(\mathbb{R} H)\arrow[ur,"c",swap]
\end{tikzcd}.
$$
To see this, note the left hand square commutes by the interchange
law, and the right hand square commutes by naturality of $c$.
The case that $\hat y$ is odd, coming from $y:G \Rightarrow
\mu_\mu H$, is similar but a little more complicated; ultimately, it
depends on the second coherence axiom from Definition~\ref{clean}(ii).
\item
To see that $c$ is a supernatural transformation, one needs to check
that 
$$
\begin{tikzcd}
(\widehat{\mathbb{R}} G) (\widehat{\mathbb{R}} F)
\arrow[d,"(\widehat{\mathbb{R}}\hat y)(\widehat{\mathbb{R}}\hat
x)",swap]
\arrow[r,"c"]&\widehat{\mathbb{R}}(GF)\arrow[d,"\widehat{\mathbb{R}}(\hat y \hat x)"]
\\
(\widehat{\mathbb{R}} K)(\widehat{\mathbb{R}} H)\arrow[r,"c",swap]&
\widehat{\mathbb{R}} (KH)
\end{tikzcd}
$$
commutes.
We leave this lengthy calculation to the reader, just noting when $\hat x$ is odd that it
depends also on the first coherence axiom from Definition~\ref{clean}(ii).
\end{itemize}
The proof of the next lemma is similar to the proof of Lemma~\ref{maineq}.
Note also that the remaining part of Theorem~\ref{bop} from the introduction
follows from this result (on restricting to 2-(super)categories with one object).

\begin{lemma}\label{groups}
The functors $D_2$ and $E_2$ are mutually inverse equivalences between
the categories
$\piTCat$ and $\piTSCat$.
\end{lemma}

\begin{proof}
We first observe that $E_2 \circ D_2 = I_{\piTCat}$. To show that $D_2 \circ E_2\cong I_{\piTSCat}$,
we have to define a natural isomorphism $\mathbb{T}:D_2 \circ E_2
\stackrel{\sim}{\Rightarrow} I_{\piTSCat}$.
So for each $\Pi$-2-supercategory $(\AA, \pi, \zeta)$, we need
to produce 
an isomorphism of 2-supercategories $\mathbb{T}_\AA:\widehat{{(}\underline{\AA}{)}}
\stackrel{\sim}{\rightarrow}
\AA$.
This is the identity on objects and
$1$-morphisms.
On a homogeneous 2-morphism $\hat x:F \Rightarrow G$ 
between 1-morphisms
$F, G:\lambda \rightarrow \mu$ in $\widehat{(\underline{\AA})}$, 
we let $\mathbb{T}_\AA \hat x := x$ if $\hat x$ is even
coming from $x:F \Rightarrow G$ in $\AA$,
or $\mathbb{T}_\AA \hat x := \zeta_\mu G \circ x$ if $\hat x$ is odd coming
from $x:F \Rightarrow \pi_\mu G$ in $\AA$.
Since $\mathbb{T}_\AA$ is 
clearly bijective on 2-morphisms, it will certainly be a
2-isomorphism, but we still need to verify that it is indeed a
well-defined 2-superfunctor, i.e. we need to show that it respects
horizontal and vertical composition of 2-morphisms.
In the next paragraph, we go through the details of this in the most interesting
situation when both 2-morphisms are odd (also assuming $\AA$ is strict to simplify notation).

For vertical composition, take $F, G, H:\lambda \rightarrow \mu$
and odd 2-morphisms $\hat x:F \Rightarrow G$, $\hat y:G \Rightarrow H$ in
$\widehat{(\underline{\AA})}$ coming from 
$x:F \Rightarrow \pi_\mu G$, $y:G \Rightarrow \pi_\mu H$ in
$\AA$.
The vertical composition $\hat y \circ\hat x$ in
$\widehat{(\underline{\AA})}$ is by definition the 
composition $\xi_\mu H \circ \pi_\mu y \circ x$ in $\AA$.
We need to show that this is equal to $\zeta_\mu H \circ y \circ
\zeta_\mu G \circ x$:
$$
\zeta_\mu H \circ y \circ
\zeta_\mu G \circ x =
- \zeta_\mu H \circ \zeta_\mu \pi_\mu H \circ \pi_\mu y \circ x =
\xi_\mu H \circ \pi_\mu y \circ x.
$$
For horizontal composition, take $F, H:\lambda \rightarrow \mu,
G, K:\mu \rightarrow \nu$ and odd 2-morphisms $\hat x:F \Rightarrow H,
\hat y:G
\Rightarrow K$
coming from $x:F \Rightarrow \pi_\mu H$, $y:G \Rightarrow \pi_\nu K$. 
Recalling that $(\beta_{\nu,\mu})^{-1}_K = \zeta^{-1}_\nu K \zeta_\mu$, we
have that
$\zeta_\nu K \circ (\beta_{\nu,\mu})^{-1}_K = K \zeta_\mu$, hence
$\xi_\nu K \circ \Pi_\nu (\beta_{\nu,\mu})^{-1}_K  = \zeta_\nu K
\zeta_\mu$.
We deduce that
$$
-\xi_\nu KH \circ \Pi_\nu (\beta_{\nu,\mu})^{-1}_K H \circ yx
=- \zeta_\nu K \zeta_\mu H \circ yx
=
(\zeta_\nu K \circ
y)(\zeta_\mu H \circ x),
$$
establishing that
$\Phi(\hat y\hat x) =  \Phi(\hat y) \Phi(\hat x).$

To complete the proof we need to check naturality:
we have that $\mathbb{R} \mathbb{T}_\AA = \mathbb{T}_\BB
\widehat{(\underline{\mathbb{R}})}$
for each 2-superfunctor $\mathbb{R}:\AA\rightarrow\BB$ between
$\Pi$-2-supercategories $\AA$ and $\BB$.
The only tricky point is to see 
that they are equal on an odd 2-morphism
$\hat x:F \Rightarrow G$ in $\widehat{(\underline{\AA})}$ 
coming from $x:F \Rightarrow \pi_\mu G$ in $\underline{\AA}$.
For this, one needs to use the last of the unit axioms from
Definition~\ref{ms}(ii) plus the definition of $j$.
\end{proof}

Finally, we upgrade $E_2$ and $D_2$ to strict 2-functors
\begin{equation}
\mathbb{E}_2:\piTSCAT \rightarrow \piTCAT,
\qquad
\mathbb{D}_2:\piTCAT \rightarrow \piTSCAT.
\end{equation}
We take $\mathbb{E}_2$ to be equal to $E_2$ 
on objects and
1-morphisms.
On 2-morphisms, $\mathbb{E}_2$ sends 2-natural transformation
$(X,x):\mathbb{R} \Rightarrow \mathbb{S}$
to $(\underline{X}, \underline{x}):\underline{\mathbb{R}} \Rightarrow
\underline{\mathbb{S}}$
defined by $\underline{X}_\lambda := X_\lambda$ and
$\underline{x}_{\mu,\lambda} := x_{\mu,\lambda}$.
To check the coherence axiom from Definition~\ref{clean}(iii), we
need to check that the outside square in the following diagram
commutes:
$$
\begin{tikzcd}
\pi_{\mathbb{S}\lambda}\arrow[dr,"\zeta_{\mathbb{S}\lambda} X_\lambda"]
\arrow[ddd,"j X_\lambda",swap]X_\lambda\arrow[rrrr,"(\beta_{\mathbb{S}\lambda,\mathbb{R}\lambda})_{X_\lambda}"]&&&&X_\lambda
\pi_{\mathbb{R}\lambda}
\arrow[ddd,"X_\lambda j"]\arrow[dl,"X_\lambda \zeta_{\mathbb{R}\lambda}",swap]\\
&\unit_{\mathbb{S}\lambda} X_\lambda\arrow[d,"i X_\lambda",swap]\arrow[r,"l"]&X_\lambda\arrow[r,"r^{-1}"]&X_\lambda \unit_{\mathbb{R}\lambda}\arrow[d,"X_{\lambda} i"]&\\
&(\mathbb{S} \unit_\lambda)
X_\lambda&&\arrow[ll,"(x_{\lambda,\lambda})_{\unit_\lambda}"] X_\lambda
(\mathbb{R} \unit_\lambda)&
\\
(\mathbb{S}\pi_\lambda)
X_\lambda\arrow[ur,swap,"(\mathbb{S} \zeta_\lambda)
X_\lambda"]&&&&\arrow[llll,"(x_{\lambda,\lambda})_{\pi_\lambda}"]
X_\lambda(\mathbb{R}\pi_\lambda)\arrow[ul,"X_\lambda (\mathbb{R} \zeta_\lambda)"]
\end{tikzcd}.
$$
This follows because the other five faces commute: the middle square
by Definition~\ref{ms}(iii), the left and right squares by definition of
$j$, the top square by definition of $\beta$, and the bottom square by
naturality of $x_{\lambda,\lambda}$.

In the other direction, the strict 2-functor $\mathbb{D}_2$ is the
same as $D_2$ on objects and 1-morphisms.
It sends
$\Pi$-2-natural transformation
$(Y,y):\mathbb{R} \Rightarrow \mathbb{S}$
to $(\widehat{Y}, \hat y):\widehat{\mathbb{R}} \Rightarrow
\widehat{\mathbb{S}}$
defined by $\widehat{Y}_\lambda := Y_\lambda$ and $\hat
y_{\mu,\lambda} := y_{\mu,\lambda}$.
The content here is to check the supernaturality of
$y_{\mu,\lambda}$ on an odd 2-morphism $\hat x: F \Rightarrow G$, so
$F, G$ are 1-morphisms $\lambda \rightarrow \mu$ and $\hat x$ is
the odd 2-morphism associated  to a 2-morphism $x:F \Rightarrow
\pi_\mu G$. We need to show that $(\widehat{\mathbb{S}} \hat x)
\widehat{Y}_\lambda \circ (\hat y_{\mu,\lambda})_F = (\hat
y_{\mu,\lambda})_G \circ \widehat{Y}_\mu
(\widehat{\mathbb{R}}\hat x)$, which amounts to checking the
commutativity of the outside
of the following diagram:
$$
\begin{tikzcd}
\arrow[d,swap,"Y_\mu (\mathbb{R} x)"] Y_\mu (\mathbb{R}
F)\arrow[rr,"(y_{\mu,\lambda})_F"]&& (\mathbb{S} F)
Y_\lambda\arrow[d,"(\mathbb{S} x) Y_\lambda"]\\
Y_\mu(\mathbb{R}(\pi_\mu G)) \arrow[rr,"(y_{\mu,\lambda})_{\pi_\mu
  G}"]\arrow[d,swap,"Y_\mu c^{-1}"]&& (\mathbb{S}(\pi_\mu G))
Y_\lambda\arrow[drr,"c^{-1} Y_\lambda"]\\
\arrow[d,swap,"Y_\mu j^{-1} (\mathbb{R} G)"]Y_\mu (\mathbb{R} \pi_\mu) (\mathbb{R}
G)\arrow[rr,"(y_{\mu,\mu})_{\pi_\mu} (\mathbb{R} G)"]&&(\mathbb{S}
\pi_\mu) Y_\mu (\mathbb{R} G)\arrow[swap,rr,"(\mathbb{S} \pi_\mu)(y_{\mu,\lambda})_G"]
\arrow[d,swap,"j^{-1} Y_\mu (\mathbb{R} G)"]&&
(\mathbb{S} \pi_\mu) (\mathbb{S} G) Y_\lambda\arrow[d,"j^{-1}
(\mathbb{S} G) Y_\lambda"]\\
Y_\mu \pi_{\mathbb{R} \mu} (\mathbb{R}
G)\arrow[rr,swap,"(\beta_{\mathbb{S}\mu,\mathbb{R}\mu})^{-1}_{Y_\mu}
(\mathbb{R} G)"]&&\pi_{\mathbb{S}\mu} Y_\mu
(\mathbb{R} G)\arrow[rr,swap,"\pi_{\mathbb{S}\mu} (y_{\mu,\lambda})_G"]&&\pi_{\mathbb{S}\mu} (\mathbb{S} G) Y_\lambda
\end{tikzcd}.
$$
Now we observe that the top square commutes by naturality of
$y_{\mu,\lambda}$;
the pentagon commutes by the first axiom from
Definition~\ref{ms}(iii) (we are assuming strictness as
usual);
the bottom left square commutes by the axiom from
Definition~\ref{clean}(iii);
and the bottom right square commutes by the interchange law.

\begin{theorem}\label{powers}
The strict 2-functors $\mathbb{D}_2$ and $\mathbb{E}_2$ are
mutually inverse 2-equivalences between the strict 2-categories
$\piTCAT$ and $\piTSCAT$.
\end{theorem}

\begin{proof}
This may be deduced from the proof of Lemma~\ref{groups} in a similar way to
how Theorem~\ref{pro} was obtained from the proof of Lemma~\ref{maineq}.
We leave the details to the reader.
\end{proof}

\begin{corollary}
The 2-supercategories
$\piTSCAT$ and $\Pi$-2-$\widehat{\mathfrak{Cat}}$
are 2-superequivalent.
\end{corollary}

\begin{proof}
We've already shown in Theorem~\ref{pro} that $\mathbb{E}_1:
\mathbb{E}_2(\Pi$-$\SCAT)\rightarrow
\piCAT$ is a $\Pi$-2-equivalence.
Now apply $\mathbb{D}_2$ and use Theorem~\ref{powers}.
\end{proof}

\begin{remark}
Like  in Remark~\ref{phone}, one can go a level higher: the strict
3-category of $\Pi$-2-categories, $\Pi$-2-functors, $\Pi$-2-natural
transformations and modifications is 3-equivalent to the strict
3-category of $\Pi$-2-supercategories, 2-superfunctors, 2-natural
transformations and even supermodifications. In particular, this assertion
implies that the monoidal category underlying the Drinfeld center of a
$\Pi$-2-supercategory $\AA$ is monoidally equivalent to the Drinfeld center
of $\underline{\AA}$.
\end{remark}

\section{Gradings}

In the final section, we explain how to incorporate an additional
$\Z$-grading.
Since this is all is very similar to the theory so far (and there are
no additional issues with signs!), we will be quite
brief, introducing suitable language but leaving detailed proofs to the reader.
We continue to assume that $\k$ is a commutative ground
ring\footnote{Actually, everything prior to Definition~\ref{hale} makes sense more generally working
  over a graded commutative 
superalgebra
$\k = \bigoplus_{n
  \in \Z} \k_n = \bigoplus_{n \in \Z} \k_{n,\0}\oplus \k_{n,\1}$.},
so a superspace means a $\Z/2$-graded $\k$-module as before.

By a {\em graded superspace} we mean a $\Z$-graded superspace
$$
V = \bigoplus_{n \in \Z} V_n =
\bigoplus_{n \in \Z} V_{n,\0}\oplus V_{n,\1}.
$$
We stress that the $\Z$- and $\Z/2$-gradings on a graded superspace are independent of
each other.
We denote the degree $n$ of $v \in V_n$ also by $\deg(v)$.
Let $\underline{\GSVec}$ be the
category of graded superspaces 
and degree-preserving even linear maps, i.e. $\k$-module homomorphisms
$f:V \rightarrow W$ such that $f(V_{n,p}) \subseteq W_{n,p}$ for each
$n \in \Z$ and $p \in \Z/2$.
This is a symmetric monoidal category with $(V \otimes W)_n = \bigoplus_{r+s=n}
V_r \otimes W_s$, and the same braiding as in $\SVec$.

\begin{definition}
By a {\em graded supercategory} we mean a category enriched in
$\underline{\GSVec}$. 
A {\em graded superfunctor} between graded supercategories is a superfunctor that preserves
degrees of morphisms. 
A supernatural transformation
$x:F \Rightarrow G$ between graded
superfunctors
$F$ and $G$ is said to be {\em homogeneous of degree $n$} if
$x_\lambda:F \lambda \rightarrow G\lambda$ is of degree $n$ for all
objects $\lambda$. Let $\Hom(F,G)_n$ denote the superspace of all
homogeneous supernatural transformations of degree $n$.
Then a {\em graded supernatural transformation}
from $F$ to $G$
is 
an element of the graded superspace 
$\Hom(F,G):=
\bigoplus_{n \in \Z} \Hom(F,G)_n$.
\end{definition}

If $\A$ is a graded supercategory, the {\em underlying category}
$\underline{\A}$ is the $\k$-linear category with the same objects as $\A$
but only the even morphisms of degree zero.
Here are some basic examples of graded supercategories:
\begin{itemize}
\item
Any graded superalgebra 
$A = \bigoplus_{n \in \Z} A_n = \bigoplus_{n \in \Z} A_{n,\0}\oplus A_{n,\1}$
can be viewed as a graded supercategory
with one object.
\item For graded superalgebras $A$ and $B$, 
let $A\lrGSMod B$
denote the graded
  supercategory 
of {\em graded $(A,B)$-superbimodules} $V = \bigoplus_{n \in \Z} V_n
= \bigoplus_{n \in \Z} V_{n,\0} \oplus V_{n,\1}$.
Morphisms are defined from $\Hom(V, W) := \bigoplus_{n \in \Z}
\Hom(V, W)_n$ where
$\Hom(V,W)_n$ consists of all $(A,B)$-superbimodule homomorphisms
$f:V \rightarrow W$ that are {\em homogeneous of degree $n$}, i.e.
$f(V_m) \subseteq W_{m+n}$ for all $m \in \Z$.
\item Taking $A=B = \k$ in (ii), we get the graded supercategory
$\GSVec$
of graded superspaces.
The underlying category is $\underline{\GSVec}$ as defined above.
\item For graded supercategories $\A$, $\B$, the graded
supercategory $\mathcal{H}om(\A,\B)$ consists of all graded superfunctors and
graded supernatural transformations.
\end{itemize}

Let $\GSCat$ be the category of all
\smallcat graded supercategories and graded superfunctors. 
We make $\GSCat$ into
a monoidal category with tensor product operation $-\boxtimes -$
defined in just the same way as was
explained after Example~\ref{snuggly} in the introduction.

\begin{definition}
A {\em strict graded 2-supercategory} is a category enriched in
$\GSCat$, i.e. it is a 2-supercategory with an additional grading on 2-morphisms
which is respected by both horizontal and vertical composition.
\end{definition}

The basic example of a strict graded 2-supercategory is 
$\GSCAT$: graded supercategories, graded superfunctors
and graded supernatural transformations.
There is also the ``weak'' notion of {\em graded
  2-supercategory}, which is the obvious graded analog of
Definition~\ref{ms}(i). 
For example, there is a graded 2-supercategory 
$\mathfrak{GSBim}$ of
{\em graded superbimodules}, which has objects that are
graded superalgebras, the
morphism supercategories are defined from $\mathcal{H}om_{\mathfrak{GSBim}}(A,B) :=
B\lrGSMod A$, and horizontal composition is defined by tensor product.

Here is the graded analog of Definition~\ref{ms}(ii):

\begin{definition}
For graded 2-supercategories $\AA$ and $\BB$,  a {\em graded
  2-superfunctor} $\mathbb{R}:\AA \rightarrow \BB$ consists of:
\begin{itemize}
\item A function $\mathbb{R}:\ob \AA \rightarrow \ob \BB$.
\item Graded superfunctors
$\mathbb{R}:\mathcal{H}om_{\AA}(\lambda,\mu) \rightarrow
\mathcal{H}om_{\BB}(\mathbb{R}\lambda,\mathbb{R}\mu)$ for $\lambda,\mu \in \ob \AA$.
\item Homogeneous graded supernatural isomorphisms
$c:(\mathbb{R}\:-)\: (\mathbb{R}\:-) \stackrel{\sim}{\Rightarrow}
\mathbb{R}(-\:-)$
that are even of degree zero.
\item Homogeneous 2-isomorphisms
  $i:\unit_{\mathbb{R}\lambda}\stackrel{\sim}{\Rightarrow} \mathbb{R}
  \unit_\lambda$
that are even of degree zero
for all $\lambda \in \ob \AA$.
\end{itemize}
This data should satisfy the same axioms as in
Definition~\ref{ms}(ii).
\end{definition}

We leave it to the reader to formulate the graded versions of
Definition~\ref{ms}(iii) (2-natural transformations between graded
2-superfunctors) and Definition~\ref{ms}(iv) (graded supermodifications).

The next two definitions give the graded analogs of Definitions~\ref{defscat}
and \ref{georgia}.

\begin{definition}
A {\em graded $(Q,\Pi)$-supercategory} is a 
graded supercategory $\A$ plus 
the extra data of graded superfunctors $Q, Q^{-1}, \Pi:\A
\rightarrow \A$, an odd supernatural isomorphism
$\zeta:\Pi \stackrel{\sim}{\Rightarrow} I$ that is homogeneous of degree 0,
and even supernatural isomorphisms $\sigma:Q \stackrel{\sim}{\Rightarrow} I$ and
$\bar\sigma:Q^{-1}
\stackrel{\sim}{\Rightarrow} I$ that are homogeneous of degrees $-1$ and $1$,
respectively.
Note that 
$\ii:= \bar\sigma \sigma :Q^{-1} Q \stackrel{\sim}{\Rightarrow} I$, 
$\jj:= \sigma \bar\sigma :Q Q^{-1} \stackrel{\sim}{\Rightarrow} I$
and $\xi := \zeta \zeta:\Pi^2 \stackrel{\sim}{\Rightarrow} I$ are even isomorphisms of
degree zero, so that $Q$ and $Q^{-1}$ are mutually inverse graded superequivalences,
and $\Pi$ is a self-inverse graded superequivalence.
\end{definition}

For example, for graded superalgebras $A$ and $B$, we can view
$A\lrGSMod B$ as a graded $(Q, \Pi)$-supercategory by defining
$\Pi$ and $\zeta$ as in Example~\ref{psf}, and letting
$Q, Q^{-1}:A\lrGSMod B \rightarrow A\lrGSMod B$
be the upward and downward grading shift functors,
i.e. $(Q V)_n := V_{n-1},
(Q^{-1} V)_n := V_{n+1}$. We
take $\sigma, \bar \sigma$
to be induced by the identity function on the underlying
sets.

\begin{definition}\label{wash}
A {\em graded $(Q,\Pi)$-2-supercategory} is a graded 2-supercategory
$\AA$ plus families $q = (q_\lambda:\lambda\rightarrow
\lambda), q^{-1} =
(q^{-1}_\lambda:\lambda\rightarrow \lambda)$ and $\pi =
(\pi_\lambda:\lambda\rightarrow \lambda)$ of 1-morphisms, and 
families
$\sigma = (\sigma_\lambda:q_\lambda \stackrel{\sim}{\Rightarrow} \unit_\lambda),
\bar\sigma = (\bar\sigma_\lambda:q^{-1}_\lambda \stackrel{\sim}{\Rightarrow}
\unit_\lambda)$ and $\zeta = (\zeta_\lambda:\pi_\lambda \stackrel{\sim}{\Rightarrow} \unit_\lambda)$ 
of 2-isomorphisms which are even, even and odd of degrees $-1$, $1$ and
0, respectively.
\end{definition}

For example, there is a graded $(Q,\Pi)$-2-supercategory $\qpiGSCAT$ 
consisting of all \smallcat graded $(Q, \Pi)$-supercategories,
graded superfunctors and graded supernatural transformations.

\begin{lemma}\label{turkey}
Let $\AA$ be a graded $(Q,\Pi)$-2-supercategory, which we assume is
strict for simplicity.
\begin{itemize}
\item[(i)]
There are families $\beta
=(\beta_{\mu,\lambda}:\pi_\mu - \stackrel{\sim}{\Rightarrow} - \pi_\lambda)$ of even
supernatural isomorphisms of degree zero
and
$\xi = (\xi_\lambda:\pi_\lambda^2 \stackrel{\sim}{\Rightarrow} \unit_\lambda)$ of even
2-isomorphisms of degree zero
defined as in Lemma~\ref{fish}.
They satisfy the properties 
from Definition~\ref{clean}(i).
\item[(ii)]
There is a family $\gamma = (\gamma_{\mu,\lambda}:q_\mu - \stackrel{\sim}{\Rightarrow}
  - q_\lambda)$ of even supernatural isomorphisms of degree zero
defined from $(\gamma_{\mu,\lambda})_F := \sigma_\mu F
  \sigma_\lambda^{-1}$ for a 1-morphism $F:\lambda\rightarrow \mu$.
The pair $(q, \gamma)$ is an invertible object of the Drinfeld center of
$\AA$ with $(\gamma_{\lambda,\lambda})_{q_\lambda} =
1_{q_\lambda^2}$
and $(\gamma_{\lambda,\lambda})_{\pi_\lambda} = (\beta_{\lambda,\lambda})^{-1}_{q_\lambda}$.
\item[(iii)]
There are even $2$-isomorphisms of degree zero
$\ii_\lambda := \bar\sigma_\lambda \sigma_\lambda :q^{-1}_\lambda q_\lambda \stackrel{\sim}{\Rightarrow} \unit_\lambda$ and
$\jj_\lambda:=\sigma_\lambda \bar \sigma_\lambda:q_\lambda q^{-1}_\lambda \stackrel{\sim}{\Rightarrow} \unit_\lambda$.
Moreover 
$q_\lambda \ii_\lambda = \jj_\lambda q_\lambda$ 
and
$\ii_\lambda q^{-1}_\lambda = q^{-1}_\lambda \jj_\lambda$ in
$\Hom_{\AA}(q_\lambda q^{-1}_\lambda q_\lambda, q_\lambda)$ and
$\Hom_{\AA}(q^{-1}_\lambda q_\lambda q^{-1}_\lambda, q^{-1}_\lambda)$, respectively.
\end{itemize}
\end{lemma}

\begin{proof}
Similar arguments to those in the proof of Lemma~\ref{fish}.
\end{proof}

\begin{corollary}\label{pigeon}
Let $\A$ and $\A'$
be graded $(Q,\Pi)$-supercategories.
\begin{itemize}
\item[(i)]
There are even supernatural
isomorphisms of degree zero 
$\xi := \zeta \zeta: \Pi^2 \stackrel{\sim}{\Rightarrow} I$,
$\ii := \bar\sigma\sigma:Q^{-1} Q \Rightarrow I$ and $\jj :=
\sigma\bar\sigma:QQ^{-1} \Rightarrow I$. 
Moreover, we have that
$\xi \Pi = \Pi \xi$, and
$\ii^{-1}$ and $\jj$ define the unit and counit of an adjunction making
$(Q, Q^{-1})$ into an adjoint pair of auto-equivalences of $\A$.
\item[(ii)]
Suppose that $F:\A \rightarrow \A'$ is a graded
superfunctor.
There are even supernatural isomorphisms of degree zero $\beta_F :=
-\zeta' F \zeta^{-1}: \Pi' F \stackrel{\sim}{\Rightarrow} F \Pi$ 
and $\gamma_F:= \sigma' F \sigma^{-1}:Q' F
\stackrel{\sim}{\Rightarrow} F Q$, with $\xi' F \xi^{-1} = \beta_F \Pi
\circ \Pi'
\beta_F$ as in Corollary~\ref{pi}(ii).
Also $\gamma_\Pi = \beta_Q^{-1}$.
\item[(iii)] Suppose that $x:F \Rightarrow G$ is a graded supernatural
  transformation.
Then $\beta_G \circ \Pi' x = x \Pi \circ \beta_F$ as in
Corollary~\ref{pi}(iii). Similarly,
$\gamma_G \circ Q'x = xQ \circ \gamma_G$.
\item[(iv)] 
We have that 
$\beta_{GF} = G \beta_F \circ \beta_G F, \beta_I = 1_\Pi$ and $\beta_\Pi =-1_{\Pi^2}$
as in Corollary~\ref{pi}(iv).
Similarly,
$\gamma_{GF}=G \gamma_F \circ \gamma_G F, \gamma_I = 1_Q$ and $\gamma_Q =1_{Q^2}$.
\end{itemize}
\end{corollary}

\begin{proof}
Everything 
follows by 
applying Lemma~\ref{turkey} to the $(Q,\Pi)$-2-supercategory
$\qpiGSCAT$.
In particular, the assertion in (i) that $\ii^{-1}$ and $\jj$
are the 
unit and counit of an adjunction means that
$Q \ii^{-1}\circ \jj Q: Q \Rightarrow Q Q^{-1} Q \Rightarrow Q$
and $\ii^{-1} Q^{-1} \circ Q^{-1} \jj: Q^{-1} \Rightarrow Q^{-1} Q Q^{-1} \Rightarrow Q^{-1}$ are identities;
this follows because $Q \ii = \jj Q$ and $\ii Q^{-1} = Q^{-1} \jj$.
\end{proof}

\vspace{2mm}

The analog of
Definition~\ref{pienv} in the presence of a grading is as follows.

\begin{definition}\label{qpe}
The {\em $(Q, \Pi)$-envelope} of a graded supercategory $\A$
is the graded $(Q,\Pi)$-supercategory $\A_{q,\pi}$ with objects $\{Q^m
\Pi^a
\lambda\:|\:\lambda \in \ob \A, m \in \Z, a \in \Z/2\}$ and
$$
\Hom_{\A_{q,\pi}}(Q^m \Pi^a \lambda, Q^n \Pi^b \mu)
:= Q^{n-m}\Pi^{a+b} \Hom_{\A}(\lambda,\mu),
$$
where $Q$ and $\Pi$ on the right hand side are the (invertible) grading and parity
shift functors on $\GSVec$.
We denote the morphism $Q^m\Pi^a \lambda \rightarrow  Q^n \Pi^b \mu$
coming from a homogeneous $f:\lambda\rightarrow\mu$ under this
identification by $f^{n,b}_{m,a}$.
Composition in $\A_{q,\pi}$ is defined by
$g^{n,c}_{m,b} \circ f^{m,b}_{l,a} := (g \circ f)^{n,c}_{l,a}$.
The parity-switching functor $\Pi$ and $\zeta$ are
defined as in Definition \ref{pienv}.
The degree shift functors $Q, Q^{-1}$ are given
by $Q (Q^m\Pi^a  \lambda) :=  Q^{m+1}\Pi^a \lambda$,
$Q^{-1}(  Q^m\Pi^a \lambda) :=  Q^{m-1}\Pi^a \lambda$, and $\sigma,
\bar \sigma$ are induced by the identity morphism in $\A$.
\end{definition}

In an analogous way to (\ref{music}),
Definition~\ref{qpe} may be extended to produce a strict graded 2-superfunctor
\begin{equation}
-_{q,\pi}:\GSCAT \rightarrow \qpiGSCAT.
\end{equation}
There is a canonical graded superfunctor $J:\A \rightarrow
\A_{q,\pi}$ which satisfies a universal property similar to
Lemma~\ref{uniprop}.
Also $J$ is a graded superequivalence if and only if 
$\A$ is {\em $(Q,\Pi)$-complete}, meaning that
every object
$\lambda$ of $\A$ is the target of even isomorphisms of degrees $\pm 1$ and
the target of an odd isomorphism of degree $0$; cf. Lemma~\ref{e}.
The analog of Theorem~\ref{2adj} is as follows:

\begin{theorem}
For all graded supercategories $\A$ and graded $(Q,\Pi)$-supercategories $\B$,
there is a functorial graded superequivalence
$\mathcal{H}om(\A, \nu \B)
\rightarrow
\mathcal{H}om(\A_{q,\pi}, \B)$, where
$\nu:\qpiGSCAT\rightarrow\GSCAT$ denotes the obvious
forgetful 2-superfunctor.
Hence, the strict graded 2-superfunctor $-_{q,\pi}$ is left 2-adjoint to $\nu$.
\end{theorem}

Moving on to 2-categories, here is the graded analog of
Definition~\ref{sixo}:

\begin{definition}\label{ed}
The {\em $(Q,\Pi)$-envelope} of a graded 2-supercategory $\AA$
is the $(Q,\Pi)$-2-supercategory $\AA_{q,\pi}$ with the same object
set as $\AA$ and
morphism supercategories 
that are the $(Q,\Pi)$-envelopes of the graded morphism supercategories
in $\AA$.
Thus, the set of 1-morphisms $\lambda \rightarrow \mu$
in $\AA_{q,\pi}$ is
$$
\{Q^m \Pi^a F\:|\:\text{for all 1-morphisms $F:\lambda\rightarrow\mu$ in
$\AA$, $m \in \Z$ and $a \in
\Z/2$}\}.
$$
The
graded superspace of
2-morphisms
$Q^m \Pi^a F \Rightarrow Q^n \Pi^b G$ in $\AA_{q,\pi}$ is defined from
$$
\Hom_{\AA_{q,\pi}}(Q^m \Pi^a F, Q^n \Pi^b G)
:=
Q^{n-m}\Pi^{a+b}\Hom_{\AA}(F, G).
$$
We denote the 2-morphism $Q^m \Pi^a F \Rightarrow Q^n \Pi^b G$ coming from
a homogeneous 2-morphism
$x:F \Rightarrow G$ in $\AA$ under this identification by
$x^{n,b}_{m,a}$.
In the strict case, one might represent $x^{n,b}_{m,a}$
diagrammatically by 
$$
\mathord{
\begin{tikzpicture}[baseline = 0]
	\draw[-,thick,darkred] (0.08,-.4) to (0.08,-.13);
	\draw[-,thick,darkred] (0.08,.4) to (0.08,.13);
      \draw[thick,darkred] (0.08,0) circle (4pt);
   \node at (0.08,0) {\color{darkred}$\scriptstyle{x}$};
   \node at (0.45,0.03) {$\scriptstyle{\lambda}$};
   \node at (-0.32,0) {$\scriptstyle{\mu}$};
   \node at (0.05,0.58) {$\scriptstyle{G}$};
   \node at (0.05,-0.55) {$\scriptstyle{F}$};
\draw[-,thin,red](.4,-.4) to (-.24,-.4);
\draw[-,thin,red](.4,.4) to (-.24,.4);
\node at (.5,.4) {$\color{red}\scriptstyle b$};
\node at (.5,-.4) {$\color{red}\scriptstyle a$};
\node at (-.4,.4) {$\color{red}\scriptstyle n$};
\node at (-.4,-.4) {$\color{red}\scriptstyle m$};
\end{tikzpicture}
}.
$$
This is of  parity $|x|+a+b$ and degree $\deg(x)+n-m$
(where $|x|$ and $\deg(x)$ denote the parity and degree of $x$ in
$\AA$).
Vertical composition 
is defined from 
$$
y^{n,c}_{m,b} \circ x^{m,b}_{l,a} := (y \circ x)^{n,c}_{l,a}.
$$
Horizontal composition of 1-morphisms
is defined by 
$$
(Q^n \Pi^b G)(Q^m \Pi^a F) := Q^{m+n} \Pi^{a+b} (GF)
$$
and 2-morphisms by
$$
y^{l,d}_{n,b} x^{k,c}_{m,a} := (-1)^{b|x|+|y|c+bc+ab}(yx)^{k+l,c+d}_{m+n,a+b}.
$$
Finally, $q, q^{-1}$ and $\pi$
are given by $q_\lambda := Q^1 \Pi^{\0}\unit_\lambda, q^{-1}_\lambda := Q^{-1}\Pi^{\0} \unit_\lambda$
and $\pi_\lambda := Q^0\Pi^{\1} \unit_\lambda$; the 2-morphisms
$\sigma_\lambda, \bar\sigma_\lambda$ and $\xi_\lambda$ are induced by $1_{\unit_\lambda}$.
\end{definition}

Again, 
there is a canonical strict 2-superfunctor
$\mathbb{J}:\AA \rightarrow \AA_{q,\pi}$,
which is a graded 2-superequivalence if and only
if  $\AA$ is {\em $(Q,\Pi)$-complete}, meaning that for each
$\lambda\in\ob\AA$ it
possesses 1-morphisms $q_\pi^{\pm}:\lambda\rightarrow\lambda$ and $\pi_\lambda:\lambda\rightarrow \lambda$,
and homogeneous 2-isomorphisms 
$q_\pi^{\pm} \stackrel{\sim}{\Rightarrow} \unit_\lambda$
that are even of degrees $\mp 1$,
and $\pi_\lambda \stackrel{\sim}{\Rightarrow}\unit_\lambda$
that is odd of degree 0.
Like in (\ref{massage}), one can extend Definition~\ref{ed} to obtain
a strict 2-functor
\begin{equation}\label{macky}
-_{q,\pi}:\TGSCAT \rightarrow \qpiTGSCAT.
\end{equation}
The analog of Lemma~\ref{uniprop2} is as follows.

\begin{lemma}\label{uniprop3}
Suppose $\AA$ is a graded 2-supercategory and $\BB$ is a
graded $(Q,\Pi)$-2-supercategory.
\begin{itemize}
\item[(i)]
Given a graded 2-superfunctor $\mathbb{R}:\AA \rightarrow \BB$, there
is a canonical graded
2-superfunctor $\tilde{\mathbb{R}}:\AA_{q,\pi} \rightarrow \BB$ such that
$\mathbb{R} = \tilde{\mathbb{R}}  \mathbb{J}$.
% and $\tilde{\mathbb{R}} \zeta_\lambda = \zeta'_{\mathbb{R} \lambda}$ for all $\lambda \in \ob\AA$.
\item[(ii)]
Given a 2-natural transformation $(X,x):\mathbb{R} \Rightarrow
\mathbb{S}$ for graded 2-superfunctors $\mathbb{R},\mathbb{S}:\AA
\rightarrow \BB$, there is a unique
2-natural transformation $(\tilde X, \tilde x):\tilde{\mathbb{R}} \Rightarrow \tilde{\mathbb{S}}$
such that 
${\tilde X}_\lambda = X_\lambda$ 
and $x_{\mu,\lambda} = \tilde x_{\mu,\lambda} \mathbb{J}$
for all
$\lambda,\mu \in \ob \AA$.
\end{itemize}
\end{lemma}

\begin{proof}
Since this is similar to the proof of Lemma~\ref{uniprop2}, we just go
briefly through the
definition of $\tilde{\mathbb{R}}$ in (i) (assuming that $\BB$ is strict).
On objects, we take 
$\tilde{\mathbb{R}} \lambda := \mathbb{R} \lambda$.
For $\lambda \in \ob \BB$, 
let $\zeta_\lambda^a:\pi_\lambda^a \stackrel{\sim}{\Rightarrow} \unit_\lambda$ be
defined as in the proof of Lemma~\ref{uniprop2}. 
Also for $m \in \Z$ let 
$\sigma_\lambda^m:q_\lambda^m \stackrel{\sim}{\Rightarrow}
\unit_\lambda$ 
be $(\sigma_\lambda)^m : (q_\lambda)^m \stackrel{\sim}{\Rightarrow}
\unit_\lambda$ if $m \geq 0$
or 
$(\bar\sigma_\lambda)^{-m} : (q_\lambda^{-1})^{-m}
\stackrel{\sim}{\Rightarrow} \unit\lambda$ if $m \leq 0$.
Then, for a 1-morphism $F:\lambda \rightarrow \mu$ in $\AA$,
$m \in \Z$ and $a\in \Z/2$,
we set $\tilde{\mathbb{R}} (Q^m\Pi^a F) := q_{\mathbb{R}\mu}^m 
\pi_{\mathbb{R}\mu}^a (\mathbb{R}F)$.
Also, if $x:F \Rightarrow G$ is a 2-morphism in $\AA$
for $F, G:\lambda\rightarrow\mu$,
we define $\tilde{\mathbb{R}}(x_{m,a}^{n,b}):\tilde{\mathbb{R}}(\Pi^a F)
\Rightarrow 
\tilde{\mathbb{R}}(\Pi^b G)$ to be the following composition:
$$
\begin{CD}
q_{\mathbb{R}\mu}^m \pi_{\mathbb{R}\mu}^a (\mathbb{R}F)
&@>\sigma_{\mathbb{R}\mu}^m\zeta_{\mathbb{R}\mu}^a(\mathbb{R}F)>>
&\mathbb{R}F
&
@>\mathbb{R}x>>
&\mathbb{R}G
&@>(\sigma_{\mathbb{R}\mu}^{n}\zeta_{\mathbb{R}\mu}^b)^{-1}(\mathbb{R}G)>>
&q_{\mathbb{R}\mu}^n \pi_{\mathbb{R}\mu}^b (\mathbb{R}G).
\end{CD}
$$
We also need coherence maps $\tilde\imath$ and $\tilde c$ for
$\mathbb{R}$, 
which are defined like in the proof of Lemma~\ref{uniprop2}.
In particular, $\tilde c_{Q^n \Pi^b G, Q^m \Pi^a F}$ is the following composition:
$$
q_{\mathbb{R}\nu}^n \pi^b_{\mathbb{R}\nu} (\mathbb{R} G) q_{\mathbb{R}\mu}^m\pi^a_{\mathbb{R}\mu} (\mathbb{R}
F)
\longrightarrow
q_{\mathbb{R}\nu}^n q_{\mathbb{R}\nu}^m \pi^b_{\mathbb{R}\nu} \pi^a_{\mathbb{R}\nu} (\mathbb{R} G) (\mathbb{R}
F)
\longrightarrow
q_{\mathbb{R}\nu}^{m+n} \pi^{a+b}_{\mathbb{R} \nu} \mathbb{R}(GF),
$$
where the first map is defined using the supernatural isomorphisms
$\beta$ and $\gamma$ from Lemma~\ref{turkey}(i)--(ii), and the second map is
defined by collapsing powers of $q_{\mathbb{R}\nu}$ using the
2-isomorphisms
from Lemma~\ref{turkey}(iii), collapsing $\pi_{\mathbb{R}\nu}
\pi_{\mathbb{R}\nu}$ using $-\xi$, and also using the given coherence
map 
$c_{G,F}:(\mathbb{R} G) (\mathbb{R} F) \stackrel{\sim}{\Rightarrow}
\mathbb{R} (GF)$.
\end{proof}

Using Lemma~\ref{uniprop3}, one gets also the analog of Theorem~\ref{golly}:
the functor $-_{q,\pi}$ from (\ref{macky}) is 
 left 2-adjoint to the
forgetful functor.

%%%%%%%%%%%%%%%%%%%%%%

Next, we explain the graded analogs of Definitions~\ref{jazz} and
\ref{clean}, and extend the results of Section 5. The following is an
efficient formulation of
the general notion of a strict action of
the group $\Z \oplus \Z/2$ on a $\k$-linear category.

\begin{definition}\label{hale}
(i) A {\em $(Q,\Pi)$-category}
is a $\k$-linear category $\A$
equipped with the following additional data: an endofunctor $\Pi:\A \rightarrow
\A$ and a natural isomorphism
$\xi:\Pi^2 \stackrel{\sim}{\Rightarrow} I$
such that $\xi \Pi = \Pi \xi$ in $\Hom(\Pi^3, \Pi)$;
endofunctors $Q, Q^{-1}:\A \rightarrow \A$ and 
natural isomorphisms $\ii:Q^{-1} Q \stackrel{\sim}{\Rightarrow} I, \jj:Q Q^{-1}
\stackrel{\sim}{\Rightarrow} I$ so that $\ii^{-1}$ and $\jj$ define a unit and a counit
 making $(Q,Q^{-1})$ into an adjoint pair of auto-equivalences;
a natural isomorphism $\beta_Q: \Pi Q \stackrel{\sim}{\Rightarrow} Q \Pi$
such that $\xi Q \xi^{-1} = \beta_Q \Pi \circ \Pi \beta_Q$ in
$\Hom(\Pi^2 Q, Q \Pi^2)$.

(ii) Given $(Q,\Pi)$-categories $\A$ and $\A'$,
a {\em $(Q,\Pi)$-functor} $F:\A \rightarrow \A'$
is a $\k$-linear functor with the additional data of natural isomorphisms
$\beta_F:\Pi' F \stackrel{\sim}{\Rightarrow} F \Pi$ 
and $\gamma_F: Q' F \stackrel{\sim}{\Rightarrow} FQ$ 
such that $\xi' F \xi^{-1}=\beta_F \Pi \circ \Pi' \beta_F$
in $\Hom((\Pi')^2 F, F \Pi^2)$.
For example,  $I$,
$\Pi$ and $Q$ are $(Q,\Pi)$-functors with $\beta_I := 1_\Pi, \beta_\Pi := -
1_{\Pi^2}$, $\beta_Q$ as
specified in (i), 
$\gamma_I := 1_Q, \gamma_\Pi := \beta_Q^{-1}$ 
and $\gamma_Q := 1_{Q^2}$.

(iii) Given $(Q,\Pi)$-functors $F, G:\A \rightarrow \A'$,
a {\em $(Q,\Pi)$-natural transformation} is a natural transformation 
$x:F \Rightarrow G$ such that 
$x \Pi \circ \beta_F=\beta_G \circ \Pi' x$
and
$x Q \circ \gamma_F= \gamma_G \circ Q' x$.
\end{definition}

There is a 2-category $\qpiCAT$
consisting of $(Q,\Pi)$-categories, $(Q,\Pi)$-functors and
$(Q,\Pi)$-natural transformations.
We want to compare this to 
$(Q,\Pi)$-$\underline{\GSCAT}$,
the 2-category
 of graded 
$(Q,\Pi)$-supercategories, graded superfunctors and homogeneous even supernatural
transformations of degree zero.
Like in (\ref{lun}), there is a strict 2-functor
\begin{equation}\label{worf}
\mathbb{E}:
(Q,\Pi)\text{-}\underline{\GSCAT}
\rightarrow
\qpiCAT
\end{equation}
sending a graded $(Q,\Pi)$-supercategory $\A$ to the underlying
category $\underline{\A}$, which is a $(Q,\Pi)$-category thanks to
Corollary~\ref{pigeon}(i).
It sends a graded superfunctor $F:\A \rightarrow \B$ to the
restriction
$\underline{F}:\underline{\A}\rightarrow \underline{\B}$, made into a
$(Q,\Pi)$-functor as in Corollary~\ref{pigeon}(ii).
It sends a homogeneous graded supernatural transformation $x:F
\Rightarrow G$ of degree zero to $\underline{x}:\underline{F}
\Rightarrow \underline{G}$ defined from $\underline{x}_\lambda :=
x_\lambda$, which is a $(Q,\Pi)$-natural transformation thanks to
Corollary~\ref{pigeon}(iii).

\begin{theorem}\label{th}
The 2-functor $\mathbb{E}$
from (\ref{worf}) is a 2-equivalence of 2-categories.
\end{theorem}

Theorem~\ref{th} is proved in a similar way to Theorem~\ref{pro}.
The key point of course is to define the appropriate strict 2-functor
$\mathbb{D}$ in the opposite direction. We just go briefly through
the definition of this, since there are a few subtleties.
So let $\A$ be a $(Q,\Pi)$-category.
Let $Q^n := Q \cdots Q$ ($n$ times) if $n \geq 0$ or $Q^{-1} \cdots Q^{-1}$ ($-n$ times) if $n \leq 0$. 
Given any composition $C$ of $r$ of the functors $Q$ and
$s$ of the functors $Q^{-1}$ (in any order), there is an isomorphism
$c:C \stackrel{\sim}{\Rightarrow} Q^{r-s}$ defined by repeatedly applying $\ii$ and $\jj$ to cancel
$Q^{-1} Q$- or $Q Q^{-1}$-pairs.
The isomorphism $c$ is independent of the particular order chosen for
these cancellations.
In particular, we obtain in this way a canonical isomorphism
$c_{m,n}:Q^m Q^n \stackrel{\sim}{\Rightarrow} Q^{m+n}$ for any $m,n \in \Z$, and deduce
that
\begin{equation}\label{prop1}
c_{l,m+n}\circ Q^l c_{m,n} = c_{l+m,n} \circ c_{l,m} Q^n
\end{equation}
in
$\Hom(Q^l Q^m Q^n, Q^{l+m+n})$.
Next, let $F: \A \rightarrow \A'$ be a
$(Q,\Pi)$-functor between two $(Q,\Pi)$-categories.
For each $n \in \Z$, we define an isomorphism 
$\gamma^{n}_F:(Q')^n F \stackrel{\sim}{\Rightarrow} F Q^n$ as follows:
set $\gamma_F^{0} := 1_F$;
then
for $n \geq 1$ recursively define
\begin{align*}\gamma_F^{n} &:= \gamma_F^{(n-1)} Q  \circ (Q')^{n-1} \gamma_F,\\
\gamma_F^{-n} &:= \gamma_F^{1-n} Q^{-1} \circ (Q')^{1-n} \ii' Q^{-1}
\circ (Q')^{-n} (\gamma_F)^{-1} Q^{-1} \circ (Q')^{-n} F \jj^{-1}.
\end{align*}
One can show that 
\begin{equation}\label{prop2}
\gamma^{m+n}_F\circ c'_{m,n} F = F c_{m,n} \circ
\gamma_F^m (Q^n)
\circ (Q')^m \gamma_F^n
\end{equation}
in $\Hom((Q')^m (Q')^n F, F Q^{m+n})$.
In particular, taking $F := \Pi:\A \rightarrow \A$, this gives us an isomorphism
$\gamma^{n}_\Pi: Q^n \Pi \stackrel{\sim}{\Rightarrow} \Pi Q^n$; let
$\beta_{Q^n}: \Pi Q^n \stackrel{\sim}{\Rightarrow} Q^n \Pi$ be its inverse.
This together with $\gamma_{Q^n}:= c_{n,1}^{-1} \circ c_{1,n}: Q Q^n
\stackrel{\sim}{\Rightarrow} Q^n Q$
makes $Q^n$ into a $(Q,\Pi)$-functor, i.e. we have that
\begin{align}
\label{prop4}
\xi Q^n &= Q^n \xi \circ \beta_{Q^n} \Pi \circ \Pi \beta_{Q^n}.\\\intertext{We note also that}
\label{prop3}
c_{m,n} \Pi\circ \beta_{Q^{m+n}} 
&= 
Q^m \beta_{Q^n} \circ \beta_{Q^m} Q^n
\circ  \Pi c_{m,n}.
\end{align}
Now, for a $(Q,\Pi)$-category $\A$, we are ready to define the {\em associated graded
  $(Q,\Pi)$-supercategory}
$\widehat{\A}$. 
It has the same objects as $\A$, and
morphisms
$\Hom_{\widehat{\A}}(\lambda,\mu)_{m,a} := \Hom_{\A}(\lambda,
Q^m \Pi^a \mu)$.
The composition $\hat g \circ \hat f$
of $\hat f, \hat g$ coming from $f:\lambda \rightarrow Q^m\Pi^a 
\mu,
g:\mu \rightarrow Q^n \Pi^b \nu$, respectively, is obtained from
$(Q^m \Pi^a g) \circ f:\lambda \rightarrow Q^m\Pi^a Q^n\Pi^b \nu$
by first using $\beta_{Q^n}$ to commute $Q^n$ past $\Pi^a$ if necessary,
then using $\xi$ and $c_{m,n}$ to simplify $Q^m Q^n \Pi^a \Pi^b \nu$ to
$Q^{m+n} \Pi^{a+b}\nu$.
The check that this is associative uses (\ref{prop1}),
(\ref{prop4})--(\ref{prop3}) and the identity $\xi \Pi = \Pi \xi$.
For a $(Q,\Pi)$-functor $F:\A \rightarrow \A'$,
we get $\widehat F:\widehat{\A} \rightarrow \widehat{\A}'$
by composing $Ff:F \lambda \rightarrow F Q^m \Pi^a \mu$ 
with the map $F Q^m \Pi^a \mu \rightarrow Q^m\Pi^a F \mu$
obtained using
$\beta_F$ and $\gamma_F^{n}$.
The check that $\widehat{F}(\hat g \circ \hat f) = (\widehat{F}\hat
g) \circ (\widehat{F}\hat f)$ uses (\ref{prop2}).
In particular, since $\Pi, Q$ and $Q^{-1}$ are all $(Q,\Pi)$-functors, this gives us
the functors
$\widehat{\Pi}, \widehat{Q}$ and $\widehat{Q}^{-1}$ needed to make
$\widehat{\C}$ into a graded $(Q,\Pi)$-supercategory.

\begin{definition}\label{dirty}
A {\em $(Q,\Pi)$-2-category} 
is a $\k$-linear $2$-category $\AA$
plus
families $\pi = (\pi_\lambda:\lambda\rightarrow\lambda), q =
(q_\lambda:\lambda\rightarrow \lambda)$ and $q^{-1} = (q^{-1}_{\lambda}:\lambda\rightarrow\lambda)$ of $1$-morphisms,
families $\beta = (\beta_{\mu,\lambda}:\pi_\mu
-\stackrel{\sim}{\Rightarrow} - \pi_\lambda)$  
and $\gamma = (\gamma_{\mu,\lambda}:q_\mu
-\stackrel{\sim}{\Rightarrow} - q_\lambda)$
of natural isomorphisms, and
families
$\xi = (\xi_\lambda:\pi_\lambda^2 \stackrel{\sim}{\Rightarrow} \unit_\lambda)$,
$\ii = (\ii_\lambda:q_\lambda^{-1} q_\lambda \stackrel{\sim}{\Rightarrow} \unit_\lambda)$
and $\jj = (\jj_\lambda:q_\lambda q_\lambda^{-1} \stackrel{\sim}{\Rightarrow} \unit_\lambda)$
of 2-isomorphisms, such that the following hold (assuming strictness):
\begin{itemize}
\item[(i)]
$(\pi, \beta)$ and $(q, \gamma)$  are objects in the Drinfeld center of
$\AA$;
\item[(ii)] 
$(\beta_{\lambda,\lambda})_{\pi_\lambda} =
-1_{\pi_\lambda^2}$,
$(\gamma_{\lambda,\lambda})_{q_\lambda} =
1_{q_\lambda^2}$
and
$(\gamma_{\lambda,\lambda})_{\pi_\lambda} =\left((\beta_{\lambda,\lambda})_{q_\lambda}\right)^{-1}$;
\item[(iii)]
$\xi_\mu F \xi_\lambda^{-1} = 
(\beta_{\mu,\lambda})_F \pi_\lambda\circ \pi_\mu (\beta_{\mu,\lambda})_F$
for all $1$-morphisms
$F:\lambda\rightarrow \mu$;
\item[(iv)]
$q_\lambda \ii_\lambda = \jj_\lambda q_\lambda$ and $\ii_\lambda
q_\lambda^{-1} = q_{\lambda}^{-1} \jj_\lambda$.
\end{itemize}
\end{definition}

The story here continues just as it did for $\Pi$-2-supercategories and
$\Pi$-2-categories. For a graded $(Q,\Pi)$-2-supercategory $\AA$,
its {\em underlying 2-category} $\underline{\AA}$, consisting of the
same objects and 1-morphisms but just the even 2-morphisms of
degree zero, is a $(Q,\Pi)$-2-category. Conversely, for a
$(Q,\Pi)$-2-category $\AA$, there is a construction of its {\em
  associated graded $(Q,\Pi)$-2-supercategory} $\widehat{\AA}$, which
we leave to the reader.
The constructions $\AA \mapsto \underline{\AA}$ 
and $\AA \mapsto \widehat{\AA}$ are mutual inverses (up to isomorphism), so that $(Q,\Pi)$-2-categories
and graded $(Q,\Pi)$-2-supercategories are equivalent notions.
Again, we leave it to the reader to formalize this statement by
writing down
the appropriate analog of Theorem~\ref{powers}.

Finally, we discuss Grothendieck groups/rings in the graded setting:
\begin{itemize}
\item
For a graded supercategory $\A$, let $\GSKar(\A)$ denote
$\Kar(\underline{\A}_{q,\pi})$, that is, the additive Karoubi envelope
of the underlying category to the $(Q,\Pi)$-envelope of $\A$.
This is a $(Q,\Pi)$-category that is additive and idempotent complete.
Its Grothendieck group $K_0(\GSKar(\A))$ is a
$\Z^\pi[q,q^{-1}]$-module
with $\pi$ acting as $[\Pi]$ and $q$ acting as $[Q]$.
\item
For a graded 2-supercategory $\AA$, let $\GSKAR(\AA) :=
\KAR(\underline{\AA}_{q,\pi})$ denote the additive Karoubi
envelope of the 
$(Q,\Pi)$-2-category underlying the $(Q,\Pi)$-envelope of $\AA$.
It is an additive,
idempotent complete $(Q,\Pi)$-2-category. Its Grothendieck ring
$K_0(\GSKAR(\AA))$ is naturally a locally unital ring with a disinguished system
of mutually orthogonal idempotents $\{1_\lambda\:|\:\lambda \in \ob
\AA\}$.
Moreover this ring is actually a $\Z^\pi[q,q^{-1}]$-algebra with 
$\pi$ and $q$ 
acting on $1_\mu K_0(\GSKAR(\AA)) 1_\lambda$ by left multiplication
by $[\pi_\mu]$ and $[q_\mu]$ (equivalently, right multiplication by
$[\pi_\lambda]$ and $[q_\lambda]$), respectively.
\end{itemize}
This construction will be used in particular in \cite{BE} in order
to pass from the Kac-Moody 2-supercategory $\mathfrak{U}(\g)$
introduced there 
to the modified integral form of the corresponding covering quantum group
$U_{q,\pi}(\g)$ as in \cite{Clark}.

\appendix
\section{Odd Temperley-Lieb}

In this appendix, we prove Theorem~\ref{stl}.
Throughout we let $\eps := -1$.
If instead one takes $\eps:= +1$ and works in the purely even
setting, replacing the quantum superalgebra $\dot
U_q(\mathfrak{osp}_{1|2})$ with the quantum algebra $\dot
U_q(\mathfrak{sl}_2)$,
the arguments below may be used to recover 
the classical result for the Temperley-Lieb category
$\mathcal{TL}(\delta)$.
We assume some familiarity with the combinatorics from that story; e.g. see \cite{W}.

Let $\k$ be a field of characteristic different from 2, and $q \in \k^\times$ be a scalar that is not a root of unity.
For any $n \in \Z$, let $[n]$ denote $\frac{q^n-(\eps q)^{-n}}{q-\eps
  q^{-1}\:}$. 
For $n \in \N$, 
the element $[n]$ is the same as $[n]_{q,\eps}$ from (\ref{qint}), 
and $[-n] = -\eps^n [n]$.
Also set $\delta := -[2] = -(q+\eps q^{-1})$.
Recall that $\mathcal{STL}(\delta)$ is the strict monoidal supercategory
with one generating object $\color{darkpurple}{\cdot}$
and two odd generating morphisms
$\mathord{
\begin{tikzpicture}[baseline = 0]
	\draw[-,thick,darkpurple] (0.3,0.25) to[out=-90, in=0] (0.1,-0.05);
	\draw[-,thick,darkpurple] (0.1,-0.05) to[out = 180, in = -90] (-0.1,0.25);
\end{tikzpicture}
}$
and 
$\mathord{
\begin{tikzpicture}[baseline = 0]
	\draw[-,thick,darkpurple] (0.3,-0.1) to[out=90, in=0] (0.1,0.2);
	\draw[-,thick,darkpurple] (0.1,0.2) to[out = 180, in = 90] (-0.1,-.1);
\end{tikzpicture}
}\,$, 
subject to the following relations:
\begin{align*}
\mathord{
\begin{tikzpicture}[baseline = 0]
  \draw[-,thick,darkpurple] (0.2,0) to (0.2,.5);
	\draw[-,thick,darkpurple] (0.2,0) to[out=-90, in=0] (0,-0.35);
	\draw[-,thick,darkpurple] (0,-0.35) to[out = 180, in = -90] (-0.2,0);
	\draw[-,thick,darkpurple] (-0.2,0) to[out=90, in=0] (-0.4,0.35);
	\draw[-,thick,darkpurple] (-0.4,0.35) to[out = 180, in =90] (-0.6,0);
  \draw[-,thick,darkpurple] (-0.6,0) to (-0.6,-.5);
\end{tikzpicture}
}
\,&=\,
\mathord{\begin{tikzpicture}[baseline=0]
  \draw[-,thick,darkpurple] (0,-0.4) to (0,.4);
\end{tikzpicture}
}\:,
\qquad
\mathord{
\begin{tikzpicture}[baseline = 0]
  \draw[-,thick,darkpurple] (0.2,0) to (0.2,-.5);
	\draw[-,thick,darkpurple] (0.2,0) to[out=90, in=0] (0,0.35);
	\draw[-,thick,darkpurple] (0,0.35) to[out = 180, in = 90] (-0.2,0);
	\draw[-,thick,darkpurple] (-0.2,0) to[out=-90, in=0] (-0.4,-0.35);
	\draw[-,thick,darkpurple] (-0.4,-0.35) to[out = 180, in =-90] (-0.6,0);
  \draw[-,thick,darkpurple] (-0.6,0) to (-0.6,.5);
\end{tikzpicture}
}\,
=\,
\eps\:\,\mathord{\begin{tikzpicture}[baseline=0]
  \draw[-,thick,darkpurple] (0,-0.4) to (0,.4);
\end{tikzpicture}
}\:,\qquad
\mathord{
\begin{tikzpicture}[baseline = 0]
	\draw[-,thick,darkpurple] (0,-.25) to[out=180,in=-90] (-0.28,0.05);
	\draw[-,thick,darkpurple] (0,-.25) to[out=0,in=-90] (0.28,0.05);
	\draw[-,thick,darkpurple] (0,.35) to[out=180,in=90] (-0.28,0.05);
	\draw[-,thick,darkpurple] (0,.35) to[out=0,in=90] (0.28,0.05);
\end{tikzpicture}
}\,
=\delta.
\end{align*}
We denote the $n$-fold tensor product of the generating object
$\color{darkpurple}{\cdot}$
by $n$ and its identity endomorphism by $e_n$.

Using the string calculus, 
any crossingless matching connecting $m$ points on the bottom 
boundary and $n$ points
on the top boundary can be interpreted as a morphism $m \rightarrow n$
in $\mathcal{STL}(\delta)$. In view of the super interchange law,
isotopic crossingless matchings produce the same morphism up
to a sign. Moreover, to get a spanning set for
$\Hom_{\mathcal{STL}(\delta)}(m,n)$,
one just has to pick a system of representatives for the isotopy
classes of crossingless matchings.
Our first claim is that any such spanning set is actually a basis for
$\Hom_{\mathcal{STL}(\delta)}(m,n)$. For example, this assertion
implies that
$\Hom_{\mathcal{STL}(\delta)}(3,3)$ is of dimension 5 (the third
Catalan number) with basis
$$
\mathord{
\begin{tikzpicture}[baseline = 0]
  \draw[-,thick,darkpurple] (0.2,-.35) to [out=90,in=-90] (0.2,.35);
	\draw[-,thick,darkpurple] (-0.2,.35) to[out=-90, in=0] (-0.4,.05);
	\draw[-,thick,darkpurple] (-.4,.05) to[out = 180, in = -90] (-0.6,.35);
	\draw[-,thick,darkpurple] (-0.2,-.35) to[out=90, in=0] (-0.4,-.05);
	\draw[-,thick,darkpurple] (-0.4,-.05) to[out = 180, in =90] (-0.6,-.35);
\end{tikzpicture}
}\:,\qquad
\mathord{
\begin{tikzpicture}[baseline = 0]
  \draw[-,thick,darkpurple] (-0.6,-.35) to [out=90,in=-90] (-0.6,.35);
	\draw[-,thick,darkpurple] (0.2,.35) to[out=-90, in=0] (0,.05);
	\draw[-,thick,darkpurple] (0,.05) to[out = 180, in = -90] (-0.2,.35);
	\draw[-,thick,darkpurple] (0.2,-.35) to[out=90, in=0] (0,-.05);
	\draw[-,thick,darkpurple] (0,-.05) to[out = 180, in =90] (-0.2,-.35);
\end{tikzpicture}
}\:,\qquad
\mathord{
\begin{tikzpicture}[baseline = 0]
  \draw[-,thick,darkpurple] (0.2,-.35) to [out=90,in=-90] (-0.6,.35);
	\draw[-,thick,darkpurple] (0.2,.35) to[out=-90, in=0] (0,.05);
	\draw[-,thick,darkpurple] (0,.05) to[out = 180, in = -90] (-0.2,.35);
	\draw[-,thick,darkpurple] (-0.2,-.35) to[out=90, in=0] (-0.4,-.05);
	\draw[-,thick,darkpurple] (-0.4,-.05) to[out = 180, in =90] (-0.6,-.35);
\end{tikzpicture}
}\:,\qquad
\mathord{
\begin{tikzpicture}[baseline = 0]
  \draw[-,thick,darkpurple] (-0.6,-.35) to [out=90,in=-90] (0.2,.35);
	\draw[-,thick,darkpurple] (-0.2,.35) to[out=-90, in=0] (-0.4,.05);
	\draw[-,thick,darkpurple] (-0.4,.05) to[out = 180, in = -90] (-0.6,.35);
	\draw[-,thick,darkpurple] (0.2,-.35) to[out=90, in=0] (0,-.05);
	\draw[-,thick,darkpurple] (0,-.05) to[out = 180, in =90] (-0.2,-.35);
\end{tikzpicture}
}\:,\qquad
\mathord{
\begin{tikzpicture}[baseline = 0]
  \draw[-,thick,darkpurple] (0.2,-.35) to [out=90,in=-90] (0.2,.35);
  \draw[-,thick,darkpurple] (-0.2,-.35) to [out=90,in=-90] (-0.2,.35);
  \draw[-,thick,darkpurple] (-0.6,-.35) to [out=90,in=-90] (-0.6,.35);
\end{tikzpicture}
}\:.
$$
To prove it, we construct an explicit representation of
$\mathcal{STL}(\delta)$.

\begin{lemma}\label{tennis}
Let $V$ be the vector superspace on basis $v_1, v_{-1}$, where $v_1$
is even and $v_{-1}$ is odd.
There is a monoidal superfunctor $G:\mathcal{STL}(\delta) \rightarrow
\SVEC$ with
$G(n) = V^{\otimes n}$
and
\begin{align*}
G\big(\,\mathord{
\begin{tikzpicture}[baseline = 0]
	\draw[-,thick,darkpurple] (0.3,0.25) to[out=-90, in=0] (0.1,-0.05);
	\draw[-,thick,darkpurple] (0.1,-0.05) to[out = 180, in = -90] (-0.1,0.25);
\end{tikzpicture}
}\,\big)
&:\k \rightarrow V \otimes V,
\quad
&1 &\mapsto v_{-1} \otimes v_1 - q v_1 \otimes v_{-1};\!\!\!\!\!\!\!\!\!\!\!\!\!\!\!\!\!\\
G\big(\,\mathord{
\begin{tikzpicture}[baseline = 0]
	\draw[-,thick,darkpurple] (0.3,-0.1) to[out=90, in=0] (0.1,0.2);
	\draw[-,thick,darkpurple] (0.1,0.2) to[out = 180, in = 90] (-0.1,-.1);
\end{tikzpicture}
}\,\big)&:V \otimes V \rightarrow \k,
\quad
&v_1 \otimes v_1 &\mapsto 0,
&v_1 \otimes v_{-1} &\mapsto 1,\\
&&v_{-1} \otimes v_1 &\mapsto -\eps q^{-1},
&v_{-1} \otimes v_{-1} &\mapsto 0.
\end{align*}
\end{lemma}

\begin{proof}
Check the three relations.
\end{proof}

\begin{theorem}\label{stlbasis}
Any set of representatives for the isotopy classes of crossingless
matchings from $m$ points to $n$ points defines a basis for
$\Hom_{\mathcal{STL}(\delta)}(m,n)$.
\end{theorem}

\begin{proof}
We just need to prove linear independence.
There is a linear map 
$$
\Hom_{\mathcal{STL}(\delta)}(m,n)
\rightarrow \Hom_{\mathcal{STL}(\delta)}(m+n,0),\qquad
f \mapsto c_n \circ (f \otimes e_n),
$$
where $c_n\in \Hom_{\mathcal{STL}(\delta)}(2n,0)$
is the morphism defined by $n$ nested caps.
Using this, one reduces to proving the result in the special case that
$m$ is even and 
$n=0$, i.e. our crossingless matchings consist of $m/2$ caps.
Let $S$ be a set of representatives for such matchings.
For $s \in S$, let 
$\theta_s:V^{\otimes m} \rightarrow \k$ be the linear map obtained by
applying the monoidal superfunctor $G$ from Lemma~\ref{tennis}
to the morphism in $\mathcal{STL}(\delta)$ that is defined by $s$.
It suffices to show that the linear maps $\{\theta_s\:|\:s \in S\}$ are
linearly independent.

By writing $+1$ underneath the left hand vertex and $-1$ underneath the right hand vertex
of each cap of $s \in S$ then reading off the resulting sequence, 
we obtain a function from $S$ to the set of {\em Dyck sequences}
$(s_1,\dots,s_m)$ with $s_1,\dots,s_m \in \{\pm 1\}$ and
$s_1+\cdots+s_k \geq 0$ for each $k=1,\dots,m$. 
As $s$ can be recovered uniquely (up to
isotopy) from its Dyck sequence, the vectors $\{v_s := v_{s_1}\otimes\cdots\otimes v_{s_m} \in
V^{\otimes m}\:|\:s \in S\}$ are linearly
independent.
Finally, we observe that $\theta_s(v_s) = \pm 1$ and $\theta_s(v_t) = 0$
unless $t \leq s$, where $\leq$ is the partial order 
defined by $s \leq t$ if and only if the corresponding Dyck sequences satisfy
$s_1 + \cdots + s_k \leq
t_1+\cdots + t_k$ for each $k=1,\dots,m$.
The required linear independence follows.
\end{proof}

Now we can prove Theorem~\ref{stl}:

\begin{theorem}\label{old}
For $\delta$ as above,
$\SKar(\mathcal{STL}(\delta))$ is
a semisimple Abelian category.
Moreover, as a based ring,
$K_0(\SKar(\mathcal{STL}(\delta)))$ is isomorphic to the subring 
of $\Z^\pi[x,x^{-1}]$ with basis $\left\{[n+1]_{x,\pi}, \pi
  [n+1]_{x,\pi}\:\big|\:n \in \N\right\}$.
\end{theorem}

\begin{proof}
We begin by defining super analogs of the {\em Jones-Wenzl
  projectors}
$$
f_n = 
\mathord{
\begin{tikzpicture}[baseline = -2.2]
	\draw[-,thick,darkpurple] (0.28,-.4) to (0.28,-.15);
	\draw[-,thick,darkpurple] (-0.12,-.4) to (-0.12,-.15);
	\draw[-,thick,darkpurple] (-0.12,.4) to (-0.12,.15);
	\draw[-,thick,darkpurple] (0.28,.4) to (0.28,.15);
	\draw[-,thick,darkpurple] (-0.3,.15) to (0.46,.15);
	\draw[-,thick,darkpurple] (0.46,-.15) to (0.46,.15);
	\draw[-,thick,darkpurple] (-0.3,-.15) to (0.46,-.15);
	\draw[-,thick,darkpurple] (-0.3,-.15) to (-0.3,.15);
   \node at (0.08,0) {\color{darkpurple}$\scriptstyle{n}$};
   \node at (0.08,.28) {\color{darkpurple}$\scriptstyle{\cdots}$};
   \node at (0.08,-.28) {\color{darkpurple}$\scriptstyle{\cdots}$};
\end{tikzpicture}
}\:
\in \End_{\mathcal{STL}(\delta)}(n).
$$
These are defined recursively by setting
$f_0 := \unit$ and
$$
\mathord{
\begin{tikzpicture}[baseline = -2.2]
	\draw[-,thick,darkpurple] (0.33,-.4) to (0.33,-.15);
	\draw[-,thick,darkpurple] (-0.17,-.4) to (-0.17,-.15);
	\draw[-,thick,darkpurple] (-0.17,.4) to (-0.17,.15);
	\draw[-,thick,darkpurple] (0.33,.4) to (0.33,.15);
	\draw[-,thick,darkpurple] (-0.35,.15) to (0.51,.15);
	\draw[-,thick,darkpurple] (0.51,-.15) to (0.51,.15);
	\draw[-,thick,darkpurple] (-0.35,-.15) to (0.51,-.15);
	\draw[-,thick,darkpurple] (-0.35,-.15) to (-0.35,.15);
   \node at (0.08,0) {\color{darkpurple}$\scriptstyle{n+1}$};
   \node at (0.08,.28) {\color{darkpurple}$\scriptstyle{\cdots}$};
   \node at (0.08,-.28) {\color{darkpurple}$\scriptstyle{\cdots}$};
\end{tikzpicture}
}
:=
\mathord{
\begin{tikzpicture}[baseline = -2.2]
	\draw[-,thick,darkpurple] (0.28,-.4) to (0.28,-.15);
	\draw[-,thick,darkpurple] (0.6,-.4) to (0.6,.4);
	\draw[-,thick,darkpurple] (-0.12,-.4) to (-0.12,-.15);
	\draw[-,thick,darkpurple] (-0.12,.4) to (-0.12,.15);
	\draw[-,thick,darkpurple] (0.28,.4) to (0.28,.15);
	\draw[-,thick,darkpurple] (-0.3,.15) to (0.46,.15);
	\draw[-,thick,darkpurple] (0.46,-.15) to (0.46,.15);
	\draw[-,thick,darkpurple] (-0.3,-.15) to (0.46,-.15);
	\draw[-,thick,darkpurple] (-0.3,-.15) to (-0.3,.15);
   \node at (0.08,0) {\color{darkpurple}$\scriptstyle{n}$};
   \node at (0.08,.28) {\color{darkpurple}$\scriptstyle{\cdots}$};
   \node at (0.08,-.28) {\color{darkpurple}$\scriptstyle{\cdots}$};
\end{tikzpicture}
}\:
+\frac{[n]}{[n+1]}
\mathord{
\begin{tikzpicture}[baseline = -2.2]
	\draw[-,thick,darkpurple] (0.28,-.15) to[out=90,in=90] (0.6,-.15);
	\draw[-,thick,darkpurple] (0.6,-.15) to (0.6,-.7);
	\draw[-,thick,darkpurple] (-0.12,.7) to (-0.12,.45);
	\draw[-,thick,darkpurple] (0.28,.7) to (0.28,.45);
	\draw[-,thick,darkpurple] (-0.3,.45) to (0.46,.45);
	\draw[-,thick,darkpurple] (0.46,.15) to (0.46,.45);
	\draw[-,thick,darkpurple] (-0.3,.15) to (0.46,.15);
	\draw[-,thick,darkpurple] (-0.3,.15) to (-0.3,.45);
   \node at (0.08,0.3) {\color{darkpurple}$\scriptstyle{n}$};
   \node at (0.08,.58) {\color{darkpurple}$\scriptstyle{\cdots}$};
   \node at (0.08,0) {\color{darkpurple}$\scriptstyle{\cdots}$};
	\draw[-,thick,darkpurple] (0.28,-.7) to (0.28,-.45);
	\draw[-,thick,darkpurple] (-0.12,-.7) to (-0.12,-.45);
	\draw[-,thick,darkpurple] (-0.12,.15) to (-0.12,-.15);
	\draw[-,thick,darkpurple] (0.28,.15) to[out=-90,in=-90] (0.6,.15);
	\draw[-,thick,darkpurple] (0.6,.15) to (0.6,.7);
	\draw[-,thick,darkpurple] (-0.3,-.15) to (0.46,-.15);
	\draw[-,thick,darkpurple] (0.46,-.45) to (0.46,-.15);
	\draw[-,thick,darkpurple] (-0.3,-.45) to (0.46,-.45);
	\draw[-,thick,darkpurple] (-0.3,-.45) to (-0.3,-.15);
   \node at (0.08,-0.3) {\color{darkpurple}$\scriptstyle{n}$};
   \node at (0.08,-.58) {\color{darkpurple}$\scriptstyle{\cdots}$};
\end{tikzpicture}
}\:.
$$
Clearly, $f_n$ is equal to $e_n$
plus a linear combination of
diagrams with at least one cup and cap. Hence, using Theorem~\ref{stlbasis},
$f_n$ is non-zero.
By (\ref{qint2}), we have that
$[n][2] = [n+1] + \eps [n-1]$.
Using this,
an easy but crucial inductive calculation shows that
$$
\mathord{
\begin{tikzpicture}[baseline = -2.2]
	\draw[-,thick,darkpurple] (0.28,-.2) to[out=-90,in=-90] (0.6,-.2);
	\draw[-,thick,darkpurple] (0.28,.2) to[out=90,in=90] (0.6,.2);
	\draw[-,thick,darkpurple] (0.28,-.2) to (0.28,-.15);
	\draw[-,thick,darkpurple] (0.6,-.2) to (0.6,.2);
	\draw[-,thick,darkpurple] (-0.12,-.4) to (-0.12,-.15);
	\draw[-,thick,darkpurple] (-0.12,.4) to (-0.12,.15);
	\draw[-,thick,darkpurple] (0.28,.2) to (0.28,.15);
	\draw[-,thick,darkpurple] (-0.3,.15) to (0.46,.15);
	\draw[-,thick,darkpurple] (0.46,-.15) to (0.46,.15);
	\draw[-,thick,darkpurple] (-0.3,-.15) to (0.46,-.15);
	\draw[-,thick,darkpurple] (-0.3,-.15) to (-0.3,.15);
   \node at (0.08,0) {\color{darkpurple}$\scriptstyle{n}$};
   \node at (0.08,.28) {\color{darkpurple}$\scriptstyle{\cdots}$};
   \node at (0.08,-.28) {\color{darkpurple}$\scriptstyle{\cdots}$};
\end{tikzpicture}
}
=
-\frac{[n+1]}{[n]}
\mathord{
\begin{tikzpicture}[baseline = -2.2]
	\draw[-,thick,darkpurple] (0.28,-.4) to (0.28,-.15);
	\draw[-,thick,darkpurple] (-0.12,-.4) to (-0.12,-.15);
	\draw[-,thick,darkpurple] (-0.12,.4) to (-0.12,.15);
	\draw[-,thick,darkpurple] (0.28,.4) to (0.28,.15);
	\draw[-,thick,darkpurple] (-0.3,.15) to (0.46,.15);
	\draw[-,thick,darkpurple] (0.46,-.15) to (0.46,.15);
	\draw[-,thick,darkpurple] (-0.3,-.15) to (0.46,-.15);
	\draw[-,thick,darkpurple] (-0.3,-.15) to (-0.3,.15);
   \node at (0.08,0) {\color{darkpurple}$\scriptstyle{n-1}$};
   \node at (0.08,.28) {\color{darkpurple}$\scriptstyle{\cdots}$};
   \node at (0.08,-.28) {\color{darkpurple}$\scriptstyle{\cdots}$};
\end{tikzpicture}
}\:,
$$
and each $f_n$ is an idempotent.
Moreover,
one gets zero 
if one vertically composes $f_n$ on top (resp. bottom)
with any diagram involving a cap (resp. a cup).

To prove the semisimplicity, we find it convenient to 
replace the supercategory $\mathcal{STL}(\delta)$ with the superalgebra
$$
A := \displaystyle\bigoplus_{m,n \in \N}
\Hom_{\mathcal{STL}(\delta)}(m,n),
$$
whose multiplication is induced by composition in $\mathcal{STL}(\delta)$.
Note that $A$ is a  {\em locally unital superalgebra}
with distinguished idempotents $\{e_n\:|\:n \in \N\}$.
Moreover, it is 
{\em locally finite dimensional} in the sense that each
$e_n A e_m$ is a finite-dimensional superspace.
Consider the $\Pi$-supercategory $\rSMod A$ consisting of right
$A$-supermodules $V$ 
which are themselves locally unital in the sense that
$V = \bigoplus_{n \in \N} V e_n$. 
Like in Example~\ref{day}(i),
there is an equivalence between $\SKar(\mathcal{STL}(\delta))$ and the
full subcategory of 
$\underline{\rSMod} A$ consisting of all finitely generated projective
supermodules.
Thus, we are reduced to working in $\rSMod A$.
Let $P(n) := f_n A$, which is a projective supermodule.
Let $L$ be any irreducible $A$-supermodule. Let $n \in \N$ be minimal such
that $L e_n \neq 0$. The minimality of $n$ ensures that any basis
element of $A$ with a cup in its diagram acts as zero on
$L e_n$.
We deduce that $\Hom_A(P(n), L) \cong L f_n = L e_n \neq 0$, demonstrating
that $L$ is a quotient of $P(n)$ or $\Pi P(n)$.
Moreover, $\End_A(P(n)) = f_n A f_n \cong \k$, so 
$P(n)$ is indecomposable; equivalently,
$f_n$ is a primitive idempotent.
Also for $m \neq n$, we have that 
$\Hom_A(P(m), P(n)) = f_n A f_m = 0$.
These observations together imply that 
every $A$-supermodule
is completely reducible, and
each irreducible $A$-supermodule is evenly isomorphic 
to a unique 
one of the supermodules
$\{P(n), \Pi P(n)\:|\:n \in \N\}$, which are themselves irreducible.

The previous paragraph implies that $\SKar(\mathcal{STL}(\delta))$
is a semisimple Abelian category.
Moreover, we get a basis for $K_0(\SKar(\mathcal{STL}(\delta)))$
by taking the isomorphism classes in
$\SKar(\mathcal{STL}(\delta))$ corresponding to
the primitive idempotents
$\left\{(f_n)^{\0}_{\0}, (f_n)^{\1}_{\1} \:\big|\:n \in \N\right\}$.
Thus, 
we can identify $K_0(\SKar(\mathcal{STL}(\delta)))$ with
the ring in the statement of the theorem using the correspondence
$(f_n)^{\0}_{\0} \leftrightarrow [n+1]_{x,\pi}$
and 
$(f_n)^{\1}_{\1} \leftrightarrow \pi [n+1]_{x,\pi}$.
To complete the proof of the theorem, it remains to check that the
ring structures agree. Since 
$[n]_{x,\pi} [2]_{x,\pi} = [n+1]_{x,\pi} + \pi [n-1]_{x,\pi}$ by (\ref{qint2}),
we must show that
the idempotents $(f_{n-1})^{\0}_{\0} \otimes (f_1)^{\0}_{\0}$
and $(f_n)^{\0}_{\0} + (f_{n-2})^{\1}_{\1}$
are equivalent for each $n \geq 2$.
We have that $(f_{n-1})^{\0}_{\0} \otimes (f_1)^{\0}_{\0}
= (f_{n-1} \otimes f_1)^{\0}_{\0} = (f_n)_{\0}^{\0}+(g_n)_{\0}^{\0}$
where
$$
g_n := 
-\frac{[n-1]}{[n]}
\mathord{
\begin{tikzpicture}[baseline = -2.2]
	\draw[-,thick,darkpurple] (0.28,-.15) to[out=90,in=90] (0.6,-.15);
	\draw[-,thick,darkpurple] (0.6,-.15) to (0.6,-.7);
	\draw[-,thick,darkpurple] (-0.12,.7) to (-0.12,.45);
	\draw[-,thick,darkpurple] (0.28,.7) to (0.28,.45);
	\draw[-,thick,darkpurple] (-0.3,.45) to (0.46,.45);
	\draw[-,thick,darkpurple] (0.46,.15) to (0.46,.45);
	\draw[-,thick,darkpurple] (-0.3,.15) to (0.46,.15);
	\draw[-,thick,darkpurple] (-0.3,.15) to (-0.3,.45);
   \node at (0.08,0.3) {\color{darkpurple}$\scriptstyle{n-1}$};
   \node at (0.08,.58) {\color{darkpurple}$\scriptstyle{\cdots}$};
   \node at (0.08,0) {\color{darkpurple}$\scriptstyle{\cdots}$};
	\draw[-,thick,darkpurple] (0.28,-.7) to (0.28,-.45);
	\draw[-,thick,darkpurple] (-0.12,-.7) to (-0.12,-.45);
	\draw[-,thick,darkpurple] (-0.12,.15) to (-0.12,-.15);
	\draw[-,thick,darkpurple] (0.28,.15) to[out=-90,in=-90] (0.6,.15);
	\draw[-,thick,darkpurple] (0.6,.15) to (0.6,.7);
	\draw[-,thick,darkpurple] (-0.3,-.15) to (0.46,-.15);
	\draw[-,thick,darkpurple] (0.46,-.45) to (0.46,-.15);
	\draw[-,thick,darkpurple] (-0.3,-.45) to (0.46,-.45);
	\draw[-,thick,darkpurple] (-0.3,-.45) to (-0.3,-.15);
   \node at (0.08,-0.3) {\color{darkpurple}$\scriptstyle{n-1}$};
   \node at (0.08,-.58) {\color{darkpurple}$\scriptstyle{\cdots}$};
\end{tikzpicture}
}\:\:.
$$
Using the properties from the first paragraph of the proof,
we have that $g_n\circ g_n = g_n$ and $g_n \circ f_n = f_n \circ g_n =
0$, so $(f_n)^{\0}_{\0}$ and $(g_n)^{\0}_{\0}$ are orthogonal
idempotents.
It just remains to observe that 
$$
u_n :=
-\frac{[n-1]}{[n]}
\mathord{
\begin{tikzpicture}[baseline = 6]
	\draw[-,thick,darkpurple] (0.28,.05) to[out=-90,in=-90] (0.6,.05);
	\draw[-,thick,darkpurple] (0.6,.05) to (0.6,.7);
	\draw[-,thick,darkpurple] (0.28,.05) to (0.28,.15);
	\draw[-,thick,darkpurple] (-0.12,.7) to (-0.12,.45);
	\draw[-,thick,darkpurple] (0.28,.7) to (0.28,.45);
	\draw[-,thick,darkpurple] (-0.3,.45) to (0.46,.45);
	\draw[-,thick,darkpurple] (0.46,.15) to (0.46,.45);
	\draw[-,thick,darkpurple] (-0.3,.15) to (0.46,.15);
	\draw[-,thick,darkpurple] (-0.3,.15) to (-0.3,.45);
	\draw[-,thick,darkpurple] (-0.12,-.1) to (-0.12,.15);
   \node at (0.08,0.3) {\color{darkpurple}$\scriptstyle{n-1}$};
   \node at (0.08,.58) {\color{darkpurple}$\scriptstyle{\cdots}$};
   \node at (0.08,0.02) {\color{darkpurple}$\scriptstyle{\cdots}$};
\end{tikzpicture}
}\:\:,
\qquad
v_n :=
\mathord{
\begin{tikzpicture}[baseline = -10.3]
	\draw[-,thick,darkpurple] (0.28,-.05) to[out=90,in=90] (0.6,-.05);
	\draw[-,thick,darkpurple] (0.6,-.05) to (0.6,-.7);
	\draw[-,thick,darkpurple] (0.28,-.05) to (0.28,-.15);
	\draw[-,thick,darkpurple] (0.28,-.7) to (0.28,-.45);
	\draw[-,thick,darkpurple] (-0.12,-.7) to (-0.12,-.45);
	\draw[-,thick,darkpurple] (-0.12,.15) to (-0.12,-.15);
	\draw[-,thick,darkpurple] (-0.3,-.15) to (0.46,-.15);
	\draw[-,thick,darkpurple] (0.46,-.45) to (0.46,-.15);
	\draw[-,thick,darkpurple] (-0.3,-.45) to (0.46,-.45);
	\draw[-,thick,darkpurple] (-0.3,-.45) to (-0.3,-.15);
   \node at (0.08,-0.3) {\color{darkpurple}$\scriptstyle{n-1}$};
   \node at (0.08,-.58) {\color{darkpurple}$\scriptstyle{\cdots}$};
   \node at (0.08,-0.02) {\color{darkpurple}$\scriptstyle{\cdots}$};
\end{tikzpicture}
}\:\:
$$
are odd morphisms in $\mathcal{STL}(\delta)$ such that $u_n \circ v_n
= g_n$ and $v_n \circ u_n = f_{n-2}$.
Hence, we get that
$(v_n)_{\0}^{\1} \circ (g_n)_{\0}^{\0} \circ (u_n)^{\0}_{\1} =
(f_{n-2})^{\1}_{\1}$, i.e.
$(g_n)^{\0}_{\0}$ is
equivalent to $(f_{n-2})^{\1}_{\1}$ in
$\SKar(\mathcal{STL}(\delta))$, as required.
\end{proof}

To explain what is really going on here, assume finally that the ground field
$\k$ is of characteristic zero.
Let $U = \dot U_q(\mathfrak{osp}_{1|2})$ be the locally unital
superalgebra
with homogeneous distinguished idempotents $\{1_n\:|\:n\in \Z\}$
and odd generators $E_n \in 1_{n+2} U 1_n$
and $F_n \in 1_{n-2} U 1_n$,
subject to the relations
$$
E_{n-2} F_n - \eps F_{n+2} E_n = [n] 1_n.
$$
This is the idempotented form of the quantum supergroup $U_q(\mathfrak{osp}_{1|2})$
introduced in \cite{CW}\footnote{More precisely, our $U$ is the 
idempotented form of the 
algebra 
from
 \cite{CW} as defined in \cite{Clark}. Also, we are using a
different convention for quantum integers compared to
\cite{CW, Clark}: our $q$ is the
same as the parameter $q^{-1}$ of \cite{CW} or the parameter $v^{-1}$ of \cite{Clark}.}.
Let $\mathcal C$ be the $\Pi$-supercategory
of all finite-dimensional left $U$-supermodules $V$
which are locally unital in the sense that $V = \bigoplus_{n \in \Z} 1_n V$.
By \cite{CW}, the underlying $\Pi$-category
$\underline{\C}$
is a semisimple Abelian category, and a complete set of pairwise
inequivalent irreducible objects is given by $\{V(n), \Pi V(n)\:|\:n
\in \N\}$, where $V(n)$ is defined as follows.
It has a homogeneous basis $v_{n},
v_{n-2},\dots, v_{-n}$ with $|v_{i}| = (n-i)/2\pmod{2}$.
We have that $1_i v_i = v_i$. The appropriate $E$'s and $F$'s act on the basis by
the following scalars:
\begin{align*}E&:v_{n} \stackrel{[n]}{\longleftarrow}
v_{n-2} 
\stackrel{[n-1]}{\longleftarrow}
\cdots 
\stackrel{[2]}{\longleftarrow}
v_{2-n}
\stackrel{[1]}{\longleftarrow}
v_{-n},\\
 F&:v_{n} \stackrel{[1]}{\longrightarrow}
v_{n-2} 
\stackrel{\eps[2]}{\longrightarrow}
v_{n-4} 
\stackrel{[3]}{\longrightarrow}
\cdots 
\stackrel{\eps^{n-1}[n]}{\longrightarrow}
v_{-n}.
\end{align*}
For example: $(E_{n-2} F_{n} - \eps F_{n+2} E_{n}) v_{n} =
E_{n-2} v_{n-2} = [n] v_{n}$.

We wish next to make $U$ into a Hopf superalgebra by introducing a
comultiplication $\Delta$ and counit $\eps$ defined on generators by
the following:
\begin{align*}
\Delta(1_n) &= \sum_{a+b=n} 1_a \otimes 1_b,
&\eps(1_n) &= \delta_{n,0} 1,\\
\\
\Delta(E_n) &= \sum_{a+b=n} (E_a \otimes 1_b + q^{-a} 1_a \otimes
E_b),
&
\eps(E_n) &= 0,\\
\\
\Delta(F_n) &= \sum_{a+b=n} (\eps^a 1_a\otimes F_b + q^b F_a \otimes
1_b),
&\eps(F_n) &= 0.
\end{align*}
However, some of these formulae involve infinite sums, so don't make sense
yet:
we need some completions!
If $A = \bigoplus_{x,y \in X} 1_x A 1_y$ is any locally unital (super)algebra
with distinguished idempotents indexed by some set
$X$,
we can form the completion $\widehat{A}$
consisting of all 
elements $(a_{xy})_{x,y \in X} \in \prod_{x,y \in X} 1_x A 1_y$
such that for each $x$ there are only finitely many $y$ with $a_{xy}
\neq 0$, and for each $y$ there are only finitely many $x$ with
$a_{xy} = 0$.
Clearly the multiplication on $A$ extends to $\widehat{A}$
to make it into a (super)algebra with $1 = \sum_{x \in X} 1_x$.
Applying this construction to $U$, we get the completion
$\widehat{U}$;
applying it to the superalgebra $U \otimes U$, which is locally unital
with distinguished
idempotents $\{1_m \otimes 1_n\:|\:m, n\in \Z\}$,
we get $\widehat{U \otimes U}$; the triple tensor product $U \otimes U
\otimes U$ may be completed similarly.
Now the formulae above extend canonically to define superalgebra
homomorphisms
$\Delta:\widehat{U} \rightarrow \widehat{U \otimes U}$ and
$\eps:\widehat{U} \rightarrow \k$,
satisfying completed versions of the usual coassociativitiy and counit axioms.
This makes $\widehat{U}$ into a Hopf superalgebra in a completed
sense.
(We remark there are several other possible choices of coproduct here; see \cite[$\S$2.4]{Clark}.)

Given $V, W \in \ob \mathcal C$, the tensor product $V
\otimes W$ is naturally a $U \otimes U$-supermodule.
Since it is finite dimensional, it is a $\widehat{U \otimes
  U}$-supermodule too, hence using $\Delta$ we can view it as a
$U$-supermodule.
This makes $\mathcal C$ into a monoidal $\Pi$-supercategory equipped with
a {\em fiber functor} 
$\nu:\mathcal C \rightarrow \SVEC$, namely,
the obvious forgetful superfunctor.
Setting $V := V(1)$, we also have the monoidal superfunctor $G:\mathcal{STL}(\delta)
\rightarrow \SVEC$ from Lemma~\ref{tennis}.

\begin{theorem}
There is a unique monoidal superfunctor $F:\mathcal{STL}(\delta)
\rightarrow \mathcal C$ such that $G = \nu \circ F$:
$$
\begin{tikzcd}
\mathcal{STL}(\delta)
\arrow[rr,"G"]\arrow[rd,swap,"F"]
&&\SVEC\\
&
\C\arrow[ur,"\nu",swap]
\end{tikzcd}.
$$
Moreover, $F$ induces a monoidal equivalence 
$\tilde F:\SKar(\mathcal{STL}(\delta))\rightarrow\underline{\mathcal C}$.
\end{theorem}

\begin{proof}
All of the superspaces $V^{\otimes n}$ are naturally objects of
$\mathcal C$. Moreover, the linear maps defined in Lemma~\ref{tennis}
are $U$-supermodule homomorphisms.
This proves the existence and uniqueness of $F$.

The proof of Theorem~\ref{stlbasis} shows that $F$ is faithful. 
Hence, so is the induced functor $\tilde F:\SKar(\mathcal{STL}(\delta))
\rightarrow \underline{\C}$.
Both $\SKar(\mathcal{STL}(\delta))$ and $\underline{\C}$ are
semisimple Abelian. So, to prove that $\tilde F$ is an equivalence,
we just need to show that the induced $\Z^\pi$-algebra
homomorphism
$K_0(\SKar(\mathcal{STL}(\delta))) \rightarrow K_0(\underline{\C})$
sends the canonical basis coming from the classes of irreducibles
in $\SKar(\mathcal{STL}(\delta))$ to that of $\underline{\C}$.

In view of Theorem~\ref{old}, we may identify
$K_0(\SKar(\mathcal{STL}(\delta)))$ with the 
subring of
$\Z^\pi[x,x^{-1}]$ having canonical basis $\left\{[n+1]_{x,\pi}, \pi[n+1]_{x,\pi}\:|\:n
\in \N\right\}$.
Note this is generated as a $\Z^\pi$-algebra just by $[2]_{x,\pi}$,
which corresponds to the object $1$ in $\mathcal{STL}(\delta)$.
To understand $K_0(\underline{\C})$, consider the
map sending a finite-dimensional $U$-supermodule $M$ to its
{\em supercharacter} $$
\operatorname{SCh} M := \sum_{n \in \Z} 
(\dim (1_n M)_\0 x^n + \dim (1_n M)_\1 \pi x^n) \in \Z^\pi[x,x^{-1}].
$$
We have that $\operatorname{SCh} V(n) = [n+1]_{x,\pi}$.
Hence, $\operatorname{SCh}$ induces 
a $\Z^\pi$-algebra isomorphism between $K_0(\underline{\C})$ and the same based
subring of $\Z^\pi[x,x^{-1}]$ as $K_0(\SKar(\mathcal{STL}(\delta)))$.
Moreover, the generator $[2]_{x,\pi}$ is the supercharacter of $V$.
It remains to observe that $F(1) = V$.
\end{proof}

\begin{corollary} The irreducible $U$-supermodule
 $V(n)$ is isomorphic to the image of the idempotent 
$F(f_n) \in \End_{U}(V^{\otimes n})$, where $f_n$ is
the Jones-Wenzl projector from the proof of Theorem~\ref{old}.
\end{corollary}


\begin{thebibliography}{BHLW}
\bibitem[BK]{BK}
B. Bakalov and A. Kirillov Jr.,
{\em Lectures on Tensor Categories and Modular Functors},
Amer. Math. Soc., 2001.

\bibitem[B]{B}
J. Brundan,
On the definition of Kac-Moody 2-category, {\em Math. Ann.}
{\bf 364} (2016), 353--372.

\bibitem[BCNR]{BCNR}
J. Brundan, J. Comes, D. Nash and A. Reynolds,
A basis theorem for the oriented Brauer
category and its cyclotomic quotients, to appear in {\em Quantum Top.}.

\bibitem[BE]{BE}
J. Brundan and A. Ellis,
Super Kac-Moody 2-categories;
\arxiv{1701.04133}.

\bibitem[BLW]{BLW}
J. Brundan, I. Losev and B. Webster,
Tensor product categorifications and the super Kazhdan-Lusztig
conjecture,
{\em Int. Math. Res. Notices} (2016), article ID rnv388, 81 pages.

\bibitem[CKM]{CKM}
S. Cautis, J. Kamnitzer and S. Morrison,
Webs and quantum skew Howe duality,
{\em Math. Ann.} {\bf 360} (2014), 351--390. 

\bibitem[C]{Clark}
S. Clark,
Quantum supergroups IV: the modified form,
{\em Math. Z.} {\bf 278} (2014), 493--528.

\bibitem[CW]{CW}
S. Clark and W. Wang,
Canonical basis for quantum $\mathfrak{osp}(1|2)$,
 {\em Lett. Math. Phys.} {\bf 103} (2013), 207--231.

\bibitem[EW1]{EW}
B. Elias and G. Williamson,
Soergel calculus, 
{\em Represent. Theory} {\bf 20} (2016), 295--374. 

\bibitem[EW2]{EW2}
B. Elias and G. Williamson,
Diagrammatics for Coxeter groups and their braid groups,
to appear in {\em Quantum Top.}.

\bibitem[EL]{EL}
A. Ellis and A. Lauda,
An odd categorification of $U_q(\mathfrak{sl}_2)$,
{\em Advances Math.} {\bf 265} (2014), 169--240.

\bibitem[EGNO]{EGNO}
P. Etingof,
S. Gelaki, 
D. Nikshych
and V. Ostrik,
{\em Tensor Categories},
Amer. Math. Soc., 2015.

\bibitem[JK]{JK}
J. H. Jung and S.-J. Kang,
Mixed Schur-Weyl-Sergeev  duality for queer Lie superalgebras,
{\em J. Algebra}
{\bf 399}
(2014), 516--545.

\bibitem[KKO]{KKO2}
S.-J. Kang, M. Kashiwara and S.-j. Oh, 
Supercategorification of quantum Kac-Moody algebras II,
{\em Advances Math.} {\bf 265} (2014), 169--240.

\bibitem[K]{kelly}
G. M. Kelly,
{\em Basic Concepts of Enriched Category Theory},
Reprints in Theory and Applications of Categories, No. 10, 2005.

\bibitem[KL]{KL3}
M. Khovanov and A. Lauda,
 A categorification of quantum $\mathfrak{sl}(n)$,
{\em Quantum Top.} {\bf 1} (2010), 1--92.

\bibitem[KT]{KT}
J. Kujawa and B. Tharp,
The marked Brauer category, to appear in {\em J. London Math. Soc.}.

\bibitem[K]{K}
G. Kuperberg,
Spiders for rank 2 Lie algebras, {\em Comm. Math. Phys.} {\bf 180} (1996), 109--151.

\bibitem[L]{Leinster}
T. Leinster,
Basic bicategories;
\arxiv{9810017}.

\bibitem[LZ]{LZ}
G. Lehrer and R.-B. Zhang,
The Brauer category and invariant theory,
{\em J. Eur. Math. Soc.} {\bf 17} (2015), 2311--2351. 

\bibitem[Mac]{Mac}
S. Mac Lane,
{\em Categories for the Working Mathematician},
Springer,
1978.

\bibitem[Man]{Man}
Yu I. Manin,
{\em Gauge Field Theory and Complex Geometry},
Springer, 1997.

\bibitem[MS]{MS}
E. Meir and M. Szymik,
Drinfeld center for bicategories,
{\em Doc. Math.} {\bf 20} (2015), 707--735. 

\bibitem[R]{Rou}
R. Rouquier, 
2-Kac-Moody algebras;
\arxiv{0812.5023}.

\bibitem[S]{S}
W. Soergel,
Kazhdan-Lusztig-Polynome und unzerlegbare Bimoduln \"uber Polynomringen,
{\em J. Inst. Math. Jussieu} {\bf 6} (2007),
501--525.

\bibitem[U]{Usher}
R. Usher,
Fermionic $6j$-symbols in superfusion categories;
\arxiv{1606.03466}.

\bibitem[W]{W}
B. Westbury,
The representation theory of the Temperley-Lieb algebras,
{\em
Math. Z.} {\bf 219} (1995), 539--565. 
\end{thebibliography}
\end{document}